\documentclass[francais,a4paper,11pt]{article}

\usepackage[utf8]{inputenc} 
\usepackage[T1]{fontenc}     
\usepackage[francais]{babel}
\usepackage{textcomp}
\usepackage{microtype}
\usepackage{lmodern}
\usepackage{amsmath,amssymb,amsthm}
\usepackage{mathtools}
\usepackage{DotArrow}
\usepackage[all]{xy}
\usepackage{graphics}
\usepackage{graphicx}
\usepackage{color}
\usepackage{array}
\usepackage{enumerate}
\usepackage{dsfont}
\newcommand{\dd}{\displaystyle}

\def\OO{{\mathcal{O}}}

\def\RR{{\mathbb R}}
\def\NN{{\mathbb N}}

\def\EE{{\mathcal E}}

\def\WW{{\mathcal W}}
\def\TS{{\mathcal T}}

\def\CC{{\mathbb C}}
\def\DD{{\mathfrak D}}
\def\TCC{{\underline{\mathbb C}}}
\def\TRR{{\underline{\mathbb R}}}
\def\CCC{{\mathcal{C}}}
\def\CCCC{{\mathcal{C}}}
\def\ZZ{{\mathbb Z}}

\def\MM{{\mathcal{M}}}

\def\MP{{\mathcal{P}}}
\def\HH{{\mathcal{H}}}

\def\PPP{{\mathfrak p}}
\def\WW{{\mathcal{W}}}

\def\DB{{\bar{\partial}}}
\def\DDJ{{\bar{\partial}_J}}

\def\d{{\mathrm d}}

\def\NF{{\mathcal{N}}}
\def\TT{{\mathcal{T}}}
\def\DDD{{\mathfrak{D}}}
\def\SSS{{\mathfrak{S}}}

\newcommand{\jet}{\mathcal{J}}

\newcommand{\riso}[1]{\RR Aut(#1)}

\newcommand{\raut}[1]{\RR GL(#1)}

\newcommand{\rdiff}[1]{\RR Diff^+(\Sigma_g,w_1(\RR #1))}

\newcommand{\rdif}{\RR Diff^+(\Sigma_g)}

\newcommand{\hczdeux}[2][]{H^1_{#1}(#2,\ZZ/2\ZZ)}

\newcommand{\surg}{\Sigma_g}
\newcommand{\sur}{\Sigma}

\newcommand{\rop}[1][N]{\RR \CCCC (#1)}
\newcommand{\ropj}[1][N]{\RR\CCCC_J(#1)}
\newcommand{\ropd}[1][N]{\RR \CCCC_D (#1)}
\newcommand{\ropc}[1][N]{\RR \CCCC_{compat} (#1)}
\newcommand{\ropcp}[1][N]{\RR \CCCC_{compat}^+ (#1)}
\newcommand{\Det}[1][N]{\ddet(#1)}

\newcommand{\cp}[1][1]{\CC P^{#1}}
\newcommand{\rp}[1][1]{\RR P^{#1}}

\newcommand{\dethun}{\det(H^1(\surg,\RR)_{-1})}
\newcommand{\dethunw}{O_X^{Spin}}
\newcommand{\dethunww}{H^1_w(\RR\surg,\RR)}
\newcommand{\relphi}{\Phi_{\varphi}}
\newcommand{\dethzero}{\RR_{w}}

\newcommand{\ud}{\underline{\mathbf{d}}}
\newcommand{\ux}{\underline{\mathbf{x}}}
\newcommand{\uz}{\underline{\mathbf{z}}}
\newcommand{\PV}[1][N]{\PPP^+(\RR #1)}

\DeclareMathOperator{\ddet}{Det}
\DeclareMathOperator{\coker}{coker}

\DeclareMathOperator{\card}{card}
\DeclareMathOperator{\id}{id}
\DeclareMathOperator{\rang}{rg}

\DeclareMathOperator{\pd}{PD}

\DeclareMathOperator{\e}{e}

\newtheorem{Theoreme}{Théorème}[section]
\newtheorem*{Theo}{Théorème}
\newtheorem{Corollaire}{Corollaire}[section]
\newtheorem{Proposition}{Proposition}[section]
\newtheorem{Lemme}{Lemme}[section]
\newtheorem{Definition}{Définition}[section]
\newtheorem*{Def}{Définition}
\newcounter{Notcount}
\newenvironment{Notation}{\refstepcounter{Notcount}\noindent\ignorespaces {\bf Notation \arabic{Notcount}}. --- }{\bigskip}

\newcounter{Remcount}[section]
\renewcommand{\theRemcount}{\thesection.\arabic{Remcount}}
\newenvironment{Rem}{\refstepcounter{Remcount}\noindent\ignorespaces {\bf Remarque \theRemcount}. --- }{\bigskip}
\newcounter{Explcount}[section]
\renewcommand{\theExplcount}{\thesection.\arabic{Explcount}}

\title{Déterminant des opérateurs de Cauchy-Riemann réels et application à l'orientabilité d'espaces de modules de courbes réelles}
\author{Rémi CR\'ETOIS}

\begin{document}

\maketitle

\centerline{\textbf{Abstract}}
The paper is motivated by the study of the orientability of the moduli spaces of real pseudoholomorphic curves in real symplectic manifolds. We begin by extending the results we obtained in \cite{article1}. Namely, we consider a complex vector bundle $(N,c_N)$ equipped with a real structure over a real curve $(\surg,c_\sur)$ of genus $g \in \NN$, and we give a description of the action of an automorphism of $(N,c_N)$ on the orientations of the determinant bundle over the space of all Cauchy-Riemann operators on $(N,c_N)$.

  In the second part, we use the previous decription to compute the first Stiefel-Whitney class of the moduli space of real pseudoholomorphic curves in many cases. We show for example that the real pseudoholomorphic curves with non-empty real part in the complex projective space of dimension three form an orientable moduli space for a generic almost complex structure.

An English summary is available online (\cite{summary}).
\\

\noindent
\textsc{Classification AMS 2010}: 14H60, 53D45\\
\textsc{Mots clés}: fibrés vectoriels, courbes réelles, opérateurs de Cauchy-Riemann, espaces de modules, invariants de Gromov-Witten

\tableofcontents

\section*{Introduction}

\subsection*{Motivation}

Une variété symplectique $(X,\omega)$ est dite réelle si elle est munie d'une involution anti-symplectique $c_X : X\rightarrow X$. Dans ce cas, on note $\RR J_\omega(X)$ l'ensemble des structures presque-complexes sur $X$ qui sont calibrées par $\omega$ et pour lesquelles $c_X$ est anti-holomorphe. \'Etant donnés un entier $g\geq 0$ et une classe $d\in H_2(X,\ZZ)$, un des problèmes ouverts de géométrie énumérative réelle consiste à déterminer à quelles conditions sur $(X,\omega,c_X)$ il est possible de définir pour tout $J\in \RR J_\omega(X)$ assez générique un compte de courbes $J$-holomorphes réelles de genre $g$ dans la classe $d$, soumises à des conditions d'incidence, de sorte que l'entier obtenu fournisse un invariant de la classe de déformation de $(X,\omega,c_X)$. En suivant l'approche de Gromov et Witten, Welschinger \cite{wel1}, \cite{wel3}, est parvenu à définir de tels invariants en genre $0$ pour les variétés symplectiques réelles de dimension $4$ et pour celles de dimension $6$ dont la forme symplectique vérifie une hypothèse de semipositivité forte. Ces travaux ont mis en lumière une différence fondamentale entre le cas réel et le cas complexe. En effet, des invariants de type Gromov-Witten sont définis en considérant le degré de l'application évaluation définie sur l'espace de modules des courbes rationnelles $J$-holomorphes dans $X$ munies de points marqués. Or, contrairement au cas complexe, l'espace de modules de courbes réelles rationnelles $J$-holomorphes dans $X$ n'est pas forcément orientable, rendant alors la définition du degré de l'évaluation sujette à caution. C'est ce problème d'orientabilité des espaces de modules de courbes pseudo-holomorphes réelles de tout genre qui motive \cite{article1} ainsi que le présent article.

 Dans leurs travaux sur l'homologie de Floer, Fukaya, Oh, Ohta et Ono (\cite{fooo}) ont introduit la notion de sous-variété relativement $Spin$ et montré que lorsque $\RR X := Fix(c_X)$ est relativement $Spin$, l'espace des disques pseudo-holomorphes à bord dans $\RR X$ et sans point marqué est orientable. D'autre part, la définition par Welschinger de ses invariants lui a notamment permis de donner un représentant géométrique du dual de cette première classe de Stiefel-Whitney dans le cas des variétés projectives réelles lisses et convexes de dimension deux et trois. Ce résultat a ensuite été précisé par Puignau (\cite{puignau}). De son côté, Cho (\cite{cho}) a repris la définition des invariants de Welschinger en utilisant les techniques de \cite{fooo}, de même que Solomon (\cite{Solomon}) qui les a étendus à des surfaces à bord de genre plus élevé et aux variétés symplectiques de dimension six quelconques. Signalons aussi l'article de Georgieva et Zinger \cite{georgieva} qui améliorent les techniques de Solomon pour les appliquer au cas des espaces de modules des courbes réelles en général.

Notre angle d'attaque de ce problème est différent des auteurs cités précédemment. En effet, nous commençons par étudier l'action, décrite plus bas, des automorphismes réels d'un fibré vectoriel complexe muni d'une structure réelle au-dessus d'une courbe réelle sur les orientations du fibré déterminant, puis nous utilisons les résultats obtenus afin d'en tirer des conclusions sur l'orientabilité des espaces de modules de courbes réelles pseudo-holomorphes dans des variétés symplectiques réelles.

\subsection*{Automorphismes réels d'un fibré et fibré déterminant}

Nous considérons le problème abstrait suivant, que nous avons commencé à aborder dans \cite{article1}. Soit $(\surg,c_\sur)$ une courbe réelle abstraite, c'est-à-dire un couple formé d'une surface compacte connexe sans bord de genre $g$ et orientée et d'un difféomorphisme involutif $c_\sur : \surg\rightarrow \surg$ qui renverse l'orientation. Nous noterons $\RR\surg$ l'ensemble des points fixes de $c_\sur$. C'est une sous-variété lisse de dimension $1$ de $\surg$, c'est-à-dire une union disjointe de cercles, éventuellement vide. Nous considérons un fibré vectoriel complexe $\pi : N\rightarrow \surg$ que nous munissons d'une structure réelle, c'est-à-dire d'un automorphisme involutif $c_N : N\rightarrow N$, $\CC$-antilinéaire dans les fibres et qui relève $c_\sur$. Nous noterons $\RR N$ l'ensemble des points fixes de $c_N$, qui est un fibré vectoriel réel au-dessus de $\RR\surg$. De plus, lorsque $c_\sur$ ou $c_N$ induisent une involution sur un espace vectoriel (ou sur un module), nous notons avec un indice $+1$ (resp. $-1$) le sous-espace propre associé à la valeur propre $+1$ (resp. $-1$) de cette involution. 

\begin{Def}
Un automorphisme du fibré $(N,c_N)$ est une paire $(\Phi,\varphi)$ composée
\begin{enumerate}
\item d'un difféomorphisme $\varphi : (\surg,c_\sur)\rightarrow (\surg,c_\sur)$ préservant les orientations et $\ZZ/2\ZZ$-équivariant,
\item d'une application $\Phi : (N,c_N) \rightarrow (N,c_N)$ qui est $\ZZ/2\ZZ$-équivariante, $\CC$-linéaire dans les fibres, et telle que $\pi\circ\Phi = \varphi\circ\pi$.
\end{enumerate}
\end{Def}

Nous noterons $\riso{N}$ l'ensemble de tous les automorphismes de $(N,c_N)$. Remarquons que lorsque $(\Phi,\varphi)\in\riso{N}$, le difféomorphisme $\varphi$ doit vérifier $\varphi^*w_1(\RR N) = w_1(\RR N)$. Nous posons donc
\[
\rdiff{N} = \{\varphi\in Diff^+(\surg)\ |\ \varphi\circ c_\sur = c_\sur \circ\varphi,\ \varphi^*w_1(\RR N) = w_1(\RR N)\}.
\]

Nous avons alors un morphisme naturel $\riso{N}\rightarrow \rdiff{N}$ qui est surjectif et dont le noyau est le groupe $\raut{N}$ des automorphismes de $(N,c_N)$ qui relèvent l'identité.

Notons maintenant $\RR J(\surg)$ l'ensemble des structures complexes de classe $\CCC^l$, $1\ll l$, sur $\surg$ qui sont compatibles avec l'orientation fixée et qui rendent $c_\sur$ anti-holomorphe. C'est une variété de Banach non vide et contractile (voir \cite{wel1}). Fixons un entier $1\ll k \ll l$ et un réel $p > 1$ tels que $kp > 2$.

\begin{Def}
  Un opérateur de Cauchy-Riemann réel sur $(N,c_N)$ est une paire $(\DB, J)$ formée d'une structure complexe $J \in \RR J(\surg)$ et d'un opérateur $\CC$-linéaire 
\[
\DB : L^{k,p}(\surg, N)\rightarrow L^{k-1,p}(\surg, \Lambda^{0,1}\surg\otimes N)
\]
équivariant sous l'action de $c_N$ et vérifiant la règle de Leibniz :
\[
\forall f \in \CCC^{\infty}(\surg,\CC),\ \forall v\in L^{k,p}(\surg,N),\ \DB(fv) = f\DB(v) + \DB_J(f)\otimes v,
\]
où $\DB_J = \frac{1}{2}(\d + i\circ \d\circ J)$.
\end{Def}

Un opérateur de Cauchy-Riemann réel $\DB$ sur $(N,c_N)$ induit par restriction un opérateur Fredholm de $L^{k,p}(\surg,N)_{+1}$ dans $L^{k-1,p}(\surg,N)_{+1}$ dont nous notons respectivement $H^0_{\DB}(\surg,N)_{+1}$ et $H_{\DB}^1(\surg,N)_{+1}$ le noyau et le conoyau.

Notons $\rop$ l'ensemble des opérateurs de Cauchy-Riemann réels sur $(N,c_N)$. Nous avons une application $(\DB,J)\in\rop \mapsto J\in\RR J(\surg)$ entre variété de Banach qui est un fibré dont la fibre $\ropj$ au-dessus de $J$ est un espace affine de dimension infinie. En particulier, $\rop$ est contractile. Il existe sur $\rop$ un fibré en droites réelles $\Det$ dont la fibre au-dessus de l'opérateur $\DB$ est son déterminant $\ddet(\DB) = \Lambda^{\max}_{\RR}H^0_{\DB}(\surg,N)_{+1}\otimes \left(\Lambda^{\max}_{\RR}H^1_{\DB}(\surg,N)_{+1}\right)^*$ (voir \cite{MDS}, Théorème A.2.2). Ce fibré est orientable, mais n'est pas canoniquement orienté.

Ce défaut d'orientation canonique est à l'origine des problèmes d'orientabilité des espaces de modules de courbes réelles dans des variétés symplectiques réelles, comme nous le verrons au \S \ref{module}. Nous étudions donc dans un premier temps l'action du groupe $\riso{N}$ sur les orientations de $\Det$. Cette action est définie de la façon suivante. Tout d'abord $\riso{N}$ agit sur les sections de $N$ et sur les $(0,1)$-formes à valeurs dans $N$ d'une part et sur les éléments de $\RR J(\surg)$ d'autre part par
\[
(\Phi,\varphi)_* s = \Phi(s)\circ\varphi^{-1},
(\Phi,\varphi)^*\alpha = \Phi^{-1}\circ\alpha\circ\d\varphi,
(\Phi,\varphi)^*J = \d\varphi^{-1}\circ J\circ\d \varphi,
\]
où $s$ est une section de $N$, $\alpha$ est une $(0,1)$-forme à valeurs dans $N$ et $J\in\RR J(\surg)$. Ceci induit une action sur $\rop$
\[
(\Phi,\varphi)^*(\DB,J) = ((\Phi,\varphi)^*(\DB((\Phi,\varphi)_*)),(\Phi,\varphi)^*J),
\]
et donc sur le fibré $\Det$.

Dans \cite{article1}, nous avons étudié l'action du groupe $\raut{N}$ des automorphismes de $(N,c_N)$ relevant l'identité sur les orientations de $\Det$. Nous avons notamment décomposé la contribution d'un élément de $f\in\raut{N}$ à cette action en deux termes : l'un comptant le nombre de composantes de $\RR N$ sur lesquelles $f$ échange les structures $Pin^{\pm}$; l'autre considérant l'action du déterminant de $f$, $\det(f)\in\raut{\det(N)}$, sur les orientations de $\Det[\det(N)]$, où $\det(N) = \Lambda^{\rang(N)}_{\CC}N$ est un fibré en droites complexes sur $\surg$. Plus précisément, nous avons montré que lorsque $N$ est de rang $1$, on peut définir une action de $\raut{N}$ sur les classes de bordisme de structures $Spin$ réelles de $(\surg,c_\sur)$ grâce à laquelle on exprime l'action de $\raut{N}$ sur les orientations de $\Det$ (voir les Théorèmes A et B de \cite{article1}).

\subsection*{Aperçu de l'article}

Nous reprenons et terminons l'étude de l'action du groupe $\riso{N}$ sur les orientations de $\Det$ là où nous nous étions arrêtés dans \cite{article1}. Nous considérons donc dans le \S 1 l'action des automorphismes de $(N,c_N)$ relevant des difféomorphismes non triviaux de $(\surg,c_\sur)$.

 Nous commençons par décomposer l'action d'un automorphisme $(\Phi,\varphi)\in\riso{N}$ sur les orientations de $\Det$ en un produit de trois termes (voir la Proposition \ref{lemrangun}). Nous obtenons le premier terme en faisant agir $\Phi$ sur le produit tensoriel $\PPP^+(\RR N) = \dd\bigotimes_{i = 1}^k Pin^+((\RR N)_i)$, où $Pin^+((\RR N)_i)$ est l'ensemble formé des deux structures $Pin^+$ sur la composante $(\RR N)_i$ de $\RR N$. Le second terme est donné par l'action d'un automorphisme particulier du fibré trivial $(\TCC^{\oplus \rang(N)-1},conj)$ sur les orientations de $\Det[\TCC^{\oplus \rang(N)-1}]$, que nous calculons au Lemme \ref{isopart}. Le dernier terme est égal à l'action de l'automorphisme $\det(\Phi)\in\riso{\det(N)}$ sur les orientations de $\Det[\det(N)]$. Le problème se réduit donc à l'étude du cas où $N$ est un fibré en droites complexes.

Nous traitons alors les cas particuliers des fibrés trivial $(\TCC,conj)$, canonique $(K_\sur,(\d c_\sur)^*)$ et tangent $(T\surg,\d c_\sur)$ (au \S \ref{extrivial} et \ref{extang}). Ceci nous permet d'étudier l'action du groupe $\RR Diff^+(\surg) = \{\varphi\in Diff^+(\surg)\ |\ c_\sur\circ\varphi = \varphi\circ c_\sur \}$ des difféomorphismes réels de $(\surg,c_\sur)$ sur les orientations de l'espace de Teichmüller réel (au \S \ref{parteich}). Ce dernier est le quotient de $\RR J(\surg)$ par l'action de la composante connexe de l'identité du groupe $\RR Diff^+(\surg)$. En particulier, nous obtenons le résultat suivant (voir Théorème \ref{actionteich}).

\begin{Theo}
Soit $(\surg,c_\sur)$ une courbe réelle de genre au moins deux. L'action d'un difféomorphisme réel $\varphi$ de $(\surg,c_\sur)$ sur les orientations de l'espace de Teichmüller réel associé à $(\surg,c_\sur)$ est donnée par le signe du déterminant de l'application $\varphi_*: H_1(\surg,\RR)_{+1}\rightarrow H_1(\surg,\RR)_{+1}$.
\end{Theo}

Nous traitons ensuite le cas d'un fibré en droites complexes $(N,c_N)$ quelconque. L'action des éléments de $\raut{N}$ sur les classes de bordisme de structures $Spin$ réelles de $(\surg,c_\sur)$ ne se prolonge pas en une action de tout $\riso{N}$; nous n'avons donc pas de généralisation directe des Théorèmes A et B de \cite{article1}. Pour remplacer cette action, nous introduisons la notion de diviseur réel associé à $(\det(N),c_{\det(N)})$ (voir \S \ref{diviseur}); ceux-ci sont les diviseurs invariants par $c_\sur$ tels que leur dual de Poincaré (respectivement, le dual de Poincaré de leur restriction à $\RR\surg$) vaut $c_1(N)$ (respectivement $w_1(\RR N)$). \'Etant donné un tel diviseur $D$, nous définissons un fibré $\ZZ/2\ZZ$-principal $\DD_D(N)$ au-dessus de $\RR J(\surg)$ (voir Définition \ref{defd}); de plus, l'action sur $\RR J(\surg)$ des difféomorphismes $\varphi\in\RR Diff^+(\surg)$ tels que $\varphi_* D = D$ se relève en une action des automorphismes $(\Phi,\varphi)\in\riso{N}$ sur $\DD_D(N)$. Nous utilisons ce fibré pour décomposer l'action d'un élément de $\riso{N}$ sur les orientations de $\Det$ en deux termes : l'un est donné par l'action sur $\DD_D(N)$ et généralise l'action des éléments de $\raut{N}$ sur les classes de bordisme de structures $Spin$ réelles de $(\surg,c_\sur)$, l'autre ne fait intervenir que le difféomorphisme sous-jacent. En particulier, un élément de $\raut{N}$ préserve les orientations de $\DD_D(N)$ si et seulement si il agit trivialement sur les orientations de $\Det$. 

Notons de plus que si tous les points de $D$ apparaissent avec multiplicité $\pm 1$, alors le choix d'une section réelle de $(\det(N),c_{\det(N)})$ qui s'annule transversalement aux points de $D$ avec un indice d'annulation donné par la multiplicité du point dans $D$ oriente le fibré $\DD_D(N)$. Deux telles sections sont, à homotopie près, image l'une de l'autre par un élément de $\raut{N}$, et un élément de $\raut{N}$ fixant une de ces sections à homotopie près est isotope à l'identité. Dans ce cas-ci, le choix d'une telle section nous permet de définir, à homotopie près, un relevé $(\Phi,\varphi)\in\riso{N}$ d'un difféomorphisme $\varphi\in\RR Diff^+ (\surg)$ tel que $\varphi_* D = D$; et si deux sections donnent la même orientation de $\DD_D(N)$, alors les deux relevés correspondant de $\varphi$ ont la même action sur les orientations de $\Det$, donnée par un terme ne faisant intervenir que $\varphi$. Pour le cas général, nous renvoyons au \S \ref{diviseur} pour la définition précise de $\DD_D(N)$.

La contribution du difféomorphisme sous-jacent $\varphi$ à l'action de $(\Phi,\varphi)\in\riso{N}$ sur les orientations de $\Det$, est définie en termes purement topologiques. Nous introduisons pour cela un espace vectoriel réel $\RR \jet_D$ qui est le sous-espace invariant par $c_\sur$ de la somme directe des puissances tensorielles successives des espaces tangents à $\surg$ aux points de $D$. En particulier, cet espace est orienté par le choix d'une numérotation des points de $D$ et d'une orientation de $\RR\surg$, et lorsque tous les points de $D$ apparaissent avec une multiplicité $\pm 1$, seule une numérotation suffit. Nous renvoyons au \S \ref{diviseur} pour la définition précise.

  Nous obtenons alors le résultat suivant (voir le Théorème \ref{totalaction}).

\begin{Theo}
Soit $(N,c_N)$ un fibré vectoriel complexe sur $(\surg,c_\sur)$. Soit $D$ un diviseur associé à $(\det(N),c_{\det(N)})$ et $(\Phi,\varphi)\in\riso{N}$ tel que $\varphi_*D = D$. L'action de $\Phi$ sur les orientations de $\Det$ est le produit de l'action de $\Phi$ sur $\PV$, de l'action de $\Phi$ sur les orientations de $\DD_D(N)$ et de l'action de $\varphi$ sur les orientations de
\[
\det(H^1(\surg,\RR)_{-1})^{\otimes\rang(N)}\otimes\det(\RR\jet_{D})\dd\bigotimes_{x\in D_{|\RR\surg}} (T_x^*\RR\surg)^{\otimes\max(0,mult_D(x))}.
\]
\end{Theo}

 Ce Théorème nous permet ainsi de décrire l'action d'un automorphisme de $(N,c_N)$ sur les orientations du fibré $\Det$ de façon plus topologique. Le Théorème \ref{totalaction} est en fait plus général, car nous obtenons dans un premier temps un isomorphisme canonique entre $\Det$ et
\[
\PV\otimes \DD_D(N)\otimes\det(H^1(\surg,\RR)_{-1})^{\otimes\rang(N)}\otimes\det(\RR\jet_{D})\dd\bigotimes_{x\in D_{|\RR\surg}} (T_x^*\RR\surg)^{\otimes\max(0,mult_D(x))}.
\]
et dans un deuxième temps, nous supprimons l'hypothèse $\varphi_* D = D$.

Nous améliorons enfin le résultat précédent dans deux cas particuliers : lorsque la courbe est séparante, c'est-à-dire lorsque $\surg\setminus\RR\surg$ n'est pas connexe, et lorsque la partie réelle de $N$ est non vide et orientable (voir Corollaires \ref{actionsepa} et \ref{actionspin}).

Nous donnons ensuite au \S \ref{module} quelques applications des résultats concernant l'action décrite précédemment au problème d'orientabilité des espaces de modules de courbes pseudo-holomorphes réelles dans des variétés symplectiques réelles. Plus précisément, nous considérons une variété symplectique réelle $(X,\omega,c_X)$, un entier $g$ et une classe $d\in H_2(X,\ZZ)$, et nous notons $\pi : \RR\MM_g^d(X)\rightarrow \RR J_{\omega}(X)$ l'espace de modules universel des courbes pseudo-holomorphes réelles de genre $g$ dans la classe $d$ qui sont injectives quelque part. Nous nous restreignons aux courbes injectives quelque part afin d'éviter tout problème analytique causé par les courbes qui sont des revêtements multiples. En effet, sous cette hypothèse, la fibre $\RR \MM_g^d(X,J) = \pi^{-1}\{J\}$ est une variété lisse de dimension finie pour tout $J \in \RR J_{\omega}(X)$ générique. Ceci nous permet de ne nous préoccuper que du problème d'orientabilité des espaces de modules. 

 Nous déduisons du \S \ref{autom} la première classe de Stiefel-Whitney des espaces de modules de courbes réelles dans $(X,\omega,c_X)$ auxquelles on ajoute une polarisation (voir le Théorème \ref{metath}). Afin d'obtenir des résultats sur l'espace $\RR \MM_g^d(X)$, nous introduisons une polarisation $D$ sur $(X,\omega,c_X)$, sorte d'analogue symplectique d'un diviseur anti-canonique, et nous notons $\RR J_D(X)\subset \RR J_{\omega}(X)$ l'ensemble des structures presque-complexes qui rendent $D$ holomorphe (voir \S \ref{defpol}). L'espace $\pi : \RR\MM_g^d(X,D)^{\pitchfork}\rightarrow \RR J_D(X)$ constitué des courbes réelles $J$-holomorphes pour des structures $J\in \RR J_D(X)$, de genre $g$, dans la classe $d$ et qui sont transverses à $D$, forme une sous-variété de Banach de $\RR\MM_g^d(X)$. Nous définissons de plus deux fibrés $\ZZ/2\ZZ$-principaux au-dessus de $\RR \MM_g^d(X,D)^{\pitchfork}$ : $T_{pol}$ dont la fibre au-dessus d'une courbe est formée des deux numérotations à permutation alternée près possible des points d'intersections de la courbe avec $D$ tensorisé avec le produit tensoriel des espaces cotangents à la partie réelle de la courbe aux points d'intersection réels positifs; $\PPP^+_X$ qui est orientable dès que $\RR X$ admet une structure $Pin^+$ (voir le Lemme \ref{fibr}). Alors, si $D$ est donné comme l'ensemble des zéros d'une section réelle de $(T X,\d c_X)$, que l'on appelle polarisante, nous obtenons (voir le Théorème \ref{caspol}) :

\begin{Theo}
Soit $(X^{2n},\omega,c_X)$ une variété symplectique réelle polarisée par le diviseur $D$ donné par une section polarisante. Pour $J\in\RR J_\omega(X)$ générique rendant $D$ holomorphe, le fibré en droites réelles $\det(T \RR\MM_g^d(X,D,J)^{\pitchfork})$ est canoniquement isomorphe à
\[\label{fi}
\PPP^+_X\otimes \dethun^{\otimes n-1}\otimes T_{pol}\tag{$*$}
\]
au-dessus de $\RR\MM_g^d(X,D,J)^{\pitchfork}$.
\end{Theo}

Le Théorème \ref{caspol} est en fait plus général. En effet, il établi un isomorphisme canonique entre (\ref{fi}) et le déterminant de l'application Fredholm $\pi : \RR \MM_g^d(X,D)^{\pitchfork}\rightarrow \RR J_D(X)$, qui est en quelque sorte le fibré des orientations relatives de $\RR \MM_g^d(X,D)^{\pitchfork}$.

Nous donnons enfin des exemples de polarisations et de sections polarisantes dans le \S \ref{existpol} notamment dans le cas des hypersurfaces de $\cp[n]$ et dans le cas général en utilisant des hypersurfaces de Donaldson.

Comme dans le \S \ref{autom}, nous faisons apparaître deux cas particuliers dans lesquels nous avons de meilleurs résultats pour le problème considéré (voir \S \ref{deuxcas}). En particulier, nous obtenons le résultat suivant (voir le Corollaire \ref{hypercorollaire}) :

\begin{Theo}
 Soit $d\in H_2(\cp[3],\ZZ)$, $g\in\NN$ et $J \in\RR J_{\omega_{Fub}}(\cp[3])$ générique. Les composantes connexes de $\RR \MM_g^d(\cp[3],J)$ dont les éléments intersectent $\rp[3]$ sont orientables.
\end{Theo}
\bigskip

Ces résultats concernant les espaces de modules sont un premier pas dans la généralisation des invariants de Welschinger en genre supérieur et en dimension supérieure à $6$. Comme nous l'avons mentionné au début de cette introduction, d'autres problèmes d'ordre analytique apparaissent. Nous avons à l'avenir l'intention de les étudier dans des cas particuliers afin de pouvoir appliquer les résultats du présent article à des problèmes énumératifs concrets.

{\bf Notations}

Dans tout l'article, $(\surg,c_\sur)$ sera une courbe réelle abstraite et $(N,c_N)$ un fibré vectoriel complexe muni d'une structure réelle au-dessus de $(\surg, c_\sur)$.

\begin{itemize}
\item $\RR J(\surg)$ : l'ensemble des structures complexes sur $\surg$ qui rendent $c_\sur$ antiholomorphe.
\item $\rop$ : l'ensemble des opérateurs de Cauchy-Riemann réels sur $(N,c_N)$ (voir \cite{article1}). Pour tout $J\in \RR J(\surg)$, $\ropj$ est le sous-ensemble de $\rop$ formé des opérateurs pour lesquels la structure complexe sous-jacente sur $\surg$ est donnée par $J$.
\item $\riso{N}$ : le groupe des automorphismes de $N$ qui commutent avec $c_N$.
\item $\raut{N}$ : le sous-groupe de $\riso{N}$ formé des automorphismes relevant l'identité de $\surg$.
\item $\rdiff$ : le groupe des difféomorphismes de $\surg$ qui préservent l'orientation fixée, commutent avec $c_\sur$ et fixent $w_1(\RR N)$.
\end{itemize}

\section{\'Etude des automorphismes réels}\label{autom}

\subsection{Action des automorphismes sur les structures $Pin^\pm$}\label{parpin}

Soit $(N,c_N)$ un fibré vectoriel complexe muni d'une structure réelle sur une courbe réelle $(\surg,c_\sur)$. Un élément $\Phi$ de $\riso{N}$ induit naturellement une permutation $\Phi_{\PPP^+}$ (resp. $\Phi_{\PPP^-}$) de l'ensemble $Pin^+(\RR N)$ (resp. $Pin^-(\RR N)$) à $2k$ éléments formé des structures $Pin^+$ (resp. $Pin^-$) de $\RR N$  en tirant en arrière celles-ci (voir \cite{kirby}). Ceci va nous permettre de nous ramener au cas du rang $1$ (voir la Proposition \ref{lemrangun} et aussi le Lemme 4.1 de \cite{article1}). Avant cela, nous remarquons dans la Proposition \ref{pinori} qu'en général l'action d'un élément de $\riso{N}$ n'est pas la même sur les structures $Pin^+$ et sur les structures $Pin^-$ de $\RR N$. Plus précisément, considérons un difféomorphisme $\varphi\in\rdiff{N}$. Il induit une permutation $\sigma_\varphi$ sur l'ensemble à $2k$ éléments formé des orientations des composantes connexes de $\RR \surg$. Cette permutation laisse stable le sous-ensemble à $2k_-$ éléments formé des orientations de composantes connexes de $\RR\surg$ sur lesquelles $\RR N$ n'est pas orientable. Notons $\sigma^-_\varphi$ la permutation induite sur cet ensemble. Enfin, pour toute permutation $\sigma$ d'un ensemble fini, nous notons $\varepsilon(\sigma)$ la signature de $\sigma$.

\begin{Proposition}\label{pinori}
  Soit $(N,c_N)$ un fibré vectoriel complexe muni d'une structure réelle sur $(\surg,c_\sur)$, de partie réelle non vide, et soit $(\Phi,\varphi)\in\riso{N}$. On a $\varepsilon(\Phi_{\PPP^+}) = \varepsilon(\Phi_{\PPP^-})\varepsilon(\sigma^-_\varphi)$, où $\Phi_{\PPP^+}$ (resp. $\Phi_{\PPP^-}$) est la permutation induite par $\Phi$ sur $Pin^+(\RR N)$ (resp. sur $Pin^-(\RR N)$), et $\sigma^-_\varphi$ est la permutation induite par $\varphi$ sur l'ensemble des orientations des composantes connexes de $\RR\surg$ sur lesquelles $\RR N$ n'est pas orientable.
\end{Proposition}

Commençons par donner une autre définition des différentes signatures intervenant dans l'énoncé. Le difféomorphisme $\varphi$ induit une permutation des composantes connexes de $\RR\surg$ dont on note $c_1\ldots c_l\ldots c_m$ la décomposition en cycles disjoints. On suppose de plus que le support des cycles $c_1,\ldots,c_l$ est formé des composantes de $\RR\surg$ sur lesquelles $\RR N$ n'est pas orientable. On note $l(c_i)$ la longueur du cycle $c_i$.

 On associe alors à chaque cycle $c_i$ 
 \begin{itemize}
 \item deux éléments $s_{\PPP^\pm}(c_i)\in\{-1,1\}$: soit $(\RR \surg)_j$ dans le support de $c_i$, alors $s_{\PPP^\pm}(c_i)$ vaut $1$ si $\Phi^{l(c_i)}$ (qui envoie la composante $(\RR N)_j$ sur elle-même) préserve les structures $Pin^\pm$ de $(\RR N)_j$ et vaut $-1$ sinon.
\item un élément $s_\varphi(c_i)\in\{-1,1\}$: soit $(\RR \surg)_j$ dans le support de $c_i$, alors $s_\varphi(c_i)$ vaut $1$ si $\varphi^{l(c_i)}$ (qui envoie la composante $(\RR \surg)_j$ sur elle-même) préserve les orientations de $(\RR\surg)_j$ et vaut $-1$ sinon.
 \end{itemize}

\begin{Lemme}\label{pinpmsign}
  Les signes $s_{\PPP^\pm}(c_i)$ et $s_\varphi(c_i)$ ne dépendent pas du choix de la composante $(\RR\surg)_j$ choisie. De plus, on a $\varepsilon(\Phi_{\PPP^\pm}) = \dd\prod_{i=1}^m s_{\PPP^\pm}(c_i)$, et $\varepsilon(\sigma^-_\varphi) = \dd\prod_{i=1}^l s_\varphi(c_i)$.
\end{Lemme}

\begin{proof}
Traitons le cas de $\sigma^-_\varphi$, l'autre étant tout à fait analogue. Fixons un cycle $c_i$, $1\leq i\leq l$. Son support est formé de $l(c_i)$ composantes connexes de $\RR \surg$. Si $c'_i$ est un cycle de la décomposition en cycles disjoints de $\sigma^-_\varphi$ dont le support contient une orientation d'une des composantes du support de $c_i$, alors $l(c'_i)$ vaut soit $l(c_i)$ soit $2l(c_i)$. Ainsi deux cas se présentent (voir Figure \ref{permut}):
\begin{itemize}
\item soit $s_\varphi(c_i) = 1$, autrement dit $(c'_i)^{l(c_i)}$ préserve les orientations de $(\RR \surg)_j$, et $l(c'_i) = l(c_i)$,
\item soit $s_\varphi(c_i) = -1$, autrement dit $(c'_i)^{l(c_i)}$ échange les orientations de $(\RR \surg)_j$, et $l(c'_i) = 2l(c_i)$.
\end{itemize}

  \begin{figure}[h]
    \centering
    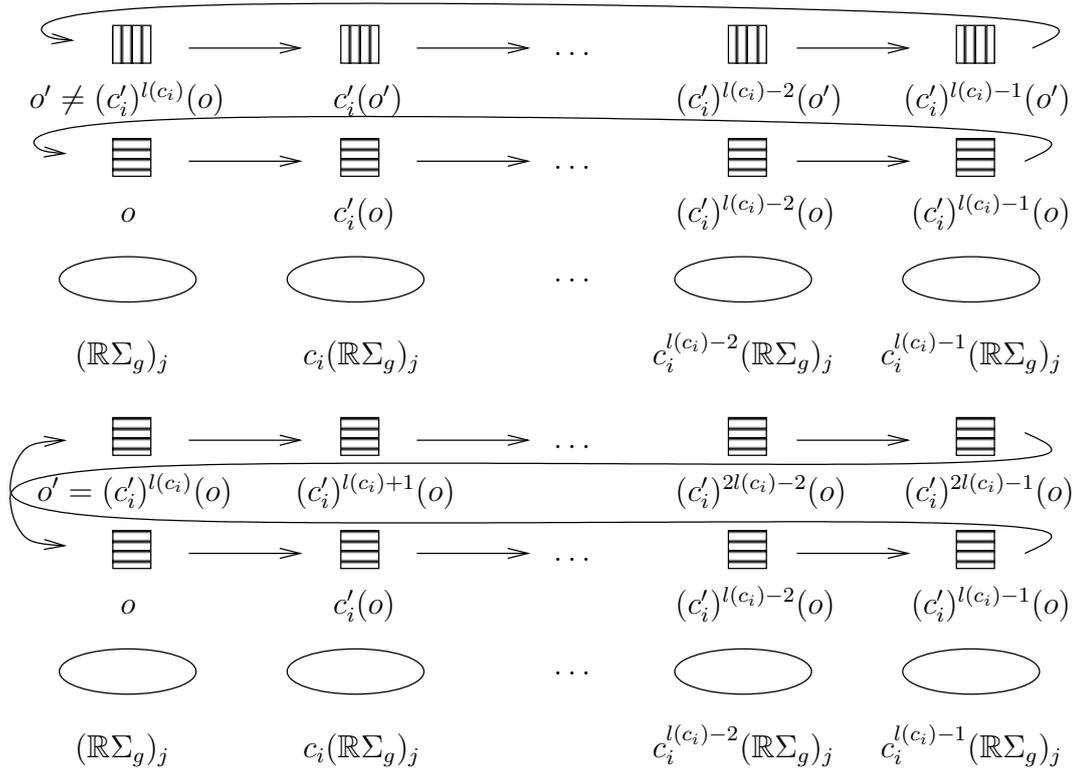
    \caption{Le cycle $c'_i$ lorsque $s_\varphi(c_i) = 1$ en haut, et lorsque $s_\varphi(c_i)=-1$ en bas.}
    \label{permut}
  \end{figure}

Dans le premier cas, il existe un autre cycle $c''_i$ dont le support est formé des orientations opposées à celles de $c'_i$. Dans le deuxième cas, la support de $c'_i$ contient toutes les orientations des composantes formant le support de $c_i$.

La décomposition en cycles de supports disjoints de $\sigma^-_\varphi$ est donc de la forme 
\[
\dd\prod_{s_\varphi(c_i)=1} (c'_ic''_i)\prod_{s_\varphi(c_i)=-1}c'_i.
\]
Dans le premier produit, comme les cycles $c'_i$ et $c''_i$ sont de même longueur, leurs contributions au calcul la signature de $\sigma^-_\varphi$ se compensent. Ainsi
\[
\varepsilon(\sigma^-_\varphi) = \prod_{s_\varphi(c_i)=-1}(-1)^{2l(c_i)-1} = \dd\prod_{i=1}^l s_\varphi(c_i).
\]
\end{proof}

\begin{proof}[Démonstration de la Proposition \ref{pinori}]
  Grâce au Lemme \ref{pinpmsign} on peut réécrire l'égalité
\[
\varepsilon(\Phi_{\PPP^+}) = \varepsilon(\Phi_{\PPP^-})\varepsilon(\sigma^-_\varphi)
\]
sous la forme
\[
\dd\prod_{i=1}^m s_{\PPP^+}(c_i) = \dd\prod_{i=1}^m s_{\PPP^-}(c_i)\dd\prod_{i=1}^l s_\varphi(c_i).
\]
Ainsi, il suffit de démontrer que pour chaque cycle $c_i$, l'action de $\Phi^{l(c_i)}$ sur les structures $Pin^+$ de $(\RR N)_j$, où $(\RR\surg)_j$ est dans le support de $c_i$, coïncide avec celle sur les structures $Pin^-$ de $(\RR N)_j$ si et seulement si $(\RR N)_j$ est orientable ou bien $(\RR N)_j$ est non orientable et $\varphi^{l(c_i)}$ préserve les orientations de $(\RR\surg)_j$. 

Remarquons d'abord que si $\varphi^{l(c_i)}$ préserve les orientations de $(\RR\surg)_j$ alors on peut homotoper $\Phi^{l(c_i)}$ en un automorphisme de $(\RR N)_j$ qui relève l'identité sur $(\RR\surg)_j$. Or l'action d'un tel automorphisme est la même sur les structures $Pin^+$ et $Pin^-$.

Supposons maintenant que $\varphi^{l(c_i)}$ échange les orientations de $(\RR\surg)_j$. On choisit un isomorphisme entre $(\RR N)_j$ et le modèle $[0,1]\times \RR^n /(0,v)\sim(1,rv)$, où $r$ est l'identité si $(\RR N)_j$ est orientable, et la symétrie orthogonale d'axe $e_1 = (1,0,\ldots,0)$ sinon. L'automorphisme $\Phi^{l(c_i)}_{|(\RR N)_j}$ induit un automorphisme $f$ du modèle, et quitte à l'homotoper, on peut supposer qu'il est de la forme $f(t,v) = (1-t,f_t(v))$ où $f_t$ est un élément de $O_n(\RR)$ et $f_0 = r f_1 r$. Or, deux relevés de $t\mapsto 1-t$ diffèrent d'un automorphisme au-dessus de l'identité et on sait que ces derniers agissent de la même manière sur les structures $Pin^+$ et $Pin^-$. Il suffit donc d'étudier l'action de l'automorphisme $g:(t,v)\mapsto (1-t,v)$ sur les structures $Pin^+$ et $Pin^-$ de $[0,1]\times \RR^n /(0,v)\sim(1,rv)$.

Notons $\pi_\pm:Pin_n^\pm(\RR)\rightarrow O_n(\RR)$ les revêtements doubles et considérons les structures $Pin^\pm$ sur $[0,1]\times \RR^n /(0,v)\sim(1,rv)$ données par $[0,1]\times Pin_n^\pm(\RR) /(0,p)\sim(1,e_1^\pm p)\xrightarrow{\pi^\pm} [0,1]\times O_n(\RR) /(0,v)\sim(1,rv)$ avec $\pi_{\pm}(e_1^\pm) = r$. Un relevé $\tilde{g}^\pm$ de $g$ au niveau des structures $Pin^\pm$ doit vérifier pour tout $p\in Pin^\pm_n(\RR)$, $g\circ \pi_\pm (p) = \pi_\pm \circ\tilde{g}^\pm(p)$. Donc les seuls relevés possibles de $g$ sont $\tilde{g}^\pm: (t,p)\mapsto (1-t,p)$ ou $\tilde{g}^\pm: (t,p)\mapsto (1-t,-p)$.

 Or ces relevés doivent aussi vérifier pour tout $p\in Pin^\pm_n(\RR)$, $\tilde{g}_0^\pm(p) = e^\pm_1\tilde{g}_1^\pm(e^\pm_1 p)$. Mais ceci n'est possible que si $e^\pm_1 e_1^\pm = 1$. Ainsi, si $(\RR N)_j$ est orientable $g$ préserve les structures $Pin^\pm$, et si $(\RR N)_j$ n'est pas orientable, $g$ préserve les structures $Pin^+$ mais pas les structures $Pin^-$.
\end{proof}

\begin{Rem}
  Notons que si $\Phi$ est un élément de $\riso{N}$, la valeur de $\varepsilon(\Phi_{\PPP^+})$ calcule aussi l'action de $\Phi$ sur $\PV =\dd \bigotimes_{i=1}^kPin^+((\RR N)_i)$. D'autre part, si $\varphi \in \RR Diff^+(\surg,w_1(\RR N))$, alors la valeur de $\varepsilon(\sigma_{\varphi}^-)$ calcule aussi l'action de $\varphi$ sur les orientations de $\dd\bigotimes_{w_1(\RR N)([\RR\surg]_i) = 1} T_{x_i}\RR \surg$, où $x_i\in (\RR\surg)_i$.
\end{Rem}

 Considérons un fibré en droites réelles $L$ au-dessus d'un cercle et notons $\TRR$ le fibré trivial au-dessus de ce même cercle. Alors, pour tout $n\geq 0$ nous contruisons une structure $Pin^+$ sur $L\oplus \TRR^{\oplus n}$ de la façon suivante. Fixons tout d'abord un isomorphisme entre $L$ et $[0,1]\times\RR /(0,v)\sim (1,rv)$, où $r$ vaut $1$ si $L$ est orientable et $-1$ sinon. Notons qu'on a en particulier orienté la base du fibré. Fixons une fois pour toutes un relevé $e$ de $-1$ dans $Pin_1^+(\RR)$. Lorsque $L$ est orientable, on lui associe la structure $Pin^+$ donnée par $[0,1]\times Pin_1^+(\RR)/(0,p)\sim(1,p)$. Dans le cas contraire on prend $[0,1]\times Pin_1^+(\RR)/(0,p)\sim(1,ep)$. On obtient alors une structure $Pin^+$ notée $\PPP^+_{L}$ sur $L\oplus \TRR^{\oplus n}$ par stabilisation (voir \cite{kirby}).

\begin{Lemme}\label{pintriv}
  Soit $n\in \NN$. Considérons $L$ et $L'$ deux fibrés en droites réelles au-dessus d'un cercle et le fibré trivial $\TRR$ au-dessus du même cercle. Les structures $\PPP^+_{L}$ et $\PPP^+_{L'}$ ne dépendent pas des choix effectués.

De plus tout isomorphisme entre $L \oplus \TRR^{\oplus n}$ et $L' \oplus \TRR^{\oplus n}$ de la forme $\Phi_L\oplus \Phi_1\oplus\ldots\oplus \Phi_n$, où $(\Phi_L,\varphi)$ est un isomorphisme entre $L$ et $L'$, et $(\Phi_i,\varphi)$ sont des automorphismes de $\TRR$ préservant les orientations, envoie la structure $\PPP^+_{L}$ sur $\PPP^+_{L'}$. En particulier, lorsque $L = L'$, un automorphisme de la forme $\Phi_L\oplus \Phi_1\oplus\ldots\oplus\Phi_n$ préserve les structures $Pin^+$ de $L\oplus \TRR^{\oplus n}$.
\end{Lemme}

Le Lemme \ref{pintriv} nous sera utile dans la Proposition \ref{lemrangun} ainsi que dans le Théoreme \ref{totalaction}.

\begin{proof}
Soit $M = [0,1]\times\RR^{n+1} /(0,v)\sim (1,rv)$ où $r$ est l'identité de $\RR^{n+1}$ si $L$ et $L'$ sont orientables et la réflexion $r_0$ définie par $r_0(x_0,x_1,\ldots,x_n)= (-x_0,x_1,\ldots,x_n)$ sinon. Par construction de $\PPP^+_L$ et $\PPP^+_{L'}$, il suffit de vérifier qu'un automorphisme $g$ de $M$ de la forme $g(t,x) = (t,(r_0)^{\varepsilon_0} (x))$ ou bien $g(t,x) = (1-t,(r_0)^{\varepsilon_0} (x))$, avec $\varepsilon_0 \in \{0,1\}$ préserve les structures $Pin^+$ de $M$.

Notons $e_0$ le relevé de $r_0$ dans $Pin^+_{n+1}(\RR)$ obtenu par stabilisation de l'élément $e\in Pin^+_1(\RR)$ fixé. Lorsque le fibré $M$ est orientable, on lui associe donc la structure $Pin^+$ donnée par $[0,1]\times Pin_{n+1}^+(\RR)/(0,p)\sim(1,p)$. Dans le cas contraire on prend $[0,1]\times Pin_{n+1}^+(\RR)/(0,p)\sim(1,e_0 p)$. Il faut vérifier que les automorphismes précédents admettent un relevé comme automorphismes des fibrés $Pin^+_{n+1}(\RR)$ principaux précédents. Or ceci découle directement du fait que $e_0e_0 = 1$.
\end{proof}

\begin{Definition}\label{pincan}
  On dira que la struture $\PPP^+_L$ sur $L\oplus \TRR^{\oplus n}$ ainsi obtenue est associée à $L$.
\end{Definition}

\begin{Rem}
  Insistons sur le fait que le Lemme \ref{pintriv} n'est plus vrai si l'on remplace $Pin^+$ par $Pin^-$.
\end{Rem}

 \'Etant donné un difféomorphisme réel $\varphi$ de $(\surg,c_\sur)$, on définit l'automorphisme $\relphi\in\riso{\TCC^{\oplus n}}$ du fibré trivial $(\TCC^{\oplus n}, conj)$ de rang $n$ par
\[
\begin{array}{l c l c}
\relphi : & \TCC^{\oplus n} & \rightarrow & \TCC^{\oplus n}  \\
            & (x,v)  & \mapsto    & (\varphi(x),v).
\end{array}
\]
Cet automorphisme est un relevé de $\varphi$.

\'Enonçons maintenant la Proposition \ref{lemrangun} qui permet de décomposer le signe de l'action d'un automorphisme sur les orientations de $\Det$ en trois termes.

\begin{Proposition}\label{lemrangun}
  Soit $(N,c_N)$ un fibré vectoriel complexe de rang au moins deux sur $(\surg,c_\sur)$ et $(\Phi,\varphi)\in\riso{N}$. Alors le signe de l'action de $\Phi$ sur les orientations du fibré $\Det$ est égal au produit de $\varepsilon(\Phi_{\PPP^+})$ avec le signe de l'action de $\det(\Phi)$ sur les orientations de $\Det[\det(N)]$ et le signe de l'action de $\relphi$ sur les orientations de $\Det[\TCC^{\oplus \rang(N)-1}]$.
\end{Proposition}

Dans le cas d'une courbe dont la partie réelle est vide, non fixons la convention que $\varepsilon(\Phi_{\PPP^+}) = 1$.

\begin{proof}
Fixons un isomorphisme $f : (N,c_N)\rightarrow (\det(N)\oplus \TCC^{\oplus \rang(N)-1},c_{\det(N)} \oplus conj)$ dont le déterminant vaut $1$. L'automorphisme $\Phi$ a même action sur les orientations de $\Det$ que $f\circ \Phi \circ f^{-1}$ sur les orientations de $\Det[\det(N)\oplus \TCC^{\oplus \rang(N)-1}]$. \'Ecrivons maintenant
\[
f\circ \Phi \circ f^{-1} = (\det(\Phi)\oplus \relphi)\circ \left((\det(\Phi)\oplus \relphi)^{-1}\circ f\circ \Phi \circ f^{-1}\right).
\]
L'automorphisme $g = (\det(\Phi)\oplus \relphi)^{-1}\circ f\circ \Phi \circ f^{-1}$ est au-dessus de l'identité et, d'après le Lemme \ref{pintriv}, lorsque $\RR\surg\neq \emptyset$, on a $\varepsilon(g_{\PPP^+}) = \varepsilon(\Phi_{\PPP^+})$. De plus, $\det(g) =\det(\Phi)^{-1}\circ \det(f)\circ \det(\Phi)\circ \det(f)^{-1} = \det(\Phi)^{-1}\circ \det(\Phi) = id_{\det(N)}$, donc son action sur les orientations de $\Det[\det(N)]$ est triviale.

Lorsque $\RR\surg\neq \emptyset$, le Théorème 3.1 de \cite{article1} implique que le signe de l'action de $g$ sur les orientations de $\Det[\det(N)\oplus \TCC^{\oplus \rang(N)-1}]$ est donné par $\varepsilon(\Phi_{\PPP^+})$. 

Lorsque $\RR\surg = \emptyset$, d'après le Lemme 2.2 de \cite{article1}, $g$ est homotope à l'identité.

Ainsi, dans tous les cas, le signe de l'action de $\Phi$ sur les orientations de $\Det$ est donné par le produit de $\varepsilon(\Phi_{\PPP^+})$ par le signe de l'action de $(\det(\Phi)\oplus \relphi)$ sur les orientations de $\Det[\det(N)\oplus \TCC^{\oplus \rang(N)-1}]$. Ce dernier est donné par le produit du signe de l'action de $\det(\Phi)$ sur les orientations de $\Det[\det(N)]$ et du signe de l'action de $\relphi$ sur les orientations de $\Det[\TCC^{\oplus \rang(N)-1}]$.
\end{proof}

L'action de $\relphi$ sur les orientations de $\Det[\TCC^n]$ est étudiée dans le \S \ref{extrivial}, et le cas du rang $1$ est traité au \S \ref{caspart}.

\subsection[Action de $\Gamma(\surg,c_\sur)$ sur l'espace de Teichmüller réel]{Action du groupe des difféotopies réel sur l'espace de Teichmüller réel}\label{diffeoteich}

\subsubsection{Les cas des fibrés trivial et canonique}\label{extrivial}

Soit $(\surg,c_\sur)$ une courbe réelle et considérons tout d'abord le fibré vectoriel complexe trivial de rang $n\geq 1$, $\TCC^{\oplus n}$ muni de la structure réelle naturelle notée $conj$. Notons $s_{\TCC^{\oplus n}} : \raut{\TCC^{\oplus n}}\rightarrow \{-1,1\}$ le morphisme calculant l'action du déterminant d'un élément de $\raut{\TCC^{\oplus n}}$ sur les orientations du fibré $\Det[\TCC]$. Nous l'avons déjà étudié au \S 4 de \cite{article1} (voir notamment les Théorèmes 4.2 et 4.3). Pour tout difféomorphisme réel $\varphi$ de $(\surg,c_\sur)$, le déterminant de l'automorphisme $\relphi$ nous fournit une section réelle du fibré $(\varphi^*(\TCC,conj))\otimes (\TCC,conj)^*$ qui devient donc isomorphe à $(\TCC,conj)$. Ainsi, on peut étendre $s_{\TCC^n}$ à tout $\riso{\TCC^n}$.

D'autre part, un difféomorphisme réel $\varphi$ de $(\surg,c_\sur)$ induit un automorphisme $\varphi^*$ de $H^1(\surg,\RR)_{-1} = \{\alpha \in H^1(\surg,\RR)\ |\ (c_{\sur})^*\alpha = -\alpha\}$.

\begin{Proposition}\label{isotriv}
  Le signe de l'action d'un élément $(\Phi,\varphi)$ de $\riso{\TCC^{\oplus n}}$ sur les orientations du fibré déterminant $\Det[\TCC^{\oplus n}]$ est donné par 
\[
\label{signtrivial}
\varepsilon(\Phi_{\PPP^{\pm}}) s_{\TCC^{\oplus n}}(\Phi)\det(\varphi^*)^n\in\{-1,1\}.
\tag{$\varepsilon$}
\]
\end{Proposition}

La Proposition \ref{isotriv} correspond au cas particulier du fibré trivial et du diviseur nul dans le Théorème \ref{totalaction}. Notons que les deux premiers termes $\varepsilon(\Phi_{\PPP^{\pm}})$ et $s_{\TCC^{\oplus n}}(\Phi)$ sont analogues à ceux obtenus aux Théorèmes A et B de \cite{article1}, et ne font pas intervenir le difféomorphisme $\varphi$, au contraire du troisième terme $\det(\varphi^*)^n$.

Commençons par quelques remarques. La projection $(\DB,J)\in\rop[\TCC^{\oplus n}]\mapsto J\in\RR J(\surg)$ admet une section naturelle $J\in\RR J(\surg)\mapsto (\DDJ^{\oplus n},J)\in\rop[\TCC^{\oplus n}]$, où $\DDJ = \dd\frac{1}{2} (\d + i\circ\d\circ J)$. Le fibré déterminant est donc défini sur $\RR J(\surg)$, et est encore trivialisable sur cet espace qui est contractile. De plus, il s'écrit
\[
\Det[\TCC^{\oplus n}]_{|\RR J(\surg)} = \left(\HH^0(\surg,\TCC)_{+1}\otimes\left(\Lambda^g\HH^1(\surg,\TCC)_{+1}\right)^*\right)^{\otimes n}
\]
où $\HH^0(\surg,\TCC)_{+1}$ et $\HH^1(\surg,\TCC)_{+1}$ sont les fibrés vectoriels réels sur $\RR J(\surg)$ de fibres respectives $H_{\DDJ}^0(\surg,\TCC)_{+1}$ et $H_{\DDJ}^1(\surg,\TCC)_{+1}$ au-dessus de $J$. D'autre part, comme $(\relphi)^*(\DDJ,J) = (\DB_{\varphi^* J},\varphi^* J)$, l'action de l'automorphisme $\relphi$ sur les orientations du fibré $\Det[\TCC^{\oplus n}]$ est donnée par son action sur celles de $\Det[\TCC^{\oplus n}]_{|\RR J(\surg)}$. Nous pouvons maintenant énoncer le résultat intermédiaire suivant.

\begin{Lemme}\label{isopart}
 Nous avons un isomorphisme canonique
\[
\Det[\TCC^{\oplus n}]_{|\RR J(\surg)} \cong \det(\HH^1(\surg,\RR)_{-1})^{\otimes n},
\]
où $\HH^1(\surg,\RR)_{-1}$ est le fibré trivial au-dessus de $\RR J(\surg)$ de fibre $H^1(\surg,\RR)_{-1}$, tel que pour tout $\varphi\in\RR Diff^+(\surg)$, le diagramme
\[
\label{diagtriv}
\xymatrix{
 \Det[\TCC^{\oplus n}]_{|\RR J(\surg)} \ar[d]_{(\relphi)_*} \ar[r] & \det(\HH^1(\surg,\RR)_{-1})^{\otimes n} \ar[d]^{(\varphi^{-1})^*}\\
  \Det[\TCC^{\oplus n}]_{|\RR J(\surg)} \ar[r] & \det(\HH^1(\surg,\RR)_{-1})^{\otimes n}
}
\tag{$*$}
\]
commute.
\end{Lemme}

\begin{proof}
D'une part, l'évaluation des fonctions $J$-holomorphes réelles sur $(\surg,c_\sur)$ donne un isomorphisme canonique
\[
\HH^0(\surg,\TCC)_{+1}\xrightarrow{ev} \underline{\RR}
\]
qui nous permet de dire que $\relphi$ agit trivialement sur la partie $\HH^0(\surg,\TCC)_{+1}$.
D'autre part, nous avons par dualité de Serre 
\[
\left(\HH^1(\surg,\TCC)_{+1}\right)^*\cong \HH^0(\surg,K_\sur)_{-1},
\]
où $\HH^0(\surg,K_\sur)_{-1}$ est le fibré sur $\RR J(\surg)$ de fibre $H^0(\surg,K_{\sur,J})_{-1}$ au-dessus de $J$, et le diagramme suivant commute
\[
\xymatrix{
\left(\HH^1(\surg,\TCC)_{+1}\right)^* \ar[d]_{^t(\relphi^*)} \ar[r]^{Serre} & \HH^0(\surg,K_\sur)_{-1} \ar[d]^{(\d \varphi^{-1})^*} \\
\left(\HH^1(\surg,\TCC)_{+1}\right)^* \ar[r]_{Serre} & \HH^0(\surg,K_\sur)_{-1}.
}
\]
 Enfin, l'intégration des $(1,0)$-formes $J$-holomorphes le long des lacets sur $\surg$ définit une dualité
\[
\HH^0(\surg,K_\sur)_{-1}\times \HH_1(\surg,\RR)_{+1} \rightarrow i\RR,
\]
où $\HH_1(\surg,\RR)_{+1}$ est le fibré trivial sur $\RR J(\surg)$ de fibre $H_1(\surg,\RR)_{+1}$, et nous avons à nouveau un diagramme 
\[
\xymatrix{
\HH^0(\surg,K_\sur)_{-1} \ar[d]_{(\d \varphi^{-1})^*}  \ar[r] & \left(\HH_1(\surg,\RR)_{+1}\right)^* \ar[d]^{ (\varphi^{-1})^*}\\
\HH^0(\surg,K_\sur)_{-1}       \ar[r] & \left(\HH_1(\surg,\RR)_{+1}\right)^* 
}
\]
qui commute grâce à la formule du changement de variables pour $\omega\in
  H^0_{\bar{\partial}_{K_{\sur},J}}(\surg,K_{\sur})$ et $c\in
  H_1(\surg,\ZZ)$
  \[
  \dd\int_{c} \omega\circ \d \varphi^{-1} = \dd\int_{\varphi^{-1}_* c}\omega.
  \]
Tout compte fait, nous avons un isomorphisme
\[
\Det[\TCC^{\oplus n}]_{|\RR J(\surg)} \cong \det(\HH^1(\surg,\RR)_{-1})^{\otimes n}
\]
et le diagramme (\ref{diagtriv}) commute bien.
\end{proof}

\begin{proof}[Démonstration de la Proposition \ref{isotriv}]
  Vérifions le résulat tout d'abord pour un automorphisme de la forme $\relphi$ avec $\varphi\in\RR Diff^+(\surg)$. Par construction, nous avons $s_{\TCC^{\oplus n}}(\relphi)=1$, et d'après le Lemme \ref{pintriv} $\varepsilon((\relphi)_{\PPP^\pm})=1$. Le signe topologique (\ref{signtrivial}) est donc simplement égal à $\det(\varphi^*)^n$.

D'autre part, le Lemme \ref{isopart} nous montre que l'action de $\relphi$ sur les orientations du fibré $\Det[\TCC^{\oplus n}]_{|\RR J(\surg)}$ est la même que celle de $(\varphi^{-1})^*$ sur $\det(H^1(\surg,\RR)_{-1})^{\otimes n}$. On retrouve ainsi le signe (\ref{signtrivial}) pour $\relphi$.

Pour un automorphisme $(\Phi,\varphi)\in\riso{\TCC^{\oplus n}}$ quelconque, on commence par le décomposer en $f\circ \relphi$, où $f\in\raut{\TCC^{\oplus n}}$. Or, nous avons déjà calculé le signe de l'action de $\relphi$, et le signe de celle de $f$ est donné par $\varepsilon(f_{\PPP^\pm}) s_{\TCC^{\oplus n}}(f)$.
Le produit des deux signes nous donne le résultat voulu.
\end{proof}

\begin{Rem}\label{remrangun}
  Le Lemme \ref{isopart} nous permet de préciser le résultat obtenu dans la Proposition \ref{lemrangun}. Nous obtenons en effet que le signe de l'action d'un automorphisme $\Phi\in\riso{N}$ est le produit de $\varepsilon(\Phi_{\PPP^+})$ par $\det(\varphi^*)^{\rang(N)-1}$ et par le signe de l'action de $\det(\Phi)$ sur les orientations de $\Det[\det(N)]$. Il reste encore à décrire cette dernière action, ce que nous faisons au \S \ref{caspart}.
\end{Rem}

\begin{Lemme}\label{isoksigma}
Il existe un isomorphisme canonique
\[
\Det[K_\sur]_{|\RR J(\surg)} \cong \det(\HH^1(\surg,\RR)_{-1}),
\]
où $\Det[K_\sur]_{|\RR J(\surg)}$ est le fibré sur $\RR J(\surg)$ de fibre $\Lambda^g H^0(\surg,K_{\sur,J})_{+1} \otimes \left(H^1(\surg,K_{\sur,J})_{+1}\right)^*$ au-dessus de $J$, tel que pour tout $\varphi\in\RR Diff^+(\surg)$, le diagramme
\[
\xymatrix{
  \Det[K_\sur]_{|\RR J(\surg)} \ar[d]_{(\d \varphi^{-1})^*} \ar[r] & \det(\HH^1(\surg,\RR)_{-1}) \ar[d]^{(\varphi^{-1})^*}\\
  \Det[K_\sur]_{|\RR J(\surg)} \ar[r] & \det(\HH^1(\surg,\RR)_{-1})
}
\]
commute.
\end{Lemme}

\begin{proof}
  On procède comme dans le Lemme \ref{isopart}, sans avoir besoin de la dualité de Serre.
\end{proof}

\subsubsection{Action des difféomorphismes sur $\Det[T\surg]$}\label{extang}

Le Lemme \ref{isoksigma} permet aussi d'étudier le cas du fibré tangent à la courbe réelle. Rappelons pour cela qu'étant donné un opérateur de Cauchy-Riemann réel $\DB$ sur un fibré vectoriel complexe $(N,c_N)$, on a un opérateur $\DB^*$ sur $(N^*,c_{N^*})$ défini par
\[\label{dbetoile}
(\DB^*\alpha)(v) = \DDJ(\alpha(v)) - \alpha(\DB v),
\tag{$*$}
\]
pour $\alpha\in \Gamma(\surg,N^*)$ et $v\in\Gamma(\surg,N)$. 

\begin{Lemme}
  L'opérateur $\DB^*$ est bien défini et est un élément de $\rop[N^*]$. Par exemple, pour les fibrés $T\sur_{g,J}$ et $K_{\sur,J}$, on a
\[
\begin{array}{c}
\DB_{T\sur,J} = \frac{1}{2} (\nabla^J + J\circ\nabla^J\circ J),\\
\DB_{K_\sur,J} = \frac{1}{2} (d + i\circ \d \circ J),\\
\DB_{T\sur,J}^* = \DB_{K_\sur,J},  
\end{array}
\]
où $\nabla^J$ est la connexion de Levi-Civita associée à une métrique induite par $J$.
\end{Lemme}

\begin{proof}
  Fixons $\alpha\in\Gamma(\surg,N^*)$. Il nous faut vérifier que $\DB^*\alpha$ tel que donné par l'égalité (\ref{dbetoile}) ne dépend que de la valeur en chaque point de la section $v\in\Gamma(\surg,N)$. Prenons pour cela une fonction $f\in\CCC^\infty(\surg,\CC)$. Nous avons alors
\[
\begin{array}{l l}
  \DDJ(\alpha(fv)) - \alpha(\DB (fv)) & = \alpha(v)\DDJ(f) + f\DDJ(\alpha(v)) - \alpha(\DDJ(f)v) - f\alpha(\DB v)\\
                                   & = f(\DDJ(\alpha(v)) - \alpha(\DB v)),
\end{array}
\]
ce qui montre que $\DB^*\alpha$ est bien défini. La formule de Leibniz se montre de la même façon.

 D'autre part, une métrique associée à $J$ s'écrit $\lambda(z)\d z\otimes\d \overline{z}$ dans des coordonnées locales holomorphes. On vérifie alors que $\nabla^J_{\frac{\partial}{\partial \overline{z}}}\frac{\partial}{\partial z} = \frac{1}{\lambda}\frac{\partial \lambda}{\partial z}\frac{\partial}{\partial z}$ et $\nabla^J_{i\frac{\partial}{\partial \overline{z}}}\frac{\partial}{\partial z} = \frac{i}{\lambda}\frac{\partial \lambda}{\partial z}\frac{\partial}{\partial z}$. Ainsi, $\nabla^J\frac{\partial}{\partial z} + i\circ\nabla\frac{\partial}{\partial z}\circ i = 0$. Donc la structure holomorphe sur $T \sur_{g,J}$ est induite par l'opérateur $\DB_{T_\sur,J} = \dd\frac{1}{2}(\nabla^J + J\circ\nabla^J\circ J)$.

Enfin, toujours dans des coordonnées locales holomorphes, on prend $\omega = w\d z$ et $\xi = s\frac{\partial}{\partial z}$ et on calcule
\[
\begin{array}{l l}
  \DDJ(\omega(\xi))  & = \DDJ(ws)\\
                     & = w\DDJ(s) + \DDJ(w)s\\
                     & = \omega(\DB_{T_\sur,J} \xi) + (\DB_{K_\sur,J}\omega)(\xi),
\end{array}
\]
ce qui montre que $\DB_{T_\sur,J}^* = \DB_{K_\sur,J}$.
\end{proof}

L'application de $\rop$ dans $\rop[N^*]$ ainsi définie est de plus continue, donc est un homéomorphisme. On peut ainsi comparer les fibrés $\Det$ et $\Det[N^*]$, ce qui est l'objet de la Proposition \ref{detdual}.

\begin{Proposition}\label{detdual}
  Soit $(N,c_N)$ un fibré vectoriel complexe muni d'une structure réelle sur $(\surg,c_\sur)$. Supposons que $\RR N$ est orientable. On a un isomorphisme canonique $\Det = \Det[N^*]^*$, où $\Det[N^*]^*$ est défini au-dessus de $\rop$ grâce à la correspondance $\DB\mapsto\DB^*$.

 En particulier, si $(\Phi,\varphi)$ est un automorphisme de $(N,c_N)$ alors son action sur les orientations du fibré $\Det$ est la même que celle de $(^t\Phi^{-1},\varphi)$ sur les orientations de $\Det[N^*]$. 
\end{Proposition}

\begin{proof}
Considérons la suite exacte
  \[
0\rightarrow N\rightarrow N\oplus N^* \rightarrow N^*\rightarrow 0.
\]
Grâce à la correspondance $\DB\in\rop\mapsto\DB^*\in\rop[N^*]$, on a un isomorphisme canonique
\[
\Det = \Det[N\oplus N^*]\otimes\Det[N^*]^*,
\]
où la fibre de $\Det[N\oplus N^*]$ au dessus de $\DB\in\rop$ est $\ddet(\DB\oplus\DB^*)$.

De plus, nous avons des structures $Spin$ naturelles sur chaque composante de $\RR N\oplus \RR N^*$ données de la façon suivante. Fixons une structure $Spin$ sur une composante $(\RR N)_i$. Celle-ci induit une décomposition de $(\RR N)_i$ en somme directe de fibrés en droites réelles orientables définie à homotopie près. Par dualité, on obtient alors, à homotopie près, une décomposition en somme directe de fibrés en droites réelles orientables de $(\RR N)_i\oplus (\RR N^*)_i$ ainsi qu'une orientation canonique sur ce même fibré. Nous avons donc une structure $Spin$ sur $(\RR N)_i\oplus (\RR N^*)_i$. De plus, cette structure ne dépend pas du choix de structure fait sur $(\RR N)_i$.

Prenons alors un isomorphisme $f:N\oplus N^*\rightarrow \TCC^{\oplus 2\rang(N)}$ tel que
\begin{itemize}
\item $\det(f):\det(N\oplus N^*) = \TCC \rightarrow \det(\TCC^{\oplus 2\rang(N)})=\TCC$ vaut $1$ et
\item $f$ envoie la structure $Spin$ naturelle de chaque composante $\RR N\oplus \RR N^*$ sur celle naturelle de la composante correspondante de $\RR (\TCC^{\oplus 2\rang(N)})$.
\end{itemize}
Remarquons qu'un tel isomorphisme est uniquement défini à homotopie près: un autre tel isomorphisme différera de $f$ par multiplication à droite par un automorphisme de $\TCC^{\oplus 2\rang(N)}$ de déterminant $1$ et préservant les structures $Spin$ sur $\RR (\TCC^{\oplus 2\rang(N)})$. Or d'après le Lemme 2.2 de \cite{article1} ces automorphismes sont homotopes à l'identité.

Nous avons ainsi grâce à la Proposition \ref{isotriv} un isomorphisme canonique entre $\Det[N\oplus N^*]$ et $\det(\HH^1(\surg,\RR)_{-1})^{\otimes 2\rang(N)}$. Comme ce dernier fibré est canoniquement trivial, nous obtenons l'isomorphisme voulu dans l'énoncé.
\end{proof}

\begin{Rem}
  Insistons sur une conséquence de la Proposition \ref{detdual}. Si $(N,c_N)$ et $(N',c_{N'})$ sont deux fibrés vectoriels complexes sur $(\surg,c_\sur)$, $\DB\in\rop[N]$ et $\DB'\in\rop[N']$ alors on n'a pas en général d'isomorphisme canonique $\ddet(\DB\otimes\DB') = \ddet(\DB)\otimes \ddet(\DB')$. En effet, en prenant $N$ de rang $1$, $N' = N^*$ et $\DB' = \DB^*$, l'action d'un élément de $(\Phi,\varphi)\in\riso{N}$ sur les orientations de $\Det\otimes \Det[N^*]$ est triviale d'après la Proposition \ref{detdual}, alors que l'automorphisme induit sur $N\otimes N^* = \TCC$ est $\relphi$ dont l'action sur les orientations de $\Det[N\otimes N^*] = \Det[\TCC]$ est donnée par $\det(\varphi^*)$ d'après la Proposition \ref{isotriv}.
\end{Rem}

Nous avons comme corollaire immédiat du Lemme \ref{isoksigma} et de la Proposition \ref{detdual} appliquée au fibré tangent à $\surg$ le résultat suivant.

\begin{Corollaire}\label{isotsigma}
 Nous avons un isomorphisme canonique
\[
\Det[T\surg]_{|\RR J(\surg)} \cong \det(\HH_1(\surg,\RR)_{+1}),
\]
où $\Det[T\surg]_{|\RR J(\surg)}$ est le fibré sur $\RR J(\surg)$ de fibre $H^0(\surg,T\sur_{g,J})_{+1} \otimes \left(\Lambda^{\max} H^1(\surg,T\sur_{g,J})_{+1}\right)^*$ au-dessus de $J$, tel que pout tout $\varphi\in\RR Diff^+(\surg)$, le diagramme
\[
\xymatrix{
  \Det[T\surg]_{|\RR J(\surg)}  \ar[d]_{\d\varphi} \ar[r] & \det(\HH_1(\surg,\RR)_{+1}) \ar[d]^{\varphi_*}\\
  \Det[T\surg]_{|\RR J(\surg)} \ar[r] & \det(\HH_1(\surg,\RR)_{+1})
}
\]
commute.\qed
\end{Corollaire}

\subsubsection{Espace de Teichmüller réel}\label{parteich}

Afin de simplifier l'exposé, nous supposons dans ce paragraphe que les courbes considérées sont de genre au moins deux. Nous commençons par rappeler quelques faits concernant l'espace de Teichmüller réel associé à une courbe réelle $(\surg,c_\sur)$ (voir en particulier \cite{sepp} et \cite{earle}).

Le groupe $Diff(\surg)$ des difféomorphismes de $\surg$ agit sur l'ensemble $J(\surg)$ des structures complexes sur $\surg$ par tiré en arrière:
\[\label{a}
(\varphi,J)\in Diff(\surg)\times J(\surg) \mapsto \varphi^*J = s(\varphi)\d\varphi^{-1}\circ J \circ\d\varphi,
\tag{$*$}
\]
où $s(\varphi)$ vaut $1$ lorsque $\varphi\in Diff^+(\surg)$ et $-1$ sinon.

\begin{Definition}
  L'espace de Teichmüller $\TT(\surg)$ associé à la surface compacte orientée $\surg$ est le quotient de l'ensemble des structures complexes de $\surg$ compatibles avec son orientation par l'action du sous-groupe $Diff_0(\surg)$ de $Diff(\surg)$ formé des difféomorphismes homotopes à l'identité,
\[
\TT(\surg)=J(\surg)/Diff_0(\surg).
\]

L'espace de Teichmüller réel associé à la courbe réelle $(\surg,c_\sur)$ est le quotient de $\RR J(\surg)$ par l'action du sous-groupe $\RR Diff_0(\surg)$ de $Diff_0(\surg)$ formé des difféomorphismes homotopes à l'identité qui commutent avec $c_\sur$,
\[
\TT(\surg,c_\sur)=\RR J(\surg)/\RR Diff_0(\surg).
\]
\end{Definition}

\begin{Lemme}\label{trivialite}
  Soit $\varphi\in Diff_0(\surg)$. S'il existe $J\in\RR J(\surg)$ telle que $\varphi^*J \in\RR J(\surg)$, alors $\varphi$ commute avec $c_\sur$. 
\end{Lemme}

\begin{proof}
  Comme $\varphi^*J\in\RR J(\surg)$,
\[
\begin{array}{l l}
(\varphi\circ c_\sur\circ\varphi^{-1})^*J & = (\varphi^{-1})^*(c_\sur)^*\left(\varphi^* J\right)  \\
                                        & = (\varphi^{-1})^*\varphi^* J\\
                                        & = J.
\end{array}
\]
Le difféomorphisme $c_\sur\circ\varphi\circ c_\sur \circ\varphi^{-1}$ est donc un automorphisme pour $J$. Il est de plus homotope à l'identité. Ainsi $c_\sur\circ\varphi\circ c_\sur \circ\varphi^{-1} = id$, et $c_\sur\circ\varphi = \varphi\circ c_\sur$.
\end{proof}

Le Lemme \ref{trivialite} nous assure que l'inclusion $\RR J(\surg)\rightarrow J(\surg)$ induit une injection de $\TT(\surg,c_\sur)$ dans $\TT(\surg)$. De plus, l'image de cette application est contenue dans le lieu fixe $\TT(\surg)_{c_\sur}\subset \TT(\surg)$ de l'action de $c_\sur$ sur $\TT(\sur)$ induite par (\ref{a}).

\begin{Proposition}\label{krav}
  \begin{itemize}
  \item [(Kravetz \cite{kravetz})] L'espace $\TT(\surg,c_\sur)$ est homéomorphe à la boule unité ouverte d'un espace vectoriel réel normé de dimension $3g-3$.
  \item [(Earle \cite{earle})] L'injection $\TT(\surg,c_\sur)\rightarrow \TT(\surg)_{c_\sur}$ est un homéomorphisme.
  \end{itemize}
\end{Proposition}

L'espace de Teichmüller réel est donc contractile et hérite d'une structure de variété analytique réelle venant de celle de $\TT(\surg)_{c_\sur}$.
En particulier, l'espace de Teichmüller réel associé à une courbe réelle est toujours orientable. 

D'autre part, l'action du groupe $Diff^+(\surg)$ (resp. $\RR Diff^+(\surg)$) sur $J(\surg)$ (resp. sur $\RR J(\surg)$) descend en une action du groupe des difféotopies (resp. du groupe des difféotopies réel) sur l'espace de Teichmüller (resp. réel).

\begin{Definition}
  Le groupe $\Gamma(\surg)$ des difféotopies de $\surg$ est l'ensemble des classes d'homotopie de difféomorphismes de $\surg$ préservant l'orientation fixée,
\[
\Gamma(\surg) = Diff^+(\surg)/Diff_0(\surg).
\]

Le groupe $\Gamma(\surg,c_\sur)$ des difféotopies réel de $(\surg,c_\sur)$ est l'ensemble des classes d'homotopie de difféomorphismes réels de $(\surg,c_\sur)$ préservant l'orientation fixée,
\[
\Gamma(\surg,c_\sur) = \RR Diff^+(\surg)/\RR Diff_0(\surg).
\]
\end{Definition}

Nous avons à nouveau une action de $c_\sur$ sur $\Gamma(\surg)$ par conjugaison. Notons $\Gamma(\surg)_{c_\sur}$ son lieu fixe.

\begin{Lemme}
  L'inclusion naturelle $\Gamma(\surg,c_\sur)\rightarrow\Gamma(\surg)$ a pour image $\Gamma(\surg)_{c_\sur}$.
\end{Lemme}

\begin{proof}
Fixons un difféomorphisme $\varphi\in\Gamma(\surg)_{c_\sur}$ et montrons qu'il est homotope à un difféomorphisme réel. Choisissons une structure complexe $J\in\RR J(\surg)$. Alors par définition, il existe $\psi\in Diff_0(\surg)$ tel que
\[
c_\sur\circ\varphi\circ c_\sur = \varphi\circ\psi,
\]
et
\[
(c_\sur)^*\varphi^* J = (c_\sur)^*\varphi^* (c_\sur)^* J = \psi^* \varphi^* J.
\]
Donc $\varphi^*J$ est un élément de $\TT(\surg)_{c_\sur}$. D'après la Proposition \ref{krav}, il existe un difféomorphisme $\gamma$ homotope à l'identité tel que $\gamma^*\varphi^* J$ est réelle. Alors, le difféomorphisme $\left((\varphi\circ\gamma)\circ c_\sur\right)\circ\left(c_\sur\circ(\varphi\circ\gamma)\right)^{-1}$ est un automorphisme pour la structure complexe $J$ et est homotope à l'identité. Comme la surface $\surg$ est de genre au moins deux, ce difféomorphisme est égal à l'identité. D'où $(\varphi\circ\gamma)\circ c_\sur = c_\sur\circ(\varphi\circ\gamma)$. Ainsi $\varphi$ est homotope à $\varphi\circ\gamma$ qui est réel.
\end{proof}

Remarquons avant de poursuivre que l'on pourrait aussi définir les groupes de difféotopies en considérant les difféomorphismes à isotopie près. Il n'est pas évident a priori que l'on obtienne les mêmes groupes. C'est l'objet de la Proposition \ref{iddiff}.

\begin{Proposition}[Earle-Eells \cite{earleeels}]\label{iddiff}
  Les groupes $Diff_0(\surg)$ et $\RR Diff_0(\surg)$ sont les composantes connexes de l'identité des groupes $Diff(\surg)$ et $\RR Diff(\surg)$.
\end{Proposition}

\begin{Theoreme}\label{actionteich}
Soit $(\surg,c_\sur)$ une courbe réelle de genre au moins deux.
Le fibré des orientations de $\TT(\surg,c_\sur)$ est canoniquement isomorphe au fibré trivial $\det(\HH_1(\surg,\RR)_{+1})^*$.

  En particulier, l'action d'un élément $\varphi$ de $\Gamma(\surg,c_\sur)$ sur les orientations de l'espace de Teichmüller réel associé à $(\surg,c_\sur)$ est donnée par le signe du déterminant de l'application $\varphi_*: H_1(\surg,\RR)_{+1}\rightarrow H_1(\surg,\RR)_{+1}$.
\end{Theoreme}

Commençons par décrire le fibré tangent à $\TT(\surg,c_\sur)$.

\begin{Lemme}\label{tangent}
  Le fibré tangent à $\TT(\surg,c_\sur)$ est le fibré $\HH^1(\surg,T\surg)_{-1}$ de fibre $H^1(\surg,T\sur_{g,J})_{-1}$ au-dessus de $J\in \TT(\surg,c_\sur)$.
\end{Lemme}

\begin{proof}
  D'après la Proposition 1.1 de \cite{wel1}, la fibre du tangent à $\RR J(\surg)$ au point $J$ est
\[
\begin{array}{l l}
\Gamma(\surg,\Lambda_J^{0,1}\surg\otimes T\sur_{g,J})_{-1} =  \{\dot{J}\in\Gamma(\surg,End(T\surg))\ |& J\dot{J}=-\dot{J}J, \\
& -\d c_\sur\circ \dot{J}\circ \d c_\sur = \dot{J} \}.  
\end{array}
\]
Le tangent à $\TT(\surg,c_\sur)$ au point $[J]$ est donc le quotient de $\Gamma(\surg,\Lambda_J^{0,1}\surg\otimes T\sur_{g,J})_{-1}$ par l'action de l'algèbre le Lie de $\RR Diff_0(\surg)$. Pour calculer cette action, prenons un chemin $\varphi_t\in \RR Diff_0(\surg)$ tel que $\varphi_0 = \id$ et $\frac{\d \varphi_t}{\d t}_{|t=0} = X\in\Gamma(\surg,T\surg)_{+1}$. Notons $\nabla^J$ la connexion de Levi-Civita pour une métrique associée à $J$. Dérivons en zéro le chemin $\varphi_t^*J$:
\[
\begin{array}{l l}
  \dd\frac{\d(\varphi_t^*J)}{\d t}_{|t=0} & = \dd\frac{\d(\d \varphi_t^{-1} \circ J\circ \d \varphi_t)}{\d t}_{|t=0}\\
                                        & = -(\nabla^J X)\circ J + J\circ(\nabla^J X)\\
                                        & = J\DB_{T\surg,J} X\\
                                        & = \DB_{T\surg,J} (J X),
\end{array}
\]
car $\nabla^J$ est sans torsion
, $\nabla^J J=0$ et $\DB_{T\surg,J}$ est $\CC$-linéaire. Or la multiplication par $J$ est un isomorphisme entre $\Gamma(\surg,T\surg)_{+1}$ et $\Gamma(\surg,T\surg)_{-1}$. Ainsi, la fibre du tangent à $\TT(\surg,c_\sur)$ est le quotient
\[
\Gamma(\surg,\Lambda_J^{0,1}\surg\otimes T\sur_{g,J})_{-1}/\DB_{T\surg,J}\left(\Gamma(\surg,T\surg)_{-1}\right) = H^1(\surg,T\sur_{g,J})_{-1}.
\]
\end{proof}

\begin{proof}[Démonstration du Théorème \ref{actionteich}]
D'après le Lemme \ref{tangent},
\[
\det(T\TT(\surg,c_\sur)) = \det(\HH^1(\surg,T\surg)_{-1}).
\]
La multiplication par $J$ en chaque point de $\RR J(\surg)$ induit un isomorphisme entre $\HH^1(\surg,T\surg)_{-1}$ et $\HH^1(\surg,T\surg)_{+1}$ qui passe au quotient par l'action de $\RR Diff_0(\surg)$. Ainsi, d'après le Corollaire \ref{isotsigma}, et comme $H^0(\surg,T\sur_{g,J}) = 0$, on a un isomorphisme canonique
\[
\det(T\TT(\surg,c_\sur)) = \det(\HH^1(\surg,\RR)_{+1})^*.
\]
\end{proof}

Le Théorème \ref{actionteich} permet de calculer comme corollaire immédiat la première classe de Stiefel-Whitney du quotient $\TT(\surg,c_\sur)/\Gamma(\surg,c_\sur)$. 

  \begin{Definition}
L'espace des modules $\MM_{g,k,\varepsilon}$ des courbes de genre $g$ munies d'une involution anti-holomorphe dont la partie réelle est séparante, c'est-à-dire telle que la partie réelle coupe la courbe en deux composantes connexes, (resp. non séparante) si $\varepsilon = 1$ (resp. si $\varepsilon = 0$) et a $k$ composantes connexes est le quotient
\[
\MM_{g,k,\varepsilon} = \TT(\surg,c_\sur)/\Gamma(\surg,c_\sur),
\]
où $(\surg,c_\sur)$ est de type topologique $(g,k,\varepsilon)$.
  \end{Definition}

Toutefois, l'action de $\Gamma(\surg,c_\sur)$ sur $\TT(\surg,c_\sur)$ n'est pas libre, et le quotient a des singularités qui correspondent à des courbes ayant des automorphismes.

\begin{Proposition}
  Si $g$ est strictement plus grand que trois alors l'ensemble des structures complexes sur $(\surg,c_\sur)$ admettant au moins un automorphisme réel non trivial est de codimension au moins deux dans $\TT(\surg,c_\sur)$.
\end{Proposition}

\begin{proof}
On rappelle simplement rapidement le schéma de la preuve (pour plus de détails, on pourra lire Arbarello-Cornalba-Griffiths \cite{arbarello} Chapter XII, Proposition (2.5)). 

Soit $J$ une structure complexe sur $(\surg,c_\sur)$ admettant un automorphisme $\varphi$ d'ordre premier $p>1$. Alors, le groupe des automorphismes réels de $J$ agit linéairement sur un voisinage de $J$ dans $\TT(\surg,c_\sur)$, et la dimension du sous-espace sur lequel $\varphi$ se propage est égale à la dimension $d$ de $H^1(\surg,T\sur_{g,J})_{-1}^\varphi$ qui est le sous espace vectoriel de $H^1(\surg,T\sur_{g,J})_{-1}$ formé des éléments $\varphi$-invariants. En notant $g'$ le genre de la courbe $\surg/\varphi$ et $h$ le nombre de points ramifiés de la projection, on obtient
\[
d = 3g'-3 +h.
\]
 La formule de Riemann-Hurwitz donne aussi
\[
3g-3 - d = (p-1)(3g'-3) + \frac{1}{2}h(3p-5).
\]
On conclut en étudiant les différents cas selon le genre $g'$.
\end{proof}

Notons $\MM^*_{g,k,\varepsilon}$ le sous-ensemble de $\MM_{g,k,\varepsilon}$ formé des courbes sans automorphisme.

\begin{Corollaire}\label{delmum}
  Pour $g\geq 4$, la première classe de Stiefel-Whitney de $\MM^*_{g,k,\varepsilon}$ est donnée par
\[
w_1(\MM^*_{g,k,\varepsilon}) = w_1(\HH_1(\surg,\RR)_{+1}).
\]\qed
\end{Corollaire}

\subsection{Automorphismes et diviseurs}\label{caspart}

\subsubsection{Opérateurs de Cauchy-Riemann et diviseurs}\label{diviseur}

Nous passons maintenant au cas général d'un fibré en droites complexes $(N,c_N)$ muni d'une structure réelle sur $(\surg,c_\sur)$.

\begin{Definition}
 Un diviseur $D = \dd\sum_i a_i x_i$, $a_i\in\ZZ$ et $x_i\in\surg$, invariant par $c_\sur$ et tel que
\begin{itemize}
\item $D$ est de degré $\deg(N)$ et
\item pour chaque composante réelle $(\RR\surg)_s$ de $\surg$, la parité du degré du diviseur $D_{|(\RR\surg)_s} = \dd\sum_{x_i\in(\RR\surg)_s}a_i x_i$ est donnée par $w_1(\RR N)([\RR\surg]_s)$
\end{itemize}
sera dit compatible avec $(N,c_N)$.
\end{Definition}

Un diviseur compatible avec $(N,c_N)$ nous fournit des sections privilégiées de la projection $(\DB,J)\in\rop \mapsto J\in\RR J(\surg)$.

\begin{Definition}
Soit $D$ un diviseur compatible avec $(N,c_N)$. Une section $\DB_D : J\in\RR J(\surg)\mapsto (\DB,J)\in\rop$ sera dite associée à $D$ si pour chaque $J\in\RR J(\surg)$ le fibré $(N,\DB_{D,J})$ est isomorphe à $\OO_{\surg,J}(D)$.

Une section $\DB_D$ associée à $D$ induit un fibré en droites réelles sur $\RR J(\surg)$ dont la fibre au-dessus de $J\in\RR J(\surg)$ est l'ensemble des sections méromorphes réelles de $(N,c_N,\DB_{D,J})$ ayant $D$ pour diviseur. Comme $\RR J(\surg)$ est contractile, ce fibré en droites est orientable. Un couple formé d'une section $\DB_D$ associée à $D$ et d'une orientation du fibré en droite induit sur $\RR J(\surg)$ est appelé section polarisée associée à $D$, et on le notera $\DB^+_D$. 
\end{Definition}

\begin{Rem}
Le quotient de $\rop$ par l'action de $\raut{N}$ est le groupe de Picard universel $\RR \mathcal{P}ic_{w_1(\RR N)}^{\deg(N)}(\surg)$, qui est un fibré principal sur $\RR J(\surg)$ dont le groupe est la Jacobienne réelle de $(\surg,c_\sur)$. \`A un diviseur $D$ compatible avec $(N,c_N)$ est associée une section $\OO_{\surg}(D): \RR J(\surg) \rightarrow \RR \mathcal{P}ic_{w_1(\RR N)}^{\deg(N)}(\surg)$ de ce fibré. Le sous-ensemble $\ropd$ des éléments de $\rop$ qui se projettent sur cette section forme un fibré principal sur $\RR J(\surg)$ de groupe $\raut{N}/\RR^*$. Nous avons donc la situation suivante
\[
\xymatrix{
\ropd \ar@<-0.2pc>@{^(-}[r] \ar@<+0.2pc>@{-}[r] & \rop \ar[rr] \ar[rd] && \RR \mathcal{P}ic_{w_1(\RR N)}^{\deg(N)}(\surg) \ar[ld] \\
             & & \RR J(\surg) \ar@/_1pc/[ru]_{\OO_{\surg}(D)} \ar@/^1pc/[llu]^{\DB_D}
}
\]
Autrement dit, les sections associées à $D$ sont les sections du fibré $\ropd \rightarrow\RR J(\surg)$. En particulier, pour chaque structure complexe $J\in \RR J(\surg)$, il existe un voisinage $J\in V\subset \RR J(\surg)$, des familles continues de coordonnées locales $J'$-holomorphes $(\zeta_{J',x_i})_{J'\in V}$ sur des voisinages disjoints $\ZZ/2\ZZ$-équivariants $U_{x_i}$ des $x_i$, et une famille continue d'isomorphismes entre $(N,\DB_{D,J'},c_N)_{J'\in V}$ et les fibrés en droites holomorphes réels définis comme le recollement de $(\surg\setminus\{x_i\})\times \CC$ et $U_{x_i}\times\CC$ par les changements de cartes
\[
\begin{array}{c c c}
(\surg\setminus\{x_i\})\cap U_{x_i}\times \CC & \rightarrow & U_{x_i}\cap(\surg\setminus\{x_i\})\times\CC  \\
(z,v) & \mapsto & (z,\zeta_{J',x_i}(z)^{a_i}v).
\end{array}
\]

Le revêtement double $\ropd^+$ de $\ropd$ formé des opérateurs polarisés associés à $D$ est un fibré principal sur $\RR J(\surg)$ de groupe $\raut{N}/\RR_+^*$. Ses sections sont les sections polarisées associées à $D$.
\end{Rem}

De plus, nous pouvons décrire comment passer d'une section associée à un diviseur compatible à une autre.

\begin{Lemme}\label{sectionpol}
Soit $D$ un diviseur compatible avec $(N,c_N)$. Si $\DB_D$ et $\DB'_D$ sont deux familles d'opérateurs de Cauchy-Riemann sur $(N,c_N)$ associées au diviseur $D$, alors il existe une application continue $F : \RR J(\surg)\rightarrow \raut{N}$ telle que pour tout $J$ dans $\RR J(\surg)$ on a $\DB_{D,J} = F_J^*\DB'_{D,J}$. De plus, cette application est unique à multiplication par une fonction continue de $\RR J(\surg)$ dans $\RR^*$ près.

Si les familles d'opérateurs sont polarisées et que l'on impose que $F$ préserve les polarisations, alors $F$ est définie de façon unique à homotopie près.
\end{Lemme}

\begin{proof}
Le groupe $\raut{N}/\RR^*$ des automorphismes de $N$ modulo les automorphismes constants agit librement sur $\rop$. En effet, un élément de $\raut{N} = \RR\CCC^\infty(\surg,\CC^*)$ qui préserve un opérateur de Cauchy Riemann sur $N$ est holomorphe, et est donc de la forme $v\in N\mapsto \lambda v\in N$, avec $\lambda\in\RR^*$.

Par définition, pour chaque $J$ dans $\RR J(\surg)$ il existe un voisinage $V$ de $J$ et une unique application continue $[F]$ de $V$ dans $\raut{N}/\RR^*$ telle que pour tout $J'$ dans $V$ on a $\DB_{D,J'} = [F]_{J'}^*\DB'_{D,J'}$. Par unicité, $[F]$ est en fait définie sur tout $\RR J(\surg)$. Puis, comme $\RR J(\surg)$ est contractile $[F]$ admet un relevé dans $\raut{N}$, c'est-à-dire une application continue $F : \RR J(\surg)\rightarrow \raut{N}$ telle que pour tout $J$ dans $\RR J(\surg)$ on a $\DB_{D,J} = F_J^*\DB'_{D,J}$.

Si les deux familles sont polarisées, alors soit $F$ soit $-F$ préserve les polarisations, et deux familles d'automorphismes qui envoient une famille polarisée sur l'autre diffèrent d'une fonction continue sur $\RR J(\surg)$ à valeurs dans $\RR_+^*$.
\end{proof}

Considérons maintenant un quadruplet d'entiers positifs ou nuls $\ud = (r^+,r^-,s^+,s^-)\in \NN^4$. Notons $\surg^{\ud}$ le sous-ensemble de $(\RR\surg)^{r^++r^-}\times \surg^{s^+ + s^-}\setminus\Delta$, où $\Delta$ est la diagonale épaisse, formé des éléments $\ux = (\ux^+,\ux^-,\uz^+,\uz^-)$, où $\uz^+ = (z^+_1,c_\sur(z^+_1),\ldots,z^+_{s^+},c_\sur(z^+_{s^+}))$ et $\uz^- = (z^-_1,c_\sur(z^-_1),\ldots,z^-_{s^-},c_\sur(z^-_{s^-}))$ et tels que le diviseur $D_{\ux} = \dd\sum_{x\in\ux^{+}} x -\sum_{x\in\ux^{-}} x +\sum_{i=1}^{s^+} \left(z_i^+ + c_\sur(z^+_i)\right) - \sum_{i=1}^{s^-} \left(z_i^- + c_\sur(z^-_i)\right)$ est compatible avec $(N,c_N)$.

\begin{Definition}
 Si $\surg^{\ud}$ est non-vide, on dira que $\ud$ est adapté à $(N,c_N)$. Dans ce cas, à tout élément $\ux$ de $\surg^{\ud}$ on associe un diviseur $D_{\ux} = \dd\sum_{x\in\ux^{+}} x -\sum_{x\in\ux^{-}} x +\sum_{i=1}^{s^+} \left(z_i^+ + c_\sur(z^+_i)\right) - \sum_{i=1}^{s^-} \left(z_i^- + c_\sur(z^-_i)\right)$ qui est compatible avec $(N,c_N)$.
\end{Definition}
 
Considérons le fibré en espaces affines $pr_1^* \rop \xrightarrow{\pi} \RR J(\surg)\times \surg^{\ud}$, où $pr_1 : \RR J(\surg)\times \surg^{\ud}\rightarrow \RR J(\surg)$ est la première projection. On définit l'ensemble $\ropc\subset pr_1^*\rop$ dont la fibre au-dessus de $\ux$ est 
\[
\RR\CCCC_{D_{\ux}}(N) = \{(\DB,J)\in \RR\CCCC(N)\ |\ (N,\DB)\cong \OO_{\surg,J}(D_{\ux})\}.
\] 
On dira qu'un élément $\DB$ de $\RR\CCC_{D_{\ux}}(N)$ est polarisé si on a choisi une orientation de la droite réelle formée des sections méromorphes réelles de $(N,c_N,\DB)$ de diviseur $D_{\ux}$.

\begin{Lemme}
  L'ensemble $\ropc$ est un fibré $\raut{N}/\RR^*$-principal au-dessus de $\RR J(\surg)\times\surg^{\ud}$. Son revêtement double $\ropcp$ formé des opérateurs polarisés est un fibré $\raut{N}/\RR_+^*$-principal.\qed
\end{Lemme}

\begin{Rem}
 Le groupe $\riso{N}$ agit naturellement sur $\ropc$. Toutefois, il n'agit pas par automorphisme de fibré principal.
\end{Rem}

Pour chaque élément $\ux$ de $\surg^{\ud}$, le fibré $\RR\CCCC_{D_{\ux}}(N)$ au-dessus de $\RR J(\surg)$ est trivialisable. Toutefois, le fibré total $\ropc$ ne l'est pas. Notons $\raut{N}^+\subset \raut{N}$ le groupe des automorphismes au-dessus de l'identité qui préservent les orientations de $\Det$. Ce groupe est décrit au \S 4 de \cite{article1}. Rappelons simplement le résultat suivant.

\begin{Proposition}\label{rapp}
  Soit $(\surg,c_\sur)$ une courbe réelle et $(N,c_N)$ un fibré en droites complexes muni d'une structure réelle. Soit $a\subset \surg$ une courbe simple orientée. On considère les éléments $f_a$ de $\raut{N}$ de la forme suivante.
  \begin{enumerate}
  \item Si $a$ est une composante de $\RR \surg$, on choisit un voisinage tubulaire réel de $a$ de la forme $(\theta,t)\in S^1\times [-1,1]$, où $a$ correspond à $t=0$ et $c_\sur(\theta,t) = (\theta,-t)$. On pose alors $f_a(\theta,t) = -\e^{i\pi t}$ sur ce voisinage et on prolonge par $1$ en dehors (voir Figure \ref{ann}).
\item Si $a$ est une courbe globalement stable par $c_\sur$, on choisit un voisinage tubulaire réel de $a$ de la forme $(\theta,t)\in S^1\times [-1,1]$, où $a$ correspond à $t=0$ et $c_\sur(\theta,t) = (-\theta,-t)$. On pose alors $f_a(\theta,t) = -\e^{i\pi t}$ sur ce voisinage et on prolonge par $1$ en dehors (voir Figure \ref{ann}).
\item Si $a \cap c_\sur (a) = \emptyset$, on choisit un voisinage tubulaire de $a$ disjoint de son conjugué de la forme $(\theta,t)\in S^1\times [-1,1]$, où $a$ correspond à $t=0$. On pose alors $f_a(\theta,t) = -\e^{i\pi t}$ sur ce voisinage et on prolonge par $\overline{f_a\circ c_\sur}$ sur son conjugué et par $1$ en dehors (voir Figure \ref{ann}).
  \end{enumerate}
\begin{figure}[!h]
    \centering
    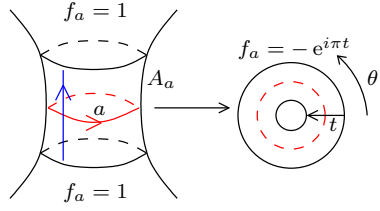
    \caption{Construction de $f_a$}
    \label{ann}
  \end{figure}
Dans le premier cas, $f_a$ est dans $\raut{N}^+$ si et seulement si $\RR N$ n'est pas orientable au-dessus de $a$. Dans le second cas, $f_a$ n'est jamais dans $\raut{N}^+$. Dans le dernier cas, $f_a$ est toujours dans $\raut{N}^+$.
\end{Proposition}

\begin{proof}
Lorsque la courbe $a$ est une composante connexe de $\RR\surg$, il existe une base symplectique réelle $(a_i,b_i)_{i=1,\ldots,g}$ de $H_1(\surg,\ZZ)$ telle que $a_1 = a$ (voir Définition 4.1 de \cite{article1}). Le premier cas suit alors de la Proposition 4.3 de \cite{article1} (avec $f_1 = f_a$).

Lorsque la courbe $a$ est globalement stable et $\RR\surg\neq \emptyset$, il existe une base symplectique réelle $(a_i,b_i)_{i=1,\ldots,g}$ de $H_1(\surg,\ZZ)$ telle que $a_k = a$. Ce cas découle à nouveau de la Proposition 4.3 de \cite{article1}. Si la partie réelle de $\surg$ est vide, on peut construire une base symplectique $(a_i,b_i)_{i=1,\ldots,g}$ de $H_1(\surg,\ZZ)$ telle que $a_1 = a$, et $c_\sur(a_i) = a_i$, $c_\sur(b_i) = -b_i +a_i +a_0$, $i=1,\ldots g$, où $a_0 = -a_1-\ldots-a_g$ est une courbe globalement stable fixée. On peut alors démontrer que la Proposition 4.3 de \cite{article1} est encore vraie dans ce cas en utilisant exactement la même méthode, ce qui conclut ce cas.

Lorsque $a\cap c_{\sur}(a) = \emptyset$, on distingue trois sous cas. Tout d'abord, si la courbe est séparante, c'est-à-dire si $\surg\setminus\RR\surg$ n'est pas connexe, il existe une base symplectique réelle $(a_i,b_i)_{i=1,\ldots,g}$ de $H_1(\surg,\ZZ)$ telle que $a_k = a$. La Proposition 4.3 de \cite{article1} nous donne alors le résultat (car $f_a = f_k$). Puis, si la courbe n'est pas séparante, mais que $\RR\surg \neq\emptyset$, on fixe une base symplectique réelle $(a_i,b_i)_{i=1\ldots,g}$ de $H_1(\surg,\ZZ)$. On peut alors vérifier en utilisant le Lemme 4.4 de \cite{article1} que $f_a$ est homotope à un produit d'un nombre pair d'isomorphismes du type $f_{a_i}$, donc est toujours dans $\raut{N}^+$. Lorsque $\RR\surg = \emptyset$, on peut à nouveau montrer que $f_a$ est homotope à un produit d'un nombre pair d'automorphismes de $(N,c_N)$ renversant les orientations de $\Det$, donc que $f_a$ appartient à $\raut{N}^+$.
\end{proof}

\begin{Definition}\label{defd}
Soit $(N,c_N)$ un fibré en droites complexes muni d'une structure réelle au-dessus de $(\surg,c_\sur)$ et soit $\ud$ un quadruplet adapté à $(N,c_N)$. Si la courbe est séparante, c'est-à-dire si $\surg\setminus\RR\surg$ n'est pas connexe, et si aucune des composantes de $\RR N$ n'est orientable, alors on note $\DD_{\ud}(N)$ le fibré $\ZZ/2\ZZ$-principal trivial au-dessus de $\RR J(\surg)\times\surg^{\ud}$. Dans tous les autres cas, $\raut{N}^+$ est d'indice $2$ dans $\raut{N}$ et on considère le fibré $\ZZ/2\ZZ$-principal $\DD_{\ud}(N) = \ropcp / \raut{N}^+$ au-dessus de $\RR J(\surg)\times\surg^{\ud}$.

Si $D$ est un diviseur compatible avec $(N,c_N)$ tel que $D = D_{\ux}$ pour un certain $\ux\in\surg^{\ud}$, alors la restriction du fibré $\DD_{\ud}(N)$ à $\RR J(\surg)\times\{\ux\}$ ne dépend pas du représentant $\ux\in\surg^{\ud}$ de $D$ choisi. On note $\DD_D(N)$ le fibré $\ZZ/2\ZZ$-principal au-dessus de $\RR J(\surg)$ ainsi obtenu.
\end{Definition}

\begin{Notation}\label{notjet}
Soit $D$ un diviseur réel sur $(\surg,c_\sur)$. Pour un point $x\in\surg$, on notera $mult_D(x)\in\ZZ$ la multiplicité du point $x$ dans le diviseur $D$. Si $D$ est effectif, on note $\RR \jet_D$ le fibré vectoriel réel au-dessus de $\RR J(\surg)$ dont la fibre est donnée par
\[
\dd\bigoplus_{x\in D_{|\RR\surg}}\left(\bigoplus_{m=0}^{mult_D(x)-1}(T^*_x\RR\surg)^{\otimes m}\right) \oplus\bigoplus_{\{z,\overline{z}\}\in D} \left(\bigoplus_{m=0}^{mult_D(z)-1}\RR \bigoplus_{y\in\{z,\overline{z}\}}(T^*_y\surg)^{\otimes m}\right).
\]
Le $\RR$ dans la deuxième somme dénote toujours l'ensemble des points fixes de l'involution induite par $c_\sur$. 

Dans le cas général, on décompose $D$ de façon unique en $D^+-D^-$, où $D^+$ et $D^-$ sont effectifs, de supports inclus dans celui de $D$ et dont les points sont de multiplicités minimales, et on pose
\[
\RR \jet_D = \RR\jet_{D^+} \oplus (\RR\jet_{D^-})^*.
\]
\end{Notation}

Notons $\SSS_{2,s^{\pm}}$ les sous-groupes des groupes de permutations $\SSS_{s^{\pm}}$ engendrés par les transpositions $(2i-1\  2i)$, $1\leq i\leq s^{\pm}$, et par les permutations de la forme $2i-1\mapsto 2\sigma(i)-1$ et $2i\mapsto 2\sigma(i)$, pour tout $\sigma\in\SSS_{s^{\pm}}$. Par définition, le groupe produit $\SSS_{\ud} = \SSS_{r^+}\times\SSS_{r^-}\times \SSS_{2,s^+}\times\SSS_{2,s^-}$ agit naturellement et librement sur $\surg^{\ud}$. On notera $\surg^{(\ud)}$ le quotient $\surg^{\ud}/\SSS_{\ud}$. 

\begin{Definition}\label{defjet}
Soit $\ud \in\NN^4$. On note $\RR \jet_{(\ud)}$ le fibré vectoriel réel au-dessus de $\RR J(\surg)\times \surg^{(\ud)}$ dont la fibre au-dessus d'un point $(J,(\ux))\in\RR J(\surg)\times \surg^{(\ud)}$ est $\RR\jet_{D_{\ux}}$.  
\end{Definition}

 D'autre part, l'action de $\SSS_{\ud}$ sur $\surg^{\ud}$ se relève naturellement en une action sur $\DD_{\ud}(N)$. Le quotient $\DD_{(\ud)}(N)$ est un fibré $\ZZ/2\ZZ$-principal au-dessus de $\RR J(\surg)\times\surg^{(\ud)}$.

\begin{Rem}
Le fibré $\RR\jet_{(\ud)}$ est le quotient d'un fibré $\RR\jet_{\ud}$ par le relevé trivial de l'action de $\SSS_{\ud}$ sur $\RR J(\surg)\times\surg^{\ud}$. Autrement dit, l'action de $\SSS_{\ud}$ sur $\RR \jet_{\ud}$ est donnée par $(\sigma,(\ux,v))\in\SSS_{\ud}\times \RR\jet_{\ud}\mapsto (\sigma.\ux,v) \in \RR\jet_{\ud}$. Le fibré $\RR\jet_{(\ud)}$ n'est pas orientable en général.
\end{Rem}

\begin{Proposition}\label{dtrivial}
 Le fibré $\DD_{\ud}(N)$ est trivialisable au-dessus de $\RR J(\surg)\times\surg^{\ud}$. De plus $\riso{N}$ agit par automorphismes de fibré principal sur $\DD_{\ud}(N)$.

Le fibré $\DD_{(\ud)}(N)\otimes \det(\RR\jet_{(\ud)})$ est trivialisable au-dessus de $\RR J(\surg)\times\surg^{(\ud)}$.
\end{Proposition}

\begin{proof}
Commençons par la première partie de la Proposition. Fixons tout d'abord $J\in\RR J(\surg)$. Prenons un lacet $(\ux_t)_{t\in [0,1]}$ dans $\surg^{\ud}$. Il nous suffit de vérifier la trivialité de $\DD_{\ud}(N)$ au-dessus des lacets de la forme $\left(\{J\}\times (\ux_t)\right)_{t\in[0,1]}$ car $\RR J(\surg)$ est contractile. De plus, comme $\surg^{\ud}$ est un produit, il nous suffit de considérer deux cas différents : soit tous les points de $(\ux_t)_{t\in [0,1]}$ sont fixes sauf ceux qui sont sur une composante de la partie réelle de $(\surg,c_\sur)$, soit seuls les points complexes conjugués de $(\ux_t)_{t\in [0,1]}$ bougent.

Prenons le premier cas. Supposons pour simplifier les notations que tous les points de $(\ux_t)_{t\in [0,1]}$ se trouvent sur une composante particulière de $\RR\surg$. Choisissons une orientation de cette composante et numérotons ces points $x^t_1,\ldots,x^t_r$ dans un ordre cyclique donné par cette orientation. Prenons des lacets de coordonnées $J$-holomorphes $\xi_i^t : U_i^t \rightarrow D$, $t\in [0,1]$, $\ZZ/2\ZZ$-équivariantes et centrées en $x_i^t$. Quitte à perturber $(\ux_t)_{t\in [0,1]}$, on peut supposer que les $U_i^t$ ne s'intersectent pas à un $t$ fixé. Pour tout $t\in [0,1]$, notons $(N_t,\DB_t,c_{N_t})$ les fibrés holomorphes obtenus par recollement de $\surg\setminus \ux_t \times \TCC$ et $U_i^t\times\TCC$ grâce aux applications
\[
\begin{array}{c c c}
(\surg\setminus\ux_t)\cap U_i^t\times \TCC & \rightarrow & U_i^t\cap (\surg\setminus\ux_t)\times \TCC\\
(\xi_i^t,v) & \mapsto & (\xi_i^t,(\xi_i^t)^{m_i}v),
\end{array}
\]
où $m_i\in\{-1,1\}$ est donné par le signe de $x_i^t$ dans $D_{\ux_t}$. De plus, les opérateurs $\DB_t$ viennent avec des sections méromorphes réelles évidentes. 

Fixons un opérateur $\DB\in\ropj$ associé à $\ux_0$ ainsi qu'une polarisation de $\DB$ et un isomorphisme entre $(N,\DB,c_N)$ et $(N_0,\DB_0,c_{N_0})$ préservant les polarisations. Si on trivialise la famille de fibrés en droites complexes $(N_t,c_{N_t})_{t\in [0,1]}$, on obtient une famille continue d'opérateurs $(\DB_t')_{t\in [0,1]}$ sur $N$ associés à $(D_{\ux_t})_{t\in [0,1]}$ et polarisés, définie à homotopie près. Il nous faut alors comparer $\DB_0' = \DB$ à $\DB_1'$. Autrement dit, il nous faut calculer l'automorphisme de $(N_0,c_{N_0})$ induit par la trivialisation de $(N_t,c_{N_t})$. C'est ce que nous faisons maintenant. 

Pour tout $t\in [0,1]$, un isomorphisme entre $(N_0,c_{N_0})$ et $(N_t,c_{N_t})$ est décrit dans les trivialisations par les applications suivantes
\[
\begin{array}{c c c}
(N_0)_{|U_i^0\cap U_j^t} & \rightarrow & (N_t)_{U_j^t}\\
(z,v) & \mapsto & (z,\eta_{i,j}^t(z)v)
\end{array}
\]
\[
\begin{array}{c c c}
(N_0)_{|U_i^0\setminus \{x_j^t\}} & \rightarrow & (N_t)_{\surg\setminus \{x_j^t\}}\\
(z,v) & \mapsto & (z,\psi_{i,j}^t(z)v)  
\end{array}
\]
\[
\begin{array}{c c c}
(N_0)_{|(\surg\setminus \{x_i^0\})\cap U_j^t} & \rightarrow & (N_t)_{U_j^t}\\
(z,v) & \mapsto & (z,\theta_{i,j}^t(z)v)
\end{array}
\]
\[
\begin{array}{c c c}
(N_0)_{|\surg\setminus\{\ux_0,\ux_t\}} & \rightarrow & (N_t)_{\surg\setminus\ux_t}\\
(z,v) & \mapsto & (z,\phi^t(z) v)
\end{array}
\]
Il nous faut définir $\phi^t(z)$, $\eta_{i,j}^t(z)$, $\psi_{i,j}^t(z)$ et $\theta_{i,j}^t(z)$. Nous voulons de plus que pour $t = 0$ on ait $\eta_{i,i}^0(z) = 1$, $\psi_{i,i}^0(z) = (\xi_i^0)^{-m_i}(z)$, $\theta_{i,i}^0(z) = (\xi_i^0)^{m_i}(z)$ et $\phi^0(z) = 1$. Supposons pour simplifier la démonstration que les disques $U_i^t$ sont disjoints à chaque temps $t$ et que un seul disque $U_i^t$ à la fois peut intersecter $U_j^0$ pout chaque $t$. Notons $n_i^t\in \NN$ le nombre de disques $U_j^t$ disjoints de $U_i^0$ au temps $t$ l'ayant touché à un temps précédent. Ainsi, on a $n_i^0 = 0$ et $n_i^1 = r$. On pose alors $\eta_{i,j}^t(z) = (-1)^{n_i^t}$. On définit aussi $\psi_{i,j}^t(z) = (-1)^{n_{i}^t}(\xi_j^t)^{-m_j}(z)$ sur $(U_i^0\setminus \{x_j^t\})\cap U_j^t$, prolongée sur tout $U_i^0\setminus \{x_j^t\}$ de sorte que lorsque $U_i^0\cap U_j^t = \emptyset$, $\psi_{i,j}^t (z) = (-1)^{n_{i}^t}$. Puis $\theta_{i,j}^t(z) = (-1)^{n_i^t}(\xi_i^0)^{m_i}(z)$ sur $(U_i^0\setminus \{x_i^0\})\cap U_j^t$, prolongée sur tout $(\surg\setminus \{x_i^0\})\cap U_j^t$ de sorte que lorsque $U_i^0\cap U_j^t = \emptyset$, $\theta_{i,j}^t(z) = (-1)^{n_i^t-1}$. Enfin, on pose $\phi^t(z) = \psi_{i,j}^t(z)(\xi_i^0)^{m_i}(z)$ sur $U_i^0\setminus \{x_j^t\}$, $\phi^t(z) = \theta_{i,j}^t(z)(\xi_j^t)^{-m_j}(z)$ sur $U_j^t\setminus \{x_i^0\}$, prolongée par $1$ en dehors d'un voisinage tubulaire de $\RR\surg$ contenant tous les disques $U_i^t$. Ce dernier prolongement est possible car l'indice de $\phi^t(z)$ le long des deux lacets formés de $\RR\surg$ où l'on a remplacé les intersections $\RR\surg\cap U_i^0$ et $\RR\surg\cap U_i^t$ par les bords supérieurs, respectivement inférieurs, des $U_i^0$ et $U_i^t$ est nul pour tout temps $t$ (voir Figures \ref{dfig1} et \ref{dfig2}).

\begin{figure}[h]
  \centering
  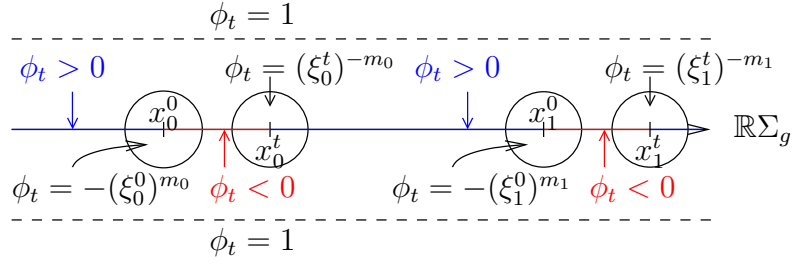
  \caption{Construction de $\phi_t$ lorsque $n_i^t = 1$}
  \label{dfig1}
\end{figure}
\begin{figure}[h]
  \centering
  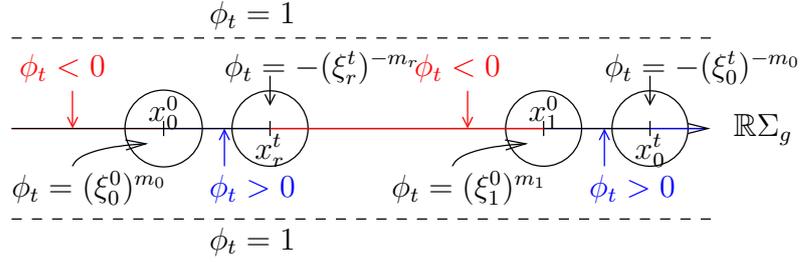
  \caption{Construction de $\phi_t$ lorsque $n_i^t = 2$}
  \label{dfig2}
\end{figure}

Par construction, ces fonctions sont compatibles avec les changements de cartes et fournissent bien une trivialisation de $(N_t,c_{N_t})_{t\in [0,1]}$. Lorsque $t=1$, on obtient un automorphisme de $(N_0,c_{N_0})$ donné par $\phi_1$. Cette fonction vaut $1$ en dehors d'un voisinage de la composante de $\RR\surg$ considérée, son signe sur celle-ci est égal à $(-1)^r$ et son indice transversalement à celle-ci vaut $r \mod 2$. Ainsi, si $\RR N$ est orientable sur cette composante, $\phi_1$ y est positif et son indice transversalement à cette composante est pair. Donc $\phi_1\in\raut{N}^+$ d'après la Proposition \ref{rapp}.

Supposons maintenant que ce sont des points complexes conjugués de $(\ux_t)_{t\in [0,1]}$ qui bougent. On peut supposer tout d'abord qu'une seule de ces paires est mobile et subdiviser à nouveau en deux cas : soit l'image de ces points forme un lacet sur $\surg$, globalement stable par $c_\sur$, soit elle est formée de deux composantes complexes conjuguées. En utilisant la même technique que dans le cas précédent et la Proposition \ref{rapp}, on montre le résultat voulu dans ce cas ci.

Passons maintenant à l'action de $\SSS_{\ud}$. Remarquons tout d'abord que la restriction à $\{J\}\times\surg^{(\ud)}$ du fibré $\det(\RR \jet_{(\ud)})$ est isomorphe en tant que fibré $\ZZ/2\ZZ$-principal à $\surg^{\ud}/\SSS_{\ud}^+\rightarrow \surg^{(\ud)}$, où $\SSS_{\ud}^+$ est le noyau de la signature sur $\SSS_{\ud}$. Or, la démonstration du premier point montre que pour tout $\sigma\in\SSS_{\ud}$ et $\ux\in\surg^{\ud}$ tels que $\ux$ et $\sigma.\ux$ sont dans la même composante connexe de $\surg^{\ud}$, l'automorphisme de $\DD_{\ux}(N)$ induit par la trivialisation de $\DD_{\ud}(N)$ au-dessus d'un chemin joignant $\ux$ à $\sigma.\ux$ ne préserve les orientations que lorsque la signature de $\sigma$ vaut $1$. Ce qui montre le deuxième point.
\end{proof}

\begin{Rem}
  Notons que la démonstration de la Proposition \ref{dtrivial} montre que le fibré $\ropcp$ n'est en général pas trivial. En effet, prenons $(\ux_t)_{t\in [0,1]}$ un lacet dans $\surg^{\ud}$ tel que seuls les points sur une courbe simple orientée $a\subset \surg$ bougent, chacun d'entre eux réalisant exactement un tour de $a$. Si l'on note $r$ le nombre des points mobiles de $(\ux_t)$, alors, en reprenant les notations de la Proposition \ref{rapp}, la monodromie de $\ropcp$ associée à ce lacet est homotope à $(f_a)^r$ dans le cas où $a$ est une composante de $\RR\surg$ ou bien globalement stable, et par $(f_a)^{\frac{r}{2}}$ si $a\cap c_\sur(a) = \emptyset$.

  Insistons aussi sur le fait que si $(\ux_t)_{t\in [0,1]}$ est un chemin de points de $\surg^{\ud}$ où l'on ne fait qu'échanger deux points complexes conjugués, on a $D_{\ux_0} = D_{\ux_1}$, mais la trivialisation de $\DD_{\ud}(N)$ le long de $(\ux_t)_{t\in [0,1]}$ ne donne pas l'isomorphisme trivial $\DD_{\ux_0}(N) = \DD_{D_{\ux_0}}(N) = \DD_{D_{\ux_1}}(N) = \DD_{\ux_1}(N)$.
\end{Rem}

\subsubsection{Action d'un automorphisme général : \'Enoncé}

Rappelons que l'on note $\PV = \dd\bigotimes_{i=1}^kPin^+((\RR N)_i)$ si $\rang(N)\geq 2$. Lorsque $(N,c_N)$ est un fibré en droites complexes, on adopte la convention que $\PV = \RR$. Nous avons défini au \S \ref{diviseur} plusieurs fibrés que nous allons utiliser maintenant pour décrire l'action d'un élément de $\riso{N}$ sur les orientations de $\Det$. En particulier, si $\ud\in\NN^4$ est compatible avec $(N,c_N)$, nous avons défini un fibré $\ZZ/2\ZZ$-principal $\DD_{(\ud)}(N)$ et un fibré vectoriel réel $\RR\jet_{(\ud)}$ au-dessus de $\RR J(\surg)\times \surg^{(\ud)}$ (voir Définitions \ref{defd} et \ref{defjet}), et si $D$ est un diviseur compatible avec $(N,c_N)$ nous avons défini un fibré $\ZZ/2\ZZ$-principal $\DD_{D}(N)$ et un fibré vectoriel réel $\RR\jet_{D}$ au-dessus de $\RR J(\surg)$ (voir Définition \ref{defd} et Notation \ref{notjet}). Nous introduisons un dernier fibré $T_{(\ud)}$ qui est le fibré tautologique sur $\surg^{(\ud)}$ de fibre $\dd\bigotimes_{i=1}^{r^+} T_{x^+_i}^*\RR\surg$ au-dessus du point $\ux = (x^+_1,\ldots,x^+_{r^+},x^-_{1},\ldots,x^-_{r^-},\uz^+,\uz^-)$.

Tous les isomorphismes entre fibrés en droites réelles que nous considérons sont définis à homotopie près, et nous ne ferons pas de distinction entre fibré en droites réelles et fibré $\ZZ/2\ZZ$-principal associé. Nous omettons de le préciser dans la suite pour soulager les énoncés.

\begin{Theoreme}\label{totalaction}
Soit $(N,c_N)$ un fibré vectoriel complexe sur $(\surg,c_\sur)$.
\begin{enumerate}
\item Soit $D$ un diviseur associé à $(\det(N),c_{\det(N)})$. Le fibré $\Det$ au-dessus de $\rop \xrightarrow{\pi} \RR J(\surg)$ est canoniquement isomorphe à
\[
\PV \otimes \pi^*\DD_D(N) \otimes \det(H^1(\surg,\RR)_{-1})^{\otimes\rang(N)}\otimes\det(\RR\jet_{D})\dd\bigotimes_{x\in D_{|\RR\surg}} (T_x^*\RR\surg)^{\otimes\max(0,mult_D(x))}.
\]
\item Soit $\ud$ un quadruplet compatible avec $(N,c_N)$. Le fibré $\Det$ au-dessus de $pr_1^* \rop \xrightarrow{\pi} \RR J(\surg)\times \surg^{(\ud)}$ est canoniquement isomorphe à
\[
\PV \otimes \pi^*\left(\DD_{(\ud)}(N) \otimes\det(\RR\jet_{(\ud)})\right)\otimes (pr_2\circ\pi)^*T_{(\ud)}\otimes \det(H^1(\surg,\RR)_{-1})^{\otimes\rang(N)}.
\]
\end{enumerate}
\end{Theoreme}

\begin{Rem}
Précisons ce que nous entendons par canonique ici : si $(\Phi,\varphi):(N,c_N)\rightarrow (N',c_{N'})$ est un isomorphisme entre deux fibrés vectoriels complexes sur $(\surg,c_\sur)$, alors le diagramme suivant doit commuter
\[
\xymatrix{\Det \ar[r] \ar[d]_{\Phi} & \parbox{4in}{\raggedright$\PV \otimes \pi^*\left(\DD_{(\ud)}(N)\otimes \det(\RR\jet_{(\ud)})\right)\otimes(pr_2\circ\pi)^*T_{(\ud)}  \linebreak  \otimes \det(H^1(\surg,\RR)_{-1})^{\otimes\rang(N)} \ar[d]^{\Phi\otimes\Phi\otimes \varphi^* \otimes \d\varphi}$}\\
\Det[N'] \ar[r] & \parbox{4in}{\raggedright$\PV[N'] \otimes (\pi')^*\left(\DD_{(\ud)}(N')\otimes\det(\RR\jet_{(\ud)})\right) \otimes(pr_2'\circ\pi')^*T_{(\ud)} \linebreak \otimes \det(H^1(\surg,\RR)_{-1})^{\otimes\rang(N)}.$}
}
\]
\end{Rem}

Le Théorème \ref{totalaction} permet donc de décomposer plus finement l'action d'un automorphisme sur les orientations du fibré déterminant. Tous les termes obtenus font intervenir des objets topologiques simples, sauf le fibré $\DD_{(\ud)}(N)$. Nous n'avons pas de description simple de ce dernier fibré dans le cas général. Nous l'expliciterons dans deux cas particuliers au \S \ref{corpart}.

Notons $r^+_{min}$ le nombre de composantes connexes de $\RR N$ qui ne sont pas orientables. On appellera le quadruplet $\ud_{min} = \left(r^+_{min},0,\max(0,\deg(N)-r^+_{min}),-min(0,\deg(N)-r^+_{min})\right)$ minimal. L'ensemble $\surg^{(\ud_{min})}$ est formé des diviseurs ayant exactement un point de multiplicité $1$ sur chaque composante connexe de $\RR\surg$ au-dessus de laquelle $\RR N$ n'est pas orientable. Dans ce cas, les fibrés $\DD_{\ud_{min}}(N)$ et $\RR\jet_{\ud_{min}}$ sont orientables et fixer une orientation au-dessus d'une composante connexe de $\surg^{\ud_{min}}$ en fixe une pour toutes les autres. Ainsi, pour tout $(\Phi,\varphi)\in\riso{N}$, on peut considérer l'action de $\Phi$ sur les deux orientations de $\DD_{\ud_{min}}(N)$ que l'on notera $\varepsilon(\Phi_{\DD_{min}})$. L'action de $(\Phi,\varphi)$ sur celles de $\RR\jet_{\ud_{min}}$ est notée $\varepsilon(\varphi_{\ud_{min}})$. Rappelons que comme dans le \S \ref{parpin} nous notons $\sigma_{\varphi}^-$ la permutation induite par $\varphi$ sur les orientations des composantes de $\RR\surg$ au-dessus desquelles $\RR N$ n'est pas orientable. Rappelons aussi que nous avons noté $\varepsilon(\Phi_{\PPP^+})$ le signe de l'action de $(\Phi,\varphi)$ sur $\PV$ et que $\varphi^*$ est l'automorphisme de $H^1(\surg,\RR)_{-1}$ induit par $\varphi$. Le Corollaire \ref{reponse} est l'application directe du Théorème \ref{totalaction} au cas de $\ud_{min}$.

\begin{Corollaire}\label{reponse}
Soit $(N,c_N)$ un fibré vectoriel complexe sur $(\surg,c_\sur)$ de partie réelle non vide. Soit $(\Phi,\varphi)\in\riso{N}$. Son action sur les orientations du fibré $\Det$ est donné par le produit $\varepsilon(\Phi_{\PPP^+})\varepsilon(\Phi_{\DD_{min}})\varepsilon(\varphi_{\ud_{min}})\varepsilon(\sigma_{\varphi}^-)\det(\varphi^*)^{\rang(N)}$.
\end{Corollaire}


\begin{Rem}
  On peut expliciter un peu plus le terme $\varepsilon(\varphi_{\ud_{min}})$ en termes topologiques. En effet, lorsque la courbe n'est pas séparante, il est égal à la signature de la permutation $\varphi^-_{\RR\sur}$ induite par $\varphi$ sur l'ensemble des composantes connexes de $\RR\surg$ au-dessus desquelles $\RR N$ n'est pas orientable. Lorsque la courbe est séparante, il est égal à cette même signature si et seulement si $\deg(N)-r_{min}^+$ est un multiple de $4$ ou si $\varphi$ préserve les deux composantes connexes de $\surg\setminus\RR\surg$.
\end{Rem}

Avant de passer à la démonstration du Théorème \ref{totalaction} qui occupe le \S \ref{transfopar}, énonçons le Lemme \ref{isomdetd}.

\begin{Lemme}\label{isomdetd}
Soient $(N,c_N)$ et $(N',c_{N'})$ deux fibrés vectoriels complexes sur $(\surg,c_\sur)$ et $D$ un diviseur réel associé à $(\det(N),c_{\det(N)})$ et $(\det(N'),c_{\det(N')})$. Une orientation sur $\Det\otimes \DD_D(N)\otimes \PV$ en induit naturellement une sur $\Det[N']\otimes \DD_D(N')\otimes \PV[N']$.

Lorsque $(N,c_N) = (\TCC,conj)$, il existe un isomorphisme canonique
\[
\Det[\TCC]\otimes\DD_0(\TCC) = \det(\HH^1(\surg,\RR)_{-1}).
\]
\end{Lemme}

\begin{proof}
  La première affirmation est une reformulation de la Proposition \ref{lemrangun} dans le cas d'un automorphisme relevant l'identité. En effet, prenons un isomorphisme $f : (N',c_{N'})\rightarrow (N,c_{N})$. Celui-ci induit naturellement un isomorphisme
  \[
\Det\otimes \DD_D(N)\otimes \PV = \Det[N']\otimes \DD_D(N')\otimes \PV[N'].
\]
Il nous suffit de vérifier que si $g : (N',c_{N'})\rightarrow (N,c_{N})$ est un autre isomorphisme, alors $g\circ f^{-1}\in \raut{N}$ préserve les orientations de $\Det\otimes\DD_D(N)\otimes\PV$. En effet, $g\circ f^{-1}$ relève $\id_{\sur}$ donc d'après la Proposition \ref{lemrangun} et la définition de $\DD_D(N)$, $g\circ f^{-1}$ préserve les orientations de $\Det$ si et seulement si $\det(g\circ f^{-1})\otimes (g\circ f^{-1})$ préserve les orientations de $\DD_D(N)\otimes \PV$.

Le second isomorphisme est une reformulation du Lemme \ref{isopart}. En effet, l'isomorphisme $\Det[\TCC]_{|\RR J(\surg)} = \det(\HH^1(\surg,\RR)_{-1})$ induit un isomorphisme $t : \Det[\TCC] \rightarrow \det(\HH^1(\surg,\RR)_{-1})$. Toutefois celui-ci ne respecte pas l'action de $\raut{\TCC}$. On obtient l'isomorphisme voulu de la façon suivante. \'Etant donnée une section polarisée $\DB^+$ de $\RR\CCCC_0(\TCC)$, il existe une famille $F : \RR J(\surg)\rightarrow \raut{\TCC}$ telle que $F^*\DB^+ = \DB_{can}^+$, où $\DB_{can,J} = \frac{1}{2}(\d + i\circ\d \circ J)$ et est munie de la polarisation triviale. De plus, $F$ est uniquement définie à homotopie près. On pose alors
\[
\begin{array}{c c c}
  \DD_0(\TCC) & \rightarrow & \Det[\TCC]^*\otimes \det(\HH^1(\surg,\RR)_{-1}) \\
   \DB^+  & \mapsto & v^*\otimes t\circ F (v),
\end{array}
\]
pour $v$ une section quelconque de $\Det[\TCC]$. On vérifie que cet isomorphisme est bien défini à homotopie près, et que l'isomorphisme qu'il induit entre $\Det[\TCC]\otimes \DD_0(\TCC)$ et $\det(\HH^1(\surg,\RR)_{-1})$ commute avec l'action de $\riso{N}$ d'un côté et $\RR Diff^+(\surg)$ de l'autre.
\end{proof}

\subsubsection{Transformations élémentaires réelles négatives}\label{transfopar}

Supposons dans un premier temps que $N$ est de rang un. La démonstration du Théorème \ref{totalaction} consiste à se ramener au cas déjà traité du fibré trivial sur la courbe (voir \S \ref{extrivial}) au moyen de transformations élémentaires négatives. Tout ce paragraphe est consacré à l'étude de ces transformations. Fixons pour cela une structure complexe $J\in\RR J(\surg)$. Prenons un opérateur de Cauchy-Riemann réel $\DB\in\ropj$. Notons $\NF$ le faisceau des sections holomorphes de $(N,\DB)$.

Soit $D' = \dd\sum_{x\in\RR\surg} a_x x + \dd\sum_{\{z,\overline{z}\}\in\surg^{(2)}}b_{\{z,\overline{z}\}}(z + \overline{z})$ un diviseur effectif réel. On notera 
\[
\begin{array}{c}
\Delta_{b_{\{z,\overline{z}\}}\{z,\overline{z}\}} = \dd\RR \left(\bigoplus_{m=0}^{b_{\{z,\overline{z}\}}-1}\oplus_{y\in\{z,\overline{z}\}}\left((T_{y}^*\surg)^{\otimes m}\otimes N_{y}\right)\right) \\
\RR N_{a_x x} = \dd \bigoplus_{m=0}^{a_x -1} \left( (T_{x}^*\RR\surg)^{\otimes m}\otimes \RR N_{x} \right)\\ 
\text{et } N_{|D'} = \dd\bigoplus_{x\in D'_{|\RR\surg}} \RR N_{a_x x} \dd\bigoplus_{\{z,\overline{z}\}\in D'-D'_{|\RR\surg}}\Delta_{b_{\{z,\overline{z}\}}\{z,\overline{z}\}}.
\end{array}
\]
Remarquons que pour orienter $N_{|D'}$ il ne suffit pas d'orienter chaque droite apparaissant dans la somme. Il faut aussi prendre en compte l'ordre dans lequel ces droites apparaissent.

\begin{Definition}
  Pour un diviseur effectif réel $D'$ sur $\surg$ dont tous les points sont de multiplicité au plus $1$, la transformation élémentaire réelle négative en $D'$ du faisceau $\NF$ est le faisceau localement libre $\NF_{-D'}$ muni d'une structure réelle défini par la suite exacte
\[
0\rightarrow \NF_{-D'} \rightarrow \NF \xrightarrow{ev_{D'}} N_{|D'} \rightarrow 0.
\]
On définit la transformation élémentaire réelle négative en un diviseur effectif réel quelconque en répétant le procédé ci-dessus. Le faisceau $\NF_{-D'}$ obtenu est de rang un, de degré $\deg(N) - \deg(D')$ et sa partie réelle est de première classe de Stiefel-Whitney $w_1(\RR N) - (D'_{|\RR\surg})^{\pd}$.
\end{Definition}

\`A première vue, une transformation élémentaire négative est une opération sur un fibré holomorphe, c'est-à-dire sur un opérateur de Cauchy-Riemann donné. Toutefois, comme nous voulons étudier l'action d'un automorphisme quelconque de $(N,c_N)$ sur les opérateurs de Cauchy-Riemann réels du fibré, il nous faut pouvoir comparer les transformations élémentaires obtenues à partir de deux opérateurs de Cauchy-Riemann différents. Dans la suite du paragraphe, on montre qu'on peut en fait effectuer cette opération simultanément sur tous les opérateurs de Cauchy-Riemann réels sur $(N,c_N)$ de façon cohérente.

Comme nous l'avons remarqué au \S 4.2.3 de \cite{article1}, après une transformation élémentaire réelle négative en $D'$ du faisceau $\NF$, on retrouve un fibré en droites holomorphe réel $(N_{-D'},\DB_{-D'},c_{N,-D'})$. On a de plus deux inclusions naturelles $\ZZ/2\ZZ$-équivariantes
\[
\begin{array}{c}
  i_{-D'} : L^{k,p}(\surg,N_{-D'})\rightarrow L^{k,p}(\surg,N) \\
  j_{-D'} : L^{k-1,p}(\surg,\Lambda^{0,1}\surg\otimes N_{-D'})\rightarrow L^{k-1,p}(\surg,\Lambda^{0,1}\surg\otimes N) \\
\end{array}
\]
où $k\geq 1$, $p>2$. De plus, les deux inclusions $i_{-D'}$ et $j_{-D'}$ induisent une application
\[
\begin{array}{c}
  t_{N_{-D'},J} : \ropj \rightarrow  \ropj[N_{-D'}]
\end{array}
\]
telle que l'injection $i_{-D'}$ induise pour chaque opérateur $\DB'\in\ropj$ un isomorphisme entre le faisceau des sections holomorphes de $(N_{-D'},t_{-D'}(\DB'))$ et celui des sections holomorphes de $(N,\DB')$ qui s'annulent en $D'$ à l'ordre donné par les multiplicités des points dans $D'$ (voir par exemple \cite{article1} ou la Proposition \ref{transfoauto} plus bas).

Si $(\Phi,\varphi) : (N,c_N)\rightarrow (M,c_M)$ est un isomorphisme entre deux fibrés en droites complexes sur $(\surg,c_\sur)$ et $D'$ un diviseur effectif réel, alors $(\Phi,\varphi)$ induit un isomorphisme entre $L^{k,p}(\surg,N)$ et $L^{k,p}(\surg,M)$ qui n'en induit pas forcément un entre $L^{k,p}(\surg,N_{-D'})$ et $L^{k,p}(\surg,M_{-\varphi_*(D')})$. Nous avons toutefois la Proposition suivante.

\begin{Proposition}\label{transfoauto}
  Soit $(N,c_N)$ un fibré en droites complexes sur $(\surg,c_\sur)$ et $J\in\RR J(\surg)$. Soit $D'$ un diviseur effectif réel et $(N',c_{N'})$ un fibré en droites complexes obtenu après transformation élémentaire réelle négative de $(N,c_N)$ en un diviseur effectif réel $D'$. Alors, il existe une application $t_{N'} : \rop/\raut{N}_0 \rightarrow \rop[N']/\raut{N'}_0$ définie naturellement et prolongeant $t_{N',J}$, où l'indice $0$ indique que l'on prend la composante connexe de l'identité d'un groupe topologique.

Soit $(M,c_M)$ un autre fibré sur $(\surg,c_\sur)$ et $(\Phi,\varphi) : (N,c_N)\rightarrow (M,c_M)$ un isomorphisme. Soit $(M',c_{M'})$ un fibré en droites complexes obtenu après transformation élémentaire réelle négatives de $(M,c_M)$ en $\varphi_*(D')$. Alors $(\Phi,\varphi)$ induit un isomorphisme $(\Phi_{-D'},\varphi) : (N',c_{N'}) \rightarrow (M',c_{M'})$ tel que $(\Phi_{-D'})^*\circ t_{M'} = t_{N'}\circ \Phi^*$. De plus, $(\Phi_{-D'},\varphi)$ est uniquement défini à multiplication par une fonction de $\RR\CCC^{\infty}(\surg,\CC^*)$ homotope à la fonction constante $1$ près. 
\end{Proposition}

La Proposition \ref{transfoauto} nous dit en particulier que nous pouvons effectuer une transformation élémentaire simultanément sur tous les opérateurs de Cauchy-Riemann réels sur $(N,c_N)$ : le résultat de cette opération est une paire formée du fibré $(N',c_{N'})$ et de l'application $t_{N'}$ de telle sorte que tout fibré holomorphe réel obtenu par transformation élémentaire en $D'$ d'un opérateur $\DB \in \rop$ est isomorphe à $(N',t_{N'}(\DB),c_{N'})$. Cette paire n'est pas définie uniquement, mais à isomorphisme canonique près.

\begin{proof}
Le fibré $N'$ est défini comme le recollement de $N_{|\surg\setminus D'}$ et $N_{|U_{D'}}$ , où $U_{D'}$ est une union disjointe de disques ouverts centrés aux points de $D'$ et $\ZZ/2\ZZ$-équivariant, par les applications de recollement
\[
\begin{array}{c c c}
  N_{U_{|a_x x}\cap (\surg\setminus D')} & \rightarrow & N_{( \surg\setminus D')\cap U_{|a_x x}} \\
(\xi_x,v) & \mapsto & (\xi_x , \xi_x^{a_x}v)
\end{array}
\]
où $(U_{a_x x},\xi_x)_{x\in D'}$ sont des cartes $J$-holomorphes $\ZZ/2\ZZ$-équivariantes. On voit donc que si $\DB$ est un élément de $\ropj$, alors ses restrictions à chaque ouvert $\surg\setminus\{D'\}$ et $U_{D'}$ se recollent et donnent l'opérateur $t_{N',J}(\DB)$. 

On remarque ensuite que $t_{N',J}$ induit bien une application notée de la même façon de $\ropj/\raut{N}_0$ dans $\ropj[N']/\raut{N'}_0$ car, pour tout $f\in\raut{N}_0 = \RR \CCCC^{\infty}(\surg,\CC^*)_0 = \raut{N'}_0$, on a $t_{N',J}(f^*\DB) = f^*(t_{N',J}(\DB))$.

Soit $J_1 \in\RR J(\surg)$ et $J_t\in\RR J(\surg)$, $t\in [0,1]$, joignant $J$ à $J_1$. Prenons une famille continue $(U_{a_x x},\xi^t_x)_{x\in D',t\in [0,1]}$ de cartes $J_t$-holomorphes et $\ZZ/2\ZZ$-équivariante avec $\xi_x^0 = \xi_x$. On obtient alors pour chaque $t\in [0,1]$ un fibré en droites complexes $(N'_t,c_{N'_t})$ qui est obtenu par transformation élémentaire réelle négative de $(N,c_N)$ en $D'$. Comme $[0,1]$ est contractile, il existe une unique famille continue d'isomorphismes $f'_t : N' \rightarrow N'_t$, $t\in [0,1]$, à homotopie près trivialisant la famille $(N'_t,c_{N'_t})_{t\in [0,1]}$. On peut de plus supposer que $f'_0 = \id_{N'}$. Posons alors $t_{N',J_1} = (f'_1)^*\circ t_{N'_1,J_1}$. Il nous faut vérifier que cette application ne dépend ni du chemin $(J_t)_{t\in [0,1]}$ choisi ni des coordonnées $(\xi_x^t)_{t\in [0,1]}$ choisies. Prenons une famille continue $(U_{a_x x},\zeta_x^t)_{t\in [0,1]}$ de cartes $J_t$-holomorphes $\ZZ /2\ZZ$-équivariantes avec $\zeta_x^0 = \xi_x$. \`A nouveau on obtient une famille $(N''_t,c_{N''_t})_{t\in [0,1]}$ de fibrés en droites complexes sur $(\surg,c_\sur)$ trivialisée par $f''_t : N' \rightarrow N''_t$, $t\in [0,1]$. De plus, on a une famille d'isomorphisme $F_t$, dépendant de $t\in [0,1]$, entre $N'_t$ et $N''_t$ donnée par l'identité sur $N_{|\surg \setminus D'}$ et par le quotient $\frac{(\zeta_x^t)^{a_x}}{(\xi_x^t)^{a_x}}$ sur $N_{|U_{a_x x}}$. On a de plus $(F_t)^*\circ t_{N''_t,J_t} = t_{N'_t,J_t}$. D'autre part, $F_t\circ f'_t : N' \rightarrow N''_t$, $t\in [0,1]$, fournit une autre trivialisation de $(N''_t,c_{N''_t})$, donc $F_1\circ f'_1$ est égal à $f''_1$ à homotopie près. Ainsi, on a $(f''_1)^*\circ t_{N''_1,J_1} = (F_1\circ f'_1)^* \circ t_{N''_1,J_1} = (f'_1)^*\circ t_{N'_1,J_1}$, ce qui montre l'indépendance de notre définition par rapport au choix des cartes holomorphes. L'indépendance par rapport au choix du chemin $(J_t)_{t\in [0,1]}$ découle du fait que $\RR J(\surg)$ est contractile.

Le fibré $M'$ est défini de façon analogue à $N$ en utilisant des cartes $J_M$-holomorphes $(V_{a_x x},\zeta_x)_{x\in \varphi_*(D')}$ qui sont $\ZZ/2\ZZ$-équivariantes. Prenons $J_t\in\RR J(\surg)$, $t\in [0,1]$, un chemin joignant $J_0 = J$ à $J_1 = \varphi^* J_M$, et $(U^t_{a_x x},\xi^t_x)_{x\in D',t\in [0,1]}$ un chemin de cartes $J_t$-holomorphes et $\ZZ/2\ZZ$-équivariante avec $\xi_x^0 = \xi_x$ et $\xi_x^1 = \pm \zeta_x^1\circ\varphi$. Le signe dans cette dernière égalité apparaît car on souhaite garder des coordonnées qui sont $\ZZ/2\ZZ$-équivariantes. On obtient ainsi une famille $(N'_t,c_{N'_t})$ de fibrés en droites complexes. L'automorphisme $(\Phi,\varphi)$ induit naturellement un isomorphisme entre $(N'_1,c_{N'_1})$ et $(M',c_{M'})$ que l'on compose avec la trivialisation de la famille $(N'_t,c_{N'_t})$ pour obtenir un isomorphisme $(\Phi_{-D'},\varphi)$ de $(N',c_{N'})$ dans $(M',c_{M'})$. Ce dernier n'est défini qu'à multiplication par un élément de $\raut{N'}$ homotope à l'identité près, et ne dépend pas des trivialisations choisies. Le fait que $(\Phi_{-D'})^*\circ t_{M'} = t_{N'}\circ \Phi^*$ découle directement des définitions.
\end{proof}

Dans les Lemmes \ref{transfodet} et \ref{transfod}, on définit des isomorphismes entre des objets associés à un fibré $(N,c_N)$ (resp. $(M,c_M)$) et des objets associés à un fibré $(N',c_{N'})$ (resp. $(M',c_{M'})$) obtenu par transformation élémentaire en un diviseur effectif réel $D'$. La naturalité de ces isomorphismes signifie que pour tout isomorphisme $(\Phi,\varphi): (N,c_N)\rightarrow (M,c_M)$, ceux-ci commutent avec l'action de $\Phi$ à droite et celle de $\Phi_{-D'}$ à gauche.

Le Lemme \ref{transfodet} décrit l'effet d'une transformation élémentaire réelle négative sur le fibré déterminant. Pour la démonstration, on renvoie à \cite{article1}, Lemme 4.16, ou à \cite{shev}.


\begin{Lemme}\label{transfodet}
Soit $(N,c_N)$ un fibré en droites complexes sur $(\surg,c_\sur)$. Soit $D'$ un diviseur effectif réel et $(N',c_{N'})$ un fibré en droites complexes obtenu après transformation élémentaire de $(N,c_N)$ en $D'$. Les inclusions $i_{-D'}$ et $j_{-D'}$ induisent un isomorphisme canonique entre $\Det$ et $((t_{N'})^*\Det[N'])\otimes N_{|D'}$.\qed
\end{Lemme}

Le Lemme \ref{transfod} décrit l'effet d'une transformation élémentaire réelle négative sur le fibré $\DD_D(N)$.

\begin{Lemme}\label{transfod}
  Soit $(N,c_N)$ un fibré en droites complexes sur $(\surg,c_\sur)$ et $D$ un diviseur compatible avec ce fibré. Soit $D'$ un diviseur effectif réel et $(N',c_{N'})$ un fibré en droites complexes obtenu après transformation élémentaire de $(N,c_N)$ en $D'$. Il existe un isomorphisme canonique à homotopie près entre $\DD_D(N)$ et $\dd \DD_{D-D'}(N')\bigotimes_{x\in D'_{|\RR\surg}} \left((T_x^*\RR\surg)^{\otimes mult_D(x)}\otimes \RR N_x\right)^{\otimes mult_{D'}(x)}$ au-dessus de $\RR J(\surg)$.
\end{Lemme}


\begin{proof}
Fixons une section $\DB^+_D$ associée à $D$. Elle induit une section de 
\[
\dd \DD_{D-D'}(N')\bigotimes_{x\in D'_{|\RR\surg}} \left((T_x^*\RR\surg)^{\otimes mult_D(x)}\otimes \RR N_x\right)^{\otimes mult_{D'}(x)}
\]
 donnée par 
\[
\dd [t_{N'}(\DB^+_{D})]\bigotimes_{x\in D'_{|\RR\surg}} (\d_x^{mult_D(x)}\sigma)^{\otimes mult_{D'}(x)},
\]
 où $\sigma$ est une section méromorphe de $(N,\DB^+_{D,J})$ pour un $J$ quelconque, de diviseur $D$ et donnant la polarisation de $\DB^+_D$. Cette section ne dépend pas du choix de $\sigma$. Supposons que $\DB^{+,'}_D$ est une autre section polarisée associée à $D$ qui induit la même section de $\DD_D$ que $\DB^+_D$. Il nous faut démontrer que $\DB^{+,'}_{D-D'}$ induit la même section de $\dd \DD_{D-D'}(N')\bigotimes_{x\in D'_{|\RR\surg}} \left((T_x^*\RR\surg)^{\otimes mult_D(x)}\otimes \RR N_x\right)^{\otimes mult_{D'}(x)}$ que $\DB^+_{D-D'}$ pour obtenir l'isomorphisme voulu. D'après le Lemme \ref{sectionpol}, il existe une famille continue $F_D: \RR J(\surg)\rightarrow \raut{N}^+$ telle que $\DB^{+,'}_{D} = F^*(\DB^+_D)$. D'après la Proposition \ref{transfoauto}, après transformation élémentaire, on obtient une famille continue $F_{D-D'}:\RR J(\surg)\rightarrow \raut{N'}$ égale à $F_D$ sous l'isomorphisme $\raut{N} = \RR \CCCC^\infty(\surg,\CC^*) = \raut{N'}$. Remarquons qu'il y a des cas où $F_{D-D'}$ n'est pas à valeurs dans $\raut{N'}^+$, auxquels cas $[t_{N'}(\DB^{+,'}_{D})] = -[t_{N'}(\DB^{+}_{D})]$, et une section de $\DD_D(N)$ n'induit pas une section de $\DD_{D-D'}(N')$.

 D'après la Proposition \ref{rapp} (voir aussi la Proposition 4.3 de \cite{article1}), le signe de l'action de $F_{D-D'}$ sur les orientations de $\Det[N']$ se décompose en deux parties : l'une ne dépend pas de $N'$, l'autre compte le nombre de composantes orientables de $\RR N'$ dont $F_{D-D'}$ échange les orientations. Ainsi $F_{D-D'}$ est à valeurs dans $\raut{N'}^+$ si et seulement si la parité de ce dernier nombre ne change pas en faisant la transformation élémentaire. Calculons donc la différence de parité lors du passage de $N$ à $N'$. Fixons une composante $(\RR\surg)_i$ de la partie réelle de $\surg$. Si $D'$ a un nombre pair de points comptés avec multiplicités sur $(\RR\surg)_i$, alors $(\RR N)_i$ et $(\RR N')_i$ ont même orientabilité et cette composante ne contribue pas au signe recherché. Notons de plus que dans ce cas $F_{D-D'}$ agit toujours trivialement sur $\bigotimes_{x\in D'_{|(\RR\surg)_i}} \left((T_x^*\RR\surg)^{\otimes mult_D(x)}\otimes \RR N_x\right)^{\otimes mult_{D'}(x)}$. Supposons maintenant que $D'$ a un nombre impair de points sur $(\RR\surg)_i$. Les fibrés $(\RR N)_i$ et $(\RR N')_i$ n'ont plus la même orientabilité. Deux cas se présentent :
\begin{enumerate}
\item soit $F_{D-D'}$ est positif sur $(\RR \surg)_i$ et cette composante ne contribue toujours pas au signe recherché,
\item soit $F_{D-D'}$ est négatif sur $(\RR \surg)_i$ et cette composante change le signe recherché.
\end{enumerate}
De plus, dans le premier cas $F_{D-D'}$ agit trivialement sur $\bigotimes_{x\in D'_{|(\RR\surg)_i}} \left((T_x^*\RR \surg)^{\otimes mult_D(x)}\otimes \RR N_x\right)^{\otimes mult_{D'}(x)}$ mais pas dans le second. Ainsi, lorsque $[\DB^{+}_{D-D'}] = [\DB^{+,'}_{D-D'}]$, on a 
\[
\bigotimes_{x\in D'_{|\RR\surg}} (\d_x^{mult_D(x)}\sigma)^{\otimes mult_{D'}(x)} = \bigotimes_{x\in D'_{|\RR\surg}} (\d_x^{mult_D(x)}\sigma')^{\otimes mult_{D'}(x)}
\]
et  lorsque $[\DB^{+}_{D-D'}] = - [\DB^{+,'}_{D-D'}]$, on a 
\[
\bigotimes_{x\in D'_{|\RR\surg}} (\d_x^{mult_D(x)}\sigma)^{\otimes mult_{D'}(x)} = - \bigotimes_{x\in D'_{|\RR\surg}} (\d_x^{mult_D(x)}\sigma')^{\otimes mult_{D'}(x)}
\]
 La section de $\dd \DD_{D-D'}(N')\bigotimes_{x\in D'_{|\RR\surg}} \left((T_x^*\RR\surg)^{\otimes mult_D(x)}\otimes \RR N_x\right)^{\otimes mult_{D'}(x)}$ obtenue ne dépend donc que de la section $[\DB^+_D]$ de $\DD_D(N)$.

On vérifie de la même façon que l'application construite est bien un isomorphisme. La naturalité de celui-ci est immédiate.
\end{proof}

Nous allons surtout utiliser la Proposition \ref{transfoauto} et les Lemmes \ref{transfodet} et \ref{transfod} de la façon suivante. Si $(N,c_N)$ un fibré en droites complexes sur $(\surg,c_\sur)$, $D$ un diviseur compatible avec ce fibré, $D'$ un diviseur effectif réel et $(N',c_{N'})$ et $(N'',c_{N''})$ deux fibrés en droites complexes obtenus après transformation élémentaire de $(N,c_N)$ en $D'$, alors $\id_N\in\riso{N}$ induit naturellement un isomorphisme $F$ entre $(N',c_{N'})$ et $(N'',c_{N''})$, défini de façon unique à homotopie près et tel que les trois diagrammes suivant commutent
\[
\xymatrix{
  & \rop[N']/\raut{N'}_0 \\
\rop/\raut{N}_0 \ar[rd]_{t_{N''}}\ar[ru]^{t_{N'}} & \\
 & \rop[N'']/\raut{N''}_0\ar[uu]_{F^*}
}
\xymatrix{
  & ((t_{N'})^*\Det[N'])\otimes N_{|D'} \\
\Det \ar@{=}[rd]\ar@{=}[ru] & \\
  & ((t_{N''})^*\Det[N''])\otimes N_{|D'} \ar[uu]_{F^*\otimes \id_N} \\
}
\]
\[
\xymatrix{
  & \dd \DD_{D-D'}(N')\bigotimes_{x\in D'_{|\RR\surg}} \left((T_x^*\RR\surg)^{\otimes mult_D(x)}\otimes \RR N_x\right)^{\otimes mult_{D'}(x)} \\
\DD_D(N) \ar@{=}[rd]\ar@{=}[ru] & \\
 & \dd \DD_{D-D'}(N'')\bigotimes_{x\in D'_{|\RR\surg}} \left((T_x^*\RR\surg)^{\otimes mult_D(x)}\otimes \RR N_x\right)^{\otimes mult_{D'}(x)}\ar[uu]_{F^*\otimes\id_N}
}
\]

Ainsi, le choix du fibré complexe obtenu après transformation élémentaire n'interviendra pas dans la démonstration du Théorème \ref{totalaction}.

\subsubsection{Action d'un automorphisme général : Démonstration}

Nous passons maintenant à la démonstration du Théorème \ref{totalaction} proprement dite. Nous commençons par le cas d'un fibré en droites complexes, puis nous montrons comment s'y ramener à partir du cas général.

\begin{proof}[Démonstration du Théorème \ref{totalaction}]

Soit $(N,c_N)$ un fibré vectoriel complexe muni d'une structure réelle sur $(\surg,c_\sur)$. Commençons par fixer un diviseur $D$ compatible avec $(\det(N),c_{\det(N)})$. Afin d'alléger les notations, nous sous-entendrons les tirés-en-arrière des fibrés et considérerons qu'ils sont tous définis au-dessus de $\rop$.

{\bf Premier cas : $\boldsymbol{\rang(N) = 1}$}

\'Ecrivons $D$ sous la forme $D = D^+ - D^-$ où $D^+$ et $D^-$ sont effectifs et pour tout $x$ dans le support de $D$, $mult_{D^+}(x) = \max(0,mult_{D}(x))$ et $mult_{D_-}(x) = -\min(0,mult_D(x))$. Prenons un fibré en droites complexes $(N',c_{N'})$ obtenu après transformation élémentaire réelle négative de $(N,c_N)$ en $D^+$ (voir \S \ref{transfopar}). D'après la Proposition \ref{transfoauto} et les Lemmes \ref{transfodet} et \ref{transfod}, nous avons un isomorphisme naturel
\begin{multline*}\label{isomun}
\Det\otimes \DD_D(N) = \Det[N']\otimes\DD_{D-D^+}(N')\\ \otimes \det(N_{|D^+})\otimes \dd\bigotimes_{x\in D^+_{|\RR\surg}} \left((T_x^*\RR\surg)^{\otimes mult_D(x)}\otimes \RR N_x\right)^{\otimes mult_{D}(x)}.\tag{$*$}
\end{multline*}
Décomposons $D^+$ en trois diviseurs effectifs : $D^+ = D^+_{|\surg\setminus\RR\surg} + D^+_{|\RR\surg,0} + D^+_{|\RR\surg,1}$, où $D^+_{|\RR\surg,0}$, respectivement $D^-_{|\RR\surg,1}$, contient les points réels de $D^+$ qui sont de multiplicité paire, respectivement impaire. On a alors les isomorphismes canoniques suivant : pour les points complexes conjugués
\[
\det(\bigoplus_{\{z,\overline{z}\}\in D^+}\RR (\bigoplus_{m = 0}^{mult_D(z)-1}\oplus_{y\in\{z,\overline{z}\}}\left((T^*_y\surg)^{\otimes m}\otimes N_y\right))) = \dd\det(\RR\jet_{|D^+_{|\surg\setminus\RR\surg}}),
\]
pour les points réels de multiplicités paires
\[
\dd\det(\bigoplus_{x\in D^+_{|\RR\surg,0}} (\oplus_{m=0}^{mult_D(x)-1} (T^*_x\RR\surg)^{\otimes m})\otimes \RR N_x) = \dd\det(\RR\jet_{D^+_{|\RR\surg,0}}),
\]
car pour chaque $x\in D^+_{|\RR\surg,0}$, $\oplus_{m=0}^{mult_D(x)-1} (T^*_x\RR\surg)^{\otimes m}$ est de dimension paire,
\[
\dd\bigotimes_{x\in D^+_{|\RR\surg,0}} \left((T_x^*\RR\surg)^{\otimes mult_D(x)}\otimes \RR N_x\right)^{\otimes mult_{D}(x)} =\RR,
\]
et pour les points réels de multiplicités impaires
\begin{multline*}
\dd\det(\bigoplus_{x\in D^+_{|\RR\surg,1}} (\oplus_{m=0}^{mult_D(x)-1} (T^*_x\RR\surg)^{\otimes m})\otimes \RR N_x) \dd\bigotimes_{x\in D^+_{|\RR\surg,1}} \left((T_x^*\RR\surg)^{\otimes mult_D(x)}\otimes \RR N_x\right)^{\otimes mult_{D}(x)} \\ = \dd\det(\RR\jet_{D^+_{|\RR\surg,1}}) \dd\bigotimes_{x\in D^+_{|\RR\surg,1}} (T_x^*\RR\surg)^{\otimes mult_D(x)}.
\end{multline*}

Ainsi l'isomorphisme (\ref{isomun}) devient
\[\label{isomunp}
\Det\otimes \DD_D(N) = \Det[N']\otimes \DD_{D-D^+}(N')\otimes\dd\det(\RR\jet_{D^+})\otimes\dd\bigotimes_{x\in D^+_{|\RR\surg}} (T_x^*\RR\surg)^{\otimes mult_D(x)}.\tag{$*'$}
\]

Prenons maintenant un fibré $(N'',c_{N''})$ obtenu après transformation élémentaire de $(\TCC,conj)$ en $D^-$. Nous avons un autre isomorphisme naturel
\[\label{isomdeux}
\Det[\TCC]\otimes \DD_0(\TCC) = \Det[N'']\otimes \DD_{-D^-}(N'')\otimes\det(\RR\jet_{D^-}).\tag{$**$}
\]

Remarquons de plus que $(N',c_{N'})$ et $(N'',c_{N''})$ sont isomorphes. D'après le Lemme \ref{isomdetd}, on a un isomorphisme canonique
\[\label{isomtrois}
\Det[N']\otimes \DD_{D-D^+}(N') = \Det[N'']\otimes \DD_{-D^-}(N'').\tag{$***$}
\]

En combinant les isomorphismes (\ref{isomunp}),(\ref{isomdeux}) et (\ref{isomtrois}), on obtient 
\[
\Det \otimes \DD_D(N) = \Det[\TCC]\otimes\DD_0(\TCC)\otimes\det(\RR\jet_{D})\dd\bigotimes_{x\in D^+_{|\RR\surg}} (T_x^*\RR\surg)^{\otimes mult_D(x)}.
\]

Ainsi, d'après le Lemme \ref{isomdetd}, on a un isomorphisme canonique
\[
\Det \otimes \DD_D(N) = \det(H^1(\surg,\RR)_{-1})\otimes\det(\RR\jet_{D})\dd\bigotimes_{x\in D^+_{|\RR\surg}} (T_x^*\RR\surg)^{\otimes mult_D(x)},
\]
ce qui conclut le cas du rang un. Notons que cet isomorphisme ne dépend pas des choix de $N'$ et $N''$ comme nous l'avons remarqué à la fin du \S \ref{transfopar}.

{\bf Second cas : $\boldsymbol{\rang(N) > 1}$}

D'après le Lemme \ref{isomdetd}, on a
\begin{multline*}
\Det\otimes \DD_D(N) \otimes \PV = \Det[\det(N)\oplus\TCC^{\oplus \rang(N)-1}]\otimes \DD_D(\det(N)\oplus \TCC^{\oplus \rang(N)-1}) \\ \otimes \PV[(\det(N)\oplus\TCC^{\oplus \rang(N)-1})].
\end{multline*}

Par définition, $\DD_D(\det(N)\oplus \TCC^{\oplus \rang(N)-1}) = \DD_D(N)$. On oriente maintenant le fibré $\PV[(\det(N)\oplus\TCC^{\oplus \rang(N)-1})]$ en utilisant la structure associée à $\RR\det(N)$ (voir Définition \ref{pincan}). On obtient donc
\[
\Det\otimes \DD_D(N) \otimes \PV = \Det[\det(N)\oplus\TCC^{\oplus \rang(N)-1}]\otimes \DD_D(N).
\]

D'autre part, on a une inclusion $(\DB,J)\in\rop[\det(N)]\mapsto (\DB\oplus \DB_J^{\oplus \rang(N)-1},J)\in \rop[\det(N)\oplus\TCC^{\rang(N)-1}]$ et un isomorphisme
\[
\Det[\det(N)\oplus\TCC^{\oplus \rang(N)-1}]_{|\rop[\det(N)]}\otimes \DD_D(N) = \Det[\det(N)]\otimes \Det[\TCC]_{|\RR J(\surg)}^{\otimes \rang(N)-1}\otimes \DD_D(N).
\]

En utilisant le cas précédent et le Lemme \ref{isopart},
\[
\Det[\det(N)\oplus\TCC^{\oplus \rang(N)-1}]_{|\rop[\det(N)]}\otimes \DD_D(N) = \det(H^1(\surg,\RR)_{-1})^{\otimes \rang(N)}\otimes\det(\RR\jet_{D})\dd\bigotimes_{x\in D^+_{|\RR\surg}} (T_x^*\RR\surg)^{\otimes mult_D(x)}.
\]

On obtient ainsi un isomorphisme
\[
\Det\otimes \DD_D(N) \otimes \PV = \det(H^1(\surg,\RR)_{-1})^{\otimes\rang(N)}\otimes\det(\RR\jet_{D})\dd\bigotimes_{x\in D^+_{|\RR\surg}} (T_x^*\RR\surg)^{\otimes mult_D(x)},
\]
défini à homotopie près. Cet isomorphisme est bien canonique.

Pour le deuxième point, pour chaque élément $\ux = (x^+_1,\ldots,x^+_{r^+},x^-_1,\ldots,x^-_{r^-},\uz^+,\uz^-) \in\surg^{(\ud)}$, nous avons un isomorphisme canonique
\[
\Det = \PV \otimes \DD_{D_{\ux}}(N)\otimes\det(H^1(\surg,\RR)_{-1})^{\otimes\rang(N)}\otimes\det(\RR\jet_{D_{\ux}})\dd\bigotimes_{i=1}^{r^+} T_{x_i^+}^*\RR\surg.
\]
En considérant une famille à un paramètre de diviseurs, on vérifie que la famille obtenue d'isomorphismes est bien continue. Ainsi, on obtient un isomorphisme canonique
\[
pr_1^*\Det = \PV \otimes \DD_{(\ud)}(N)\otimes\det(H^1(\surg,\RR)_{-1})^{\otimes\rang(N)}\otimes\det(\RR\jet_{(\ud)})\otimes T_{(\ud)}.
\]
 \end{proof}

\subsubsection{Deux cas particuliers}\label{corpart}

Nous précisons l'énoncé du Théorème \ref{totalaction} dans deux cas particuliers : lorsque la courbe $(\surg,c_\sur)$ est séparante puis lorsque la partie réelle du fibré $(N,c_N)$ est orientable. Nous retrouverons ces deux cas au \S \ref{module}.

\paragraph{Le cas séparant}

Soit $(N,c_N)$ un fibré vectoriel complexe au-dessus d'une courbe $(\surg,c_\sur)$ séparante. On désignera par $\RR\surg^+$ (resp. $\RR\surg^-$) la réunion des composantes connexes de $\RR\surg$ sur lesquelles $\RR N$ est orientable (resp. ne l'est pas). On note $k = b_0(\RR\surg)$ et $k_- = b_0(\RR\surg^-)$. Notons $O_{\RR N}$ la droite réelle $\dd\bigotimes_{c\subset\RR \surg^+} o((\RR N)_{|c})$, où $o((\RR N)_{|c})$ est la droite réelle engendrée par les deux orientations opposées de $\RR N$ restreint à la composante $c$ de $\RR\surg^+$. Un automorphisme $\Phi\in\riso{N}$ induit un élément $\Phi_o$ de $GL(O_{\RR N})$. D'autre part, on note $H$ la droite réelle engendrée par les deux orientations complexes de $\RR\surg$. Un difféomorphisme $\varphi\in\rdiff{N}$ d'une courbe $(\surg,c_\sur)$ séparante induit une transposition $\sigma_\varphi^{\RR\surg}$ sur les deux orientations complexes de $\RR\surg$ et une permutation $\varphi^-_{|\RR\surg}$ des composantes de $\RR\surg^-$.

\begin{Corollaire}\label{actionsepa}
  Soit $(N,c_N)$ un fibré vectoriel complexe sur $(\surg,c_\sur)$ que l'on suppose séparante. Il existe un isomorphisme canonique
\[
\Det = \PV\otimes O_{\RR N}\otimes \det(H_0(\RR\surg^-,\RR)) \otimes\det(H^1(\surg,\RR)_{-1})^{\otimes\rang(N)}\otimes H^{\otimes\frac{\deg(N)+k_-}{2}}.
\]

En particulier, soit $\varphi\in\rdiff{N}$ et $\Phi\in\riso{N}$ un relevé de $\varphi$. Le signe de l'action de $\Phi$ sur les orientations du fibré $\Det$ est donné par
  \[
  \varepsilon(\Phi_{\PPP^+}) \det(\Phi_o) \varepsilon_N(\varphi),
  \]
où $\varepsilon_N(\varphi) = \varepsilon(\sigma_\varphi^{\RR\surg})^{\frac{\deg(N)+k_-}{2}} \varepsilon(\varphi^-_{\RR\surg}) \det(\varphi^*)^{\rang(N)}$ ne dépend que de $\varphi$ et de $(N,c_N)$.
\end{Corollaire}

\begin{proof}
Prenons $\ud_{min}$ comme dans le Corollaire \ref{reponse}. D'après le Théorème \ref{totalaction}, il existe un isomorphisme canonique entre fibrés en droites réelles au-dessus de $pr_1^*\rop$
\[
pr_1^* \Det = \PV\otimes\det(H^1(\surg,\RR)_{-1})^{\otimes\rang(N)}\otimes \pi^*(\DD_{(\ud_{min})}\otimes \det(\RR\jet_{(\ud_{min})}))\otimes (pr_2\circ\pi)^*(T_{(\ud_{min})}).
\]
Or, comme la courbe est séparante, on a un isomorphisme canonique entre $T_{(\ud)_{min}}$ au-dessus de $\surg^{(\ud_{min})}$ et le fibré produit $\surg^{(\ud_{min})}\times H^{\otimes k_-}$. D'autre part, on a un isomorphisme canonique entre $\det(\RR\jet_{(\ud_{min})})$ et $(\RR J(\surg)\times \surg^{(\ud_{min})})\times (\det(H_0(\RR\surg^-,\RR))\otimes H^{\otimes\frac{\deg(N)-k_-}{2}})$. Enfin, le fibré $\DD_{(\ud_{min})}$ au-dessus de $\RR J(\surg)\times \surg^{(\ud_{min})}$ est canoniquement isomorphe au fibré produit $(\RR J(\surg)\times \surg^{(\ud_{min})})\times O_{\RR N}$. En effet, l'application
\[
\begin{array}{l c l}
\DD_{(\ud_{min})} & \rightarrow & O_{\RR N} \\
\left[ \DB^+_{D,J}\right] & \mapsto & \bigotimes_{c \subset \RR\surg^+}\sigma_{|c},
\end{array}
\]
où $\sigma$ est une section méromorphe réelle de $(\det(N),c_{\det(N)},\DB^+_{D,J})$ de diviseur $D$ et donnant la polarisation de $\DB^+_{D,J}$, est bien définie d'après le Corollaire 4.1 de \cite{article1} et donne l'isomorphisme canonique voulu. Ceci termine la démonstration.
\end{proof}

\paragraph{Le cas $Spin$}\label{parspin}

Lorsque la courbe n'est pas séparante, nous ne trouvons pas de résultat aussi général que le Corollaire \ref{actionsepa}. Nous pouvons toutefois préciser le Théorème \ref{totalaction}. Nous utilisons pour cela des structures $Spin$ associées à un fibré vectoriel complexe muni d'une structure réelle $(N,c_N)$ (qui est supposé de degré pair), dont nous notons $\RR Spin(N)$ l'ensemble (voir \cite{degitkha} ou le \S 4.1.2 de \cite{article1}). Insistons ici sur le fait que si l'on a encore une bijection entre les
structures $Spin$ sur $N$ (muni d'une structure holomorphe) et les classes d'isomorphisme
de racines carrées de son fibré déterminant $\det(N)$ (voir \cite{atiyah}),
celle-ci dépend de la structure holomorphe sur $N$. Ainsi lorsque nous
parlerons d'une structure $Spin$ sur $N$ comme d'une racine carrée du
fibré déterminant, nous ferons attention à préciser l'opérateur de
Cauchy-Riemann sur $N$ considéré.

Enfin, comme nous considèrerons des structures $Spin$ réelles sur
$(N,c_N)$ (lorsque $w_1(\RR N) = 0$), remarquons qu'une telle
structure, donnée par un fibré en droite holomorphe $L$, induit une semi-orientation sur le fibré $\RR N$;
c'est-à-dire que si nous fixons une orientation sur une des
composantes de $\RR N$, toutes les autres sont automatiquement
orientées (voir \S 3.4 de \cite{degitkha}). De plus, comme nous supposons toujours que $\RR\surg$ est non vide, la structure réelle $c_{\det(N)}$ se relève en une structure réelle $c_L$ sur $L$ unique à multiplication par $i$ près. Le choix d'un tel relevé fixe une orientation sur $\RR N$, et l'autre relevé induit l'autre orientation. Notons $O^{\xi}_{\RR N}$ la droite réelle engendrée par les deux semi-orientations de $\RR N$ associées à la structure $\xi$. Si un automorphisme $\Phi\in\riso{N}$ préserve une structure $Spin$ réelle $\xi$ sur $N$, alors $\Phi$ induit une permutation $\Phi_{o,\xi}$ sur les semi-orientations associées à $\xi$.

 Par ailleurs, la première classe de
Stiefel-Whitney $w_{\xi}$ d'une structure $Spin$ réelle $\xi$ est toujours bien
définie et ne dépend pas de la structure holomorphe choisie sur $N$. Pour $w\in H^1(\RR\surg,\ZZ/2\ZZ)$, on pose $H^1_w(\RR\surg,\RR) = \dd\bigotimes_{w([\RR\surg]_i) = 1} H^1((\RR\surg)_i,\RR)$.

\begin{Corollaire}\label{actionspin}
  Soit $(N,c_N)$ un fibré vectoriel complexe sur une courbe $(\surg,c_\sur)$ de partie réelle non vide. Supposons de plus que le degré de $N$ est pair et que la première classe de Stiefel-Whitney de $\RR N$ est nulle. Il existe un isomorphisme canonique
\[
\Det = \PV\otimes\det(H^1(\surg,\RR)_{-1})^{\otimes\rang(N)}\otimes (O_{\RR N}^{\xi})^{\otimes 1-g} \otimes H^1_{w_\xi}(\RR\surg,\RR).
\]
entre fibrés en droites au-dessus de $\rop\times \RR Spin(N)$.

En particulier, soit $\varphi\in\rdif$ et $\Phi\in\riso{N}$ un relevé de $\varphi$ qui préserve une structure $Spin$ réelle $\xi\in\RR Spin(N)$. Le signe de l'action de $\Phi$ sur les orientations du fibré $\Det$ est donné par
\[
\varepsilon(\Phi_{\PPP^\pm})\varepsilon(\Phi_{o,\xi})^{1-g} \varepsilon(\sigma_\varphi^{w_\xi})\det(\varphi^*)^{\rang(N)},
\]
où $\sigma_\varphi^{w_\xi}$ est la permutation induite par $\varphi$ sur l'ensemble des orientations des composantes de $\RR\surg$ sur lesquelles $w_\xi$ est non nulle.
\end{Corollaire}

\begin{proof}
Supposons pour simplifier que $N$ est de rang $1$. Prenons une structure $Spin$ réelle $\xi$ sur $N$ de première classe de Stiefel-Whitney $w\in\hczdeux{\RR\surg}$. Fixons un point $x_i$ sur chaque composante de $\RR\surg$ sur laquelle $w$ est non nulle et prenons un diviseur $D$ associé à $N$ de la forme $D = \dd\sum_{i=1}^l 2x_i + 2\frac{\deg(N)-2l}{4}(z+\overline{z})$. D'après le Théorème \ref{totalaction}, il existe un isomorphisme canonique
\[
\Det = \pi^*\DD_D(N)\otimes \det(H^1(\surg,\RR)_{-1}) \otimes \det(\RR \jet_D),
\]
car les points réels de $D$ sont tous de multiplicité $2$. De plus,
\[
\det(\RR\jet_D) = \dd\bigotimes_{i = 1}^l T^*_{x_i} \RR\surg = H^1_w(\RR\surg,\RR),
\]
car la multiplicité de $z$ est paire. 

Il nous faut vérifier que $\pi^*\DD_D(N)\otimes (O_{\RR N}^\xi)^{\otimes 1-g}$ est naturellement orienté. Prenons tout d'abord un élément $[\DB^+_{D,J}]$ de $\DD_D(N)$. Remarquons qu'une section méromorphe réelle $s$ associée à $\DB^+_{D,J}$ induit une structure $Spin$ réelle sur $N$: au-dessus d'un lacet sur $\surg$ qui ne passe pas par un des points $x_i$, $z$ ou $\overline{z}$, $s$ donne une trivialisation de $N$ qui, par définition se relève dans la structure $Spin$ induite. Si on homotope un lacet $\gamma$ sur $\surg$ pour le faire traverser un des points $z$ ou $\overline{z}$ et obtenir un lacet $\gamma'$, alors on a 
\[
[s\circ\gamma] = [s\circ\gamma'] + 2\frac{\deg(N)-2l}{4}f = [s\circ\gamma']\in H_1(R_{N}^+,\ZZ/2\ZZ),
\]
où $f$ est la classe de la fibre. Si on homotope un lacet $\gamma$ sur $\surg$ pour le faire traverser un des points $x_i$ et obtenir un lacet $\gamma'$, alors on a 
\[
[s\circ\gamma] = [s\circ\gamma'] + 2f = [s\circ\gamma']\in H_1(R_{N}^+,\ZZ/2\ZZ).
\]
 De plus, la structure $Spin$ réelle donnée par $\OO_{\surg,J}(\frac{D}{2})$ est la même que celle induite par une section méromorphe associée à $\DB^+_{D,J}$, et nous la noterons $\xi'$. La section $s$ induit une orientation de $\RR N$ qui est compatible avec la semi-orientation donnée par $\xi'$.

Si $[\DB^{',+}_{D,J'}]$ est un autre élément de $\DD_D(N)$ qui induit la même orientation de $\DD_D(N)$ que $[\DB^+_{D,J}]$, notons $\xi''$ la nouvelle structure $Spin$ réelle obtenue. Alors, d'après le Lemme 4.4 de \cite{article1}, il existe un automorphisme $f\in\raut{N}^+$ uniquement défini à homotopie près tel que $f^* \xi' = \xi''$ et $f$ envoie l'orientation de $\RR N$ fixée par $[\DB^+_{D,J}]$ sur celle fixée par $[\DB^{',+}_{D,J'}]$.

D'autre part, il existe un automorphisme $h\in\raut{N}$ qui envoie la structure $\xi$ sur $\xi'$. De plus, d'après le Lemme 4.4 de \cite{article1}, l'automorphisme $h$ est unique à homotopie et multiplication par $-1$ près. Si $g$ est impair, alors $h$ et $-h$ ont même action sur $\Det$. La structure $\xi$ induit une orientation de $\DD_D(N)$ donnée par $h^*[\DB^+_{D,J}]$. Celle-ci ne dépend pas du choix initial de la section $[\DB^+_{D,J}]$, car l'orientation obtenue en utilisant $[\DB^{',+}_{D,J'}]$ est donnée par $h^*f^*[\DB^{',+}_{D,J'}]$ et $f\in\raut{N}^+$. Si $g$ est pair, $h$ et $-h$ ont des actions opposées sur $\Det$. Supposons par exemple que $h\in \raut{N}^+$. Alors $h^*[\DB^+_{D,J}]$ induit une semi-orientation de $\xi$. Celle-ci ne dépend pas non plus du choix initial de la section $[\DB^+_{D,J}]$, car la semi-orientation obtenue en utilisant $[\DB^{',+}_{D,J'}]$ est donnée par $h^*f^*[\DB^{',+}_{D,J'}]$ et $f$ envoie l'orientation de $\RR N$ fixée par $[\DB^+_{D,J}]$ sur celle fixée par $[\DB^{',+}_{D,J'}]$. Ceci fournit donc un isomorphisme $\DD_D(N)\rightarrow O^\xi_{\RR N}$.

On vérifie enfin que l'isomorphisme obtenu ne dépend pas du diviseur $D$ choisit : tous les autres qui conviennent lui sont homotopes, et la construction précédente dépend continûment de $D$.
\end{proof}

\section{Orientabilité des espaces de modules de courbes réelles}\label{module}

Soit $(X,\omega,c_X)$ une variété symplectique réelle de dimension $2n$ au moins $4$. Nous nous intéressons à l'orientabilité des espaces de modules de courbes réelles dans $X$. Nous rappelons tout d'abord succintement la construction de ces espaces (en suivant principalement \cite{shev} et \cite{wel1}) ainsi que le rapport avec le fibré déterminant. Nous énonçons ensuite des résultats concernant leur orientabilité.

\subsection{Espaces de modules et fibré déterminant}\label{espmod}

Soit $\surg$ une surface réelle compacte orientée de genre $g$, et $l$ un entier assez grand. Nous notons comme précédemment $J(\surg)$ (resp. $J_\omega(X)$) l'espace des structures complexes de classe $C^l$ sur $\surg$ compatibles avec l'orientation fixée sur la surface (resp. sur $X$ et calibrées par $\omega$). Le sous-ensemble $\RR J_\omega(X)$ de $J_\omega(X)$ est formé des structures presque complexe $J$ telles que $c_X$ est $J$-antiholomorphe. C'est une variété de Banach contractile (voir la Proposition 1.1 de \cite{wel1}).

\begin{Definition}
  Une courbe pseudo-holomorphe paramétrée de genre $g$ dans $X$ est un triplet $(u,J_\sur,J)\in L^{k,p}(\surg,X)\times J(\surg)\times J_\omega(X)$, où $1\ll k \ll l$ et $2<p$, vérifiant
\[
\d u + J\circ \d u \circ J_\sur = 0.
\]
On dit qu'une courbe pseudo-holomorphe $(u,J_\sur,J)$ réalise la classe $d\in H_2(X,\ZZ)$ si $u_*([\surg]) = d$. Enfin, une courbe $(u,J_\sur,J)$ est dite injective quelque part s'il existe un ouvert non vide $U\subset \surg$ tel que $u$ est une immersion sur $U$ et
\[
\forall x\in U,\ u^{-1}\{u(x)\} = \{x\}.
\]
\end{Definition}

\begin{Rem}
  Dans la suite, nous ne considèrerons que des courbes injectives quelque part. C'est une hypothèse technique pour éviter d'éventuels problèmes analytiques dans l'étude des espaces de modules de courbes.
\end{Rem}

\begin{Definition}
Pour $d\in H_2(X,\ZZ)$, et $r \in \NN$, on note $\MP_{g,r}^d(X)$ l'espace de modules des courbes pseudo-holomorphes paramétrées de genre $g$ dans $X$, injectives quelque part, réalisant la classe $d$, avec $r$ points marqués, c'est-à-dire,

\begin{multline*}
\MP_{g,r}^d(X) =  \left\{(u,J_\sur,J,\underline{z})\in L^{k,p}(\surg,X)\times J(\surg)\times J_\omega(X)\times\surg^r\ |\right.  \d u + J\circ \d u \circ J_\sur = 0\\
\begin{array}{l}
 u_*([\surg]) = d \\
 u \text{ est injective quelque part} \\
 \forall 1\leq i\neq j\leq r,\ z_i\neq z_j\left.\right\}.
\end{array}
\end{multline*}

\end{Definition}

Prenons une permutation $\tau$ de $\{1,\ldots,r\}$ d'ordre $2$. Comme expliqué dans le \S 1.3 de \cite{wel1}, le groupe $Diff(\surg)$ des difféomorphismes de classe $C^{l+1}$ de $\surg$ agit sur $\MP_{g,r}^d(X)$ par reparamétrage:
\[
\begin{split}
(\phi,(u,J_\sur,J,\underline{z}))\in Diff(\surg)\times\MP_{g,r}^d(X)\hspace{7cm} \\ \mapsto \left\{\begin{array}{l r}\left(u\circ\phi^{-1},(\phi^{-1})^*J_\sur,J,\phi(\underline{z})\right) & \text{si } \phi\in Diff^+(\surg)\\
\left(c_X\circ u\circ\phi^{-1},-(\phi^{-1})^*J_\sur,J,(\phi(z_{\tau(1)}),\ldots,\phi(z_{\tau(r)}))\right) & \text{sinon}.
\end{array}\right.
\end{split}
\]
Une courbe pseudo-holomorphe injective quelque part est dite réelle si elle est fixée par l'action d'un élément non trivial $c_\sur$ de $Diff(\surg)$ qui est alors unique; autrement dit, un élément de $\MP_{g,r}^d(X)$ n'est fixée que par un seul difféomorphisme non trivial. De plus, le difféomorphisme $c_\sur$ est une involution qui renverse l'orientation fixée sur $\surg$, c'est-à-dire une structure réelle (voir le Lemme 1.3 de \cite{wel1}). \`A chaque structure réelle $c_\sur$ est alors associé $\RR_{c_\sur}\MP_{g,r}^d(X)$, l'ensemble de ses points fixes dans $\MP_{g,r}^d(X)$. Notons $\RR_{\tau}\MP_{g,r}^d(X)$ la réunion de tous ces ensembles. Cet espace est une variété de Banach séparable de classe $C^{l-k}$ (voir Proposition $1.4$ de \cite{wel1}).

Remarquons qu'une structure réelle $c_\sur$ agit sur $\surg^r$ par
\[
(z_1,\ldots,z_r)\rightarrow (c_\sur(z_{\tau(1)}),\ldots,c_\sur(z_{\tau(r)})).
\]
 Sur une courbe réelle marquée, la configuration de points $\underline{z}$ est un élément de $\RR_{c_\sur}(\surg^r)$, le lieu des points fixes de l'action de $c_\sur$ sur $\surg^r$.

Prenons une structure réelle $c_\sur$ sur $\surg$ et considérons $(u,J_\sur,J,\underline{z})\in\RR_{c_\sur}\MP_{g,r}^d(X)$. Le fibré $E_u = u^*TX$ est un fibré vectoriel complexe sur $(\surg,c_\sur)$, de rang $n$, de degré $c_1(X)d$, et $c_X$ induit une structure réelle $c_{E_u}$ dessus. On note $H^1_{c_1(X)d}(\RR\surg,\ZZ/2\ZZ)\subset H^1(\RR\surg,\ZZ/2\ZZ)$ l'image réciproque de $c_1(X)d$ par l'application d'augmentation $w\in H^1(\RR\surg,\ZZ/2\ZZ)\mapsto \dd\sum_i w[(\RR\surg)_i] \in \ZZ/2\ZZ$. Ainsi, on a $w_1(\RR E_u)\in H^1_{c_1(X)d}(\RR\surg,\ZZ/2\ZZ)$.

\begin{Lemme}
  L'application $w_{c_\sur} : (u,J_\sur,J,\underline{z})\in\RR_{c_\sur}\MP_{g,r}^d(X)\mapsto w_1(\RR E_u)\in H_{c_1(X)d}^1(\RR\surg,\ZZ/2\ZZ)$ est localement constante.\qed
\end{Lemme}

Ainsi, on peut écrire $\RR_{c_\sur}\MP_{g,r}^d(X) = \dd\bigcup_{c\in H^1_{c_1(X)d}(\RR\surg,\ZZ/2\ZZ)}w_{c_\sur}^{-1}(c)$, où chaque $w_{c_\sur}^{-1}(c)$ est une union de composantes connexes de $\RR_{c_\sur}\MP_{g,r}^d(X)$. On décompose ainsi l'espace de modules selon le type topologique de la restriction à la courbe du fibré tangent à la partie réelle de la variété ambiante. On dira qu'un quadruplet d'entiers naturels $\ud$ est adapté à $c\in H^1_{c_1(X)d}(\RR\surg,\ZZ/2\ZZ)$, si $\ud$ est adapté à $(E_u,c_{E_u})$ pour chaque courbe $(u,J_\sur,J,\underline{z})\in w_{c_\sur}^{-1}(c)$ (voir \S \ref{diviseur}). Posons alors
\[
\RR_{\tau}\MP_{g,r,pol}^d(X) = \dd\bigsqcup_{\begin{subarray}{c}c_\sur\in Diff^-(\surg) \\ (c_\sur)^2=\id_\sur \end{subarray}}\bigsqcup_{\begin{subarray}{c}c\in H^1_{c_1(X)d}(\RR\surg,\ZZ/2\ZZ) \\ \ud \text{ adapté à } c \end{subarray}}w_{c_\sur}^{-1}(c)\times \surg^{(\ud)}
\]

Autrement dit, on a ajouté à chaque courbe une configuration réelle de points marqués signés supplémentaire formant un diviseur adapté au tiré en arrière du tangent à la variété ambiante. Insistons sur le fait que deux composantes $w_{c_\sur}^{-1}(c)\times \surg^{(\ud)}$ et $w_{c_\sur}^{-1}(c)\times \surg^{(\ud')}$ peuvent ne pas être disjointes en oubliant les signes des points marqués, mais elles apparaissent tout de même toutes les deux dans la réunion totale si $\ud\neq\ud'$.

Introduisons quelques fibrés sur ces espaces de modules.

\begin{Lemme}\label{fibr}
  \begin{enumerate}
  \item Il existe un fibré $\ZZ/2\ZZ$-principal $\PPP^+_X$ au-dessus de $\RR \MP_{g}^d(X)$ dont la fibre au-dessus d'un point $(u,J_{\surg},J)$ est $\dd\bigotimes_{i=1}^{b_0(\RR\surg)} Pin^+((\RR E_u)_i)$. De plus, l'action de $Diff^+(\surg)$ sur $\RR \MP_g^d(X)$ se relève naturellement sur $\PPP^+_X$.
  \item Il existe un fibré $\ZZ/2\ZZ$-principal $\DD$ au-dessus de $\RR\MP_{g,pol}^d(X)$ dont la fibre au-dessus d'un point $(u,J_{\surg},J,\ux)$ est $\DD_{J_\sur,D_{\ux}} (E_u)$ (voir la définition \ref{defd} au \S \ref{diviseur}). De plus, l'action de $Diff^+(\surg)$ sur $\RR \MP_{g,pol}^d(X)$ se relève naturellement sur $\DD$.
  \item Il existe un fibré en droites réelles $T_{pol}$ au-dessus de $\RR\MP_{g,pol}^d(X)$ dont la fibre au-dessus d'un point $(u,J_{\surg},J,\ux)$ est $\det(\RR\jet_{D_{\ux}})\otimes \dd\bigotimes_{x\in\ux^+\cap\RR\surg} T_{x}^*\RR\surg$ (voir la définition \ref{defjet} au \S \ref{diviseur}). De plus, l'action de $Diff^+(\surg)$ sur $\RR\MP_{g,pol}^d(X)$ se relève naturellement sur $T_{pol}$.
  \item Il existe un fibré en droites réelles $L_{r,pol}$ (resp. $L_r$) au-dessus de $\RR_{\tau}\MP_{g,r,pol}^d(X)$ (resp. au-dessus de $\RR_{\tau}\MP_{g,r}^d(X)$) dont la fibre au-dessus d'un point $(u,J_{\surg},J,\underline{z},\ux)$ est $\det(\RR_{\tau} \bigoplus_{z\in\underline{z}} T_z\surg)\otimes \det(T_{\ux}\RR \surg^{(\ud)})$ (resp. $\det(\RR_{\tau} \bigoplus_{z\in\underline{z}} T_z\surg)$). De plus, l'action de $Diff^+(\surg)$ sur $\RR\MP_{g,r,pol}^d(X)$ (resp. sur $\RR\MP_{g,r}^d(X)$) se relève naturellement sur $L_{r,pol}$ (resp. $L_r$).
  \item Il existe un fibré en droites réelles $\det(H^1(\surg,\RR)_{-1})$ au-dessus de $\RR\MP_{g}^d(X)$ dont la fibre au-dessus d'un point $(u,J_{\surg},J)$ est $\det(H^1(\surg,\RR)_{-1})$. De plus, l'action de $Diff^+(\surg)$ sur $\RR\MP_{g}^d(X)$ se relève naturellement sur $\det(H^1(\surg,\RR)_{-1})$.
  \end{enumerate}
\end{Lemme}

\begin{proof}
Notons $ev : (u,J_\sur,J,z)\in \RR \MP_g^d(X) \times \surg \mapsto u(z)\in X$ et $E = ev^*TX$. Ce dernier est un fibré vectoriel complexe, et $c_X$ induit une structure réelle $c_E$ dessus, de sorte que la restriction de $(E,c_E)$ au-dessus de $\{(u,J_\sur,J)\}\times\surg$ est égale à $(E_u,c_{E_u})$. Prenons un voisinage $U\subset \RR \MP_g^d(X)$ de $(u,J_\sur,J)$ qui trivialise $(E,c_E)$. C'est-à-dire que l'on a une application $F$ de classe $C^l$ entre $(E,c_E)_{|U}$ et $(E_u,c_{E_u})\times U$ telle que pour tout $(v,J'_\sur,J')\in U$, $F$ induise un isomorphisme entre $(E_v,c_{E_v})$ et $(E_u,c_{E_u})$. L'application $F$ induit un isomorphisme entre le fibré $\PPP^+_X$ restreint à $U$ et le fibré trivial $\dd\bigotimes_{i=1}^{b_0(\RR\surg)} Pin^+((\RR E_u)_i) \times U$. Ceci définit la structure de fibré $\ZZ/2\ZZ$-principal sur $\PPP^+_X$. De plus, un difféomorphisme $\varphi\in Diff^+(\surg)$ induit un isomorphisme $\varphi^* : (E_u,c_{E_u}) \rightarrow (E_{u\circ\varphi^{-1}},c_{E_{u\circ\varphi^{-1}}}) = (\varphi^{-1})^*(E_u,c_{E_u})$ naturel, ce qui induit une action naturelle de $Diff^+(\surg)$ sur $\PPP^+_X$.

Notons $\Pi : \RR\MP_{g,pol}^d(X) \rightarrow \RR\MP_g^d(X)$ l'application d'oubli des points marqués. Reprenons les mêmes notations que ci-dessus. Le fibré $\Pi^*(E,c_{E})$ est trivialisé par $F$ au-dessus de $U\times \surg^{\ud}$ pour tout les $\ud$ adaptés. On a de plus une application $\eta : U\times \surg^{\ud} \rightarrow \RR J(\surg)\times \surg^{\ud}$ d'oubli de $X$. L'application $F$ induit alors un isomorphisme entre le fibré $\DD$ restreint à $U\times \surg^{\ud}$ et le fibré $\eta^*\DD(E_u)$. Ceci donne à $\DD$ une structure de fibré $\ZZ/2\ZZ$-principal. L'action de $Diff^+(\surg)$ sur $\DD$ est donnée de la même façon que précédemment.

Les trois derniers points du Lemme sont immédiats.
\end{proof}

Notons $\Pi_r : \RR_{\tau} \MP_{g,r,pol}^d(X)\rightarrow \RR \MP_{g,pol}^d(X)$ et $\Pi : \RR_{\tau} \MP_{g,r,pol}^d(X)\rightarrow \RR \MP_{g}^d(X)$ les applications d'oubli des points marqués. Celles-ci sont $Diff^+(\surg)$-équivariantes. D'après le Lemme \ref{fibr}, nous avons cinq fibrés $\Pi^*\PPP^+_X$, $(\Pi_r)^*\DD$, $(\Pi_r)^*T_{pol}$, $L_{r,pol}$ et $\Pi^* \det(H^1(\surg,\RR)_{-1})$ sur $\RR_{\tau} \MP_{g,r,pol}^d(X)$.

Nous voulons maintenant considérer des courbes non paramétrées. Pour cela, on remarque que $Diff^+(\surg)$ agit librement sur $\RR_{\tau}\MP_{g,r,pol}^d(X)$. Nous notons $\RR_{\tau}\MM_{g,r,pol}^d(X)$ le quotient de $\RR_{\tau}\MP_{g,r,pol}^d(X)$ par l'action libre de $Diff^+(\surg)$. C'est une variété de Banach de classe $C^{l-k}$ (voir par exemple la Proposition 3.2.1 de \cite{MDS} ou bien le Corollaire 2.2.3 de \cite{shev}). Les fibrés $\Pi^*\PPP^+_X$, $(\Pi_r)^*\DD$, $(\Pi_r)^*T_{pol}$, $L_{r,pol}$ et $\Pi^* \det(H^1(\surg,\RR)_{-1})$ induisent des fibrés notés de la même façon sur $\RR_{\tau}\MM_{g,r,pol}^d(X)$.

  La projection naturelle $\pi : \RR\MM_{g}^d(X)\rightarrow \RR J_\omega(X)$ est une application Fredholm entre variétés de Banach (voir la Proposition 1.9 \cite{wel1}). Il existe un fibré en droites réelles $\ddet(\pi)$ au-dessus de $\RR\MM_{g}^d(X)$ dont la fibre en $[u,J_\sur,J]$ est $\Lambda^{\max}\ker(\d_{[u,J_{\sur},J]} \pi) \otimes \left(\Lambda^{\max}\coker(\d_{[u,J_{\sur},J]} \pi)\right)^*$ (voir \cite{MDS} Annexe A.2).

 \begin{Theoreme}\label{metath}
   Le fibré $\Pi^*\ddet(\pi)$ est canoniquement isomorphe à $\Pi^*\PPP^+_X \otimes (\Pi_r)^*(\DD \otimes T_{pol}) \otimes \Pi^* (\det(H^1(\surg,\RR)_{-1}))^{\otimes n-1}$.

   En particulier, pour une structure presque complexe $J\in\RR J_\omega(X)$ générique, $\RR_{\tau}\MM_{g,r,pol}^d(X,J) = (\pi\circ\Pi)^{-1}\{J\}$ est une variété lisse vérifiant
\[
w_1(\RR\MM_{g,r,pol}^d(X,J)) = w_1\left(\Pi^*\PPP^+_X \otimes (\Pi_r)^*(\DD \otimes T_{pol})\right)_{|J} +  w_1\left(L_{r,pol} \otimes \Pi^* (\det(H^1(\surg,\RR)_{-1}))^{\otimes n-1}\right)_{|J}.
\]
 \end{Theoreme}

 \begin{Rem}
   On peut remarquer que le fibré en droites réelles $\det(H^1(\surg,\RR)_{-1})$ sur $\RR \MM_g^d(X)$ est en fait le tiré en arrière d'un fibré sur $\RR\MM_g$ par l'application $\eta : \RR \MM_g^d(X) \rightarrow \RR\MM_g$ d'oubli de $X$. Et d'après le Corollaire \ref{delmum}, $w_1(\det(H^1(\surg,\RR)_{-1})) = \eta^*w_1(\RR\MM_g)$.
 \end{Rem}

La deuxième partie du Théorème \ref{metath} découle de la première en remarquant d'une part que l'application $\Pi$ est partout régulière, donc que $\coker(\d_{[u,J_\sur,J,\underline{z},\ux]}(\pi\circ\Pi)) = \coker(\d_{[u,J_\sur,J]}\pi)$, d'autre part que le noyau de $\d_{[u,J_\sur,J,\underline{z},\ux]}(\pi\circ\Pi)$ entre dans la suite exacte
\[
0\rightarrow \RR \bigoplus_{z\in\underline{z}} T_z\surg \oplus \RR T_{\ux}\surg^{\ud} \rightarrow \ker(\d_{[u,J_\sur,J,\underline{z},\ux]} (\pi\circ \Pi))\xrightarrow{\d_{[u,J_\sur,J,\underline{z},\ux]}\Pi} \ker(\d_{[u,J_\sur,J]}\pi)\rightarrow 0,
\]
puis en appliquant les Théorèmes de Sard et des fonctions implicites. Nous noterons dans la suite $\RR J_{\omega}^{reg}(X)$ le sous-ensemble de $\RR J_\omega(X)$ formé des valeurs régulières de $\pi$. Remarquons que le fibré $\ddet(\pi)$ est en quelque sorte le fibré des orientations de l'espace de modules universel relativement à $\RR J(\surg)$.

Le fibré vectoriel complexe $(E,c_E)$ sur $\RR\MP_g^d(X)\times \surg$ induit deux fibrés de Banach $\EE$ et $\EE'$ sur $\RR\MP_g^d(X)$ de fibres respectives  en $(u,J_\sur,J)$, $\EE_u = L^{k,p}(\surg,E_u)_{+1}$ et $\EE'_u = L^{k-1,p}(\surg,\Lambda_{J_\sur}^{0,1}\surg\otimes E_u)_{+1}$. L'action de $Diff^+(\surg)$ sur $\RR\MP_g^d(X)$ s'étend de façon naturelle à ces deux fibrés. De plus, étant donnée une courbe $(u,J_\sur,J)\in\RR\MP_{g}^d(X)$, le fibré $E_u$ est muni d'un opérateur de Cauchy Riemann réel généralisé noté $D_u$, dont la partie $J$-linéaire notée $\DB_u$ induit une structure holomorphe sur $E_u$ (voir Lemme 6.1.1 de \cite{ivshev}). On obtient ainsi un opérateur $D : \EE \rightarrow \EE'$ qui est de plus $Diff^+(\surg)$-équivariant (voir Lemme 1.5 de \cite{wel1}).

D'autre part, pour une structure réelle donnée $c_\sur$ sur $\surg$, on a deux fibrés de Banach $\TS$ et $\TS'$ au-dessus de $\RR J(\surg)$ de fibre respectives $\TS_{J_\sur} = L^{k,p}(\surg,T\surg)_{+1}$ et $\TS'_{J_{\sur}} = L^{k-1,p}(\surg,\Lambda_{J_\sur}^{0,1}\surg\otimes T\surg)_{+1}$ et un opérateur $\DB : \TS\rightarrow\TS'$ valant $\DB_{J_\sur}$ au point $J_\sur$. De plus, cet opérateur est $\RR Diff^+(\surg)$-équivariant. En tirant en arrière ces objets par l'application d'oubli de $X$, on obtient deux fibrés de Banach $\TS$ et $\TS'$ au-dessus de $\RR\MP_g^d(X)$ et un opérateur $\DB : \TS\rightarrow \TS'$ qui est $Diff^+(\surg)$-équivariant.

On notera $\ddet(D)$, respectivement $\ddet(\DB)$, les fibrés en droites réelles au-dessus de $\RR\MM_g^d(X)$ dont la fibre en une courbe $[u,J_\sur,J]$ est $\ddet(D_u)$, respectivement $\ddet(\DB_{J_\sur})$. Rappelons maintenant la Proposition \ref{fredholm} (voir \cite{shev} Corollaires $2.2.3$ et $1.5.4$, et la Proposition 1.9 de \cite{wel1}).

\begin{Proposition}\label{fredholm}
  Il existe un isomorphisme canonique $\ddet(\pi) = \ddet(D)\otimes \ddet(\DB)^*$.\qed
\end{Proposition}

\begin{proof}[Démonstration du Théorème \ref{metath}]
  Le Théorème \ref{metath} découle maintenant de la Proposition \ref{fredholm}, du Théorème \ref{totalaction} et du Corollaire \ref{isotsigma}. En effet, en combinant ces trois résultats on a
\[
\begin{array}{c c l}
\Pi^*\ddet(\pi) & = & \Pi^* \ddet(D) \otimes \Pi^* \ddet(\DB)^* \\
               & = &  \left(\Pi^*\PPP^+_X \otimes (\Pi_r)^*\DD  \otimes \Pi^* (\det(H^1(\surg,\RR)_{-1}))^{\otimes n} \otimes (\Pi_r)^*T_{pol}\right) \otimes \left(\Pi^* \det(H^1(\surg,\RR)_{-1})\right).  
\end{array}
\]
\end{proof}

\subsection{Deux cas particuliers}\label{deuxcas}

Nous précisons le Théorème \ref{metath} dans les deux cas traités au \S \ref{corpart}.

\subsubsection{Le cas séparant}

 Nous pouvons partitionner $\RR_\tau\MM_{g,r}^d(X)$ (resp. $\RR_\tau\MP_{g,r}^d(X)$) en sous-ensembles $\RR^{sep}_\tau\MM_{g,r}^d(X)$ (resp. $\RR^{sep}_\tau\MP_{g,r}^d(X)$) qui contient les courbes réelles séparantes et $\RR^{nsep}_\tau\MM_{g,r}^d(X)$ (resp. $\RR^{nsep}_\tau\MP_{g,r}^d(X)$) qui contient les courbes réelles non séparantes.

Remarquons que sur $\RR_\tau\MM_{g,r}^d(X)$, nous avons une fonction localement constante $k_-$ à valeurs dans $\{0,\ldots,g+1\}$ qui associe à la courbe $[u,J_\sur,J,\underline{z}]$ le nombre de composantes de $\RR_{c_\sur}\surg$ sur lesquelles $\RR E_u$ n'est pas orientable.

Considérons sur $\RR_\tau\MP_{g,r}^d(X)$ le fibré en droites réelles $O_X$ de fibre $O_{\RR E_u}$ au-dessus de $(u,J_\sur,J,\underline{z})$ (voir \S \ref{corpart}). 

Enfin, lorsque la courbe $(\surg,c_\sur)$ est séparante, c'est-à-dire lorsque $\surg\setminus\RR\surg$ a deux composantes connexes, notons $H_{c_\sur}$ l'ensemble constitué des deux orientations de $\RR_{c_\sur}\surg$ provenant de l'orientation fixée sur $\surg$ et $\RR\surg^-$ la réunion des composantes de $\RR\surg$ au-dessus dequelles $\RR E_u$ n'est pas orientable. Posons $H$ (resp. $\dethzero$) le fibré en droites réelles sur $\RR^{sep}_\tau\MP_{g,r}^d(X)$ associé au $\ZZ/2\ZZ$-fibré principal de fibre $H_{c_\sur}$ (resp. $\det(H_0(\RR\surg^-,\RR))$) au-dessus de $(u,J_\sur,J,\underline{z})$.

Tous ces fibrés en droites sur $\RR^{sep}_\tau\MP_{g,r}^d(X)$ induisent des fibrés sur $\RR^{sep}_\tau\MM_{g,r}^d(X)$ que nous noterons de la même façon. Le Corollaire suivant découle de la Proposition \ref{fredholm} et du Corollaire \ref{actionsepa}.

\begin{Corollaire}\label{MODSEP}
  Le fibré $\Pi^*\ddet(\pi)$ sur $\RR^{sep}_\tau\MM_{g,r}^d(X)$ est canoniquement isomorphe au fibré en droites
\[
  \dethun^{\otimes n-1}\otimes \dethzero\otimes
  H^{\otimes\frac{c_1(X)d+k_-}{2}} \otimes \PPP_X^+ \otimes O_X.
\]\qed
\end{Corollaire}

Le cas des courbes séparantes a déjà été traité par plusieurs auteurs. Welschinger dans \cite{wel3} a calculé la première classe de Stiefel-Whitney du compactifié de l'espace $\RR_\tau \MM_{0,\frac{c_1(X)d}{2}}^d(X)$ lorsque $(X,\omega,c_X)$ est une variété symplectique réelle fortement semi-positive dont la partie réelle est $Pin^\pm$. Le Corollaire \ref{MODSEP} et le Lemme \ref{tang} nous permettent de retrouver une partie de ce résultat (voir Corollaire \ref{genrz}).

\begin{Lemme}\label{tang}
  \begin{enumerate}
  \item   Si $\tau$ n'est pas l'identité, alors les fibrés $L_r$ et $H$ sur $\RR^{sep}_\tau\MM_{g,r}^d(X)$ sont orientables.
  \item   Les fibrés $L_r$ et $H$ sont orientables sur les composantes connexes de $\RR^{sep}_{\id}\MM_{g,r}^d(X)$ qui regroupent les courbes dont au moins une composante de la partie réelle contient trois points marqués.
  \end{enumerate}
\end{Lemme}

\begin{proof}
  Prenons $[u_0,(J_\sur)_0,J_0,\underline{z}_0]\in \RR_\tau\MM_{g,r}^d(X)$ et choisissons un lacet $\gamma : [0,1] \rightarrow \RR_\tau\MM_{g,r}^d(X)$ tel que $\gamma(0) = \gamma(1) = [u_0,(J_\sur)_0,J_0,\underline{z}_0]$. Ce lacet se relève en un chemin $(u_t,(J_\sur)_t,J_t,\underline{z}_t)_{t\in[0,1]}$ dans $\RR_\tau\MP_{g,r}^d(X)$ joignant $(u_0,(J_\sur)_0,J_0,\underline{z}_0)$ à $(u_1,(J_\sur)_1,J_1,\underline{z}_1) = \varphi.(u_0,(J_\sur)_0,J_0,\underline{z}_0)$, où $\varphi\in\rdif$. On obtient alors une trivialisation de $L_r$ (resp. de $H$) au-dessus de $[0,1]$ en choisissant une famille continue $\phi_t\in\rdif$, $t\in[0,1]$, telle que pour tout $t\in[0,1]$, $\phi_t^{-1}(\underline{z}_t) = \underline{z}_0$. La valeur de la première classe de Stiefel-Whitney du fibré $L_r$ (resp. de $H$) appliquée à $\gamma$ est alors donnée par le signe de l'action de $\d\phi_1^{-1} \circ \d\varphi$ sur les orientations de $\RR \left( T_{\underline{z}_0} \surg^r\right)$ (resp. de $H_{c_\sur}$). 

Dans le premier cas, $\underline{z}_0$ contient au moins une paire de points complexes conjugués. En particulier, $\phi_1^{-1}\circ\varphi$ fixe ces deux points. Comme les courbes considérées sont séparantes, cela implique que $\phi_1^{-1}\circ\varphi$ préserve les deux orientations de $\RR\surg$ provenant de l'orientation fixée sur $\surg$. Ainsi, $\phi_1^{-1}\circ\varphi$ préserve les orientations de $\RR \left( T_{\underline{z}_0} \surg^r\right)$ et de $H_{c_\sur}$.

Dans le second cas, $\phi_1^{-1}\circ\varphi$ fixe trois points sur une des composantes de $\RR\surg$. En particulier $\phi_1^{-1}\circ\varphi$ préserve les orientations de cette composante. \`A nouveau, $\phi_1^{-1}\circ\varphi$ préserve les orientations de $\RR \left( T_{\underline{z}_0} \surg^r\right)$ et de $H_{c_\sur}$.
\end{proof}

\begin{Rem}
  En général, le fibré $H$ n'est pas orientable. Considérons par exemple $\RR^{sep}\MM_{0,0}^1(\CC P^2)$. La première classe de Stiefel-Whitney de $H$ n'est pas nulle au-dessus d'un faisceau de droites réelles passant par un point fixé.
\end{Rem}

\begin{Corollaire}\label{genrz}
  Soit $(X,\omega,c_X)$ une variété symplectique réelle de dimension au moins $4$. Supposons que la partie réelle de $X$ admet une structure $Spin$. Alors pour tout $d\in H_2(X,\ZZ)$ tel que $c_1(X)d \geq 0$, pour tout $r\geq 3$, pout toute permutation $\tau$ de $\{1,\ldots,r\}$ d'ordre $2$ admettant au moins un point fixe, et pour tout $J\in\RR J^{reg}_\omega(X)$, la variété $\RR_\tau\MM_{0,r}^d(X,J)$ est orientable.\qed
\end{Corollaire}

 \subsubsection{Le cas $Spin$}\label{hyperpar}

Pour $d\in H_2(X,\ZZ)$ tel que $c_1(X)d$ est pair, notons $\RR_\tau\MM_{g,r}^{d,0}(X)$ (resp. $\RR_\tau\MP_{g,r}^{d,0}(X)$) l'union des composantes de $\RR_\tau\MM_{g,r}^d(X)$ (resp. $\RR_\tau\MP_{g,r}^{d,0}(X)$) sur lesquelles nous avons $w_1(\RR E_u) = 0$ et $\RR\surg\neq \emptyset$. Considérons $\RR_\tau\MM_{g,r}^{d,Spin}(X)$ (resp. $\RR_\tau\MP_{g,r}^{d,Spin}(X)$) le revêtement à $2^{g+k-1}$ feuillets associé au $(\ZZ/2\ZZ)^{g+k-1}$-fibré principal sur $\RR_\tau\MM_{g,r}^{d,0}(X)$ (resp. $\RR_\tau\MP_{g,r}^{d,0}(X)$) dont la fibre en $[u,J_\sur,J,\underline{z}]$ est formée des structures $Spin$ réelles sur $E_u$. Nous disposons sur $\RR_\tau\MM_{g,r}^{d,Spin}(X)$ du fibré en droites réelles $\dethunw$ associé au $\ZZ/2\ZZ$-fibré principal dont la fibre au-dessus de $([u,J_\sur,J,\underline{z}],\xi)$ est composée des deux semi-orientations de $\RR E_u$ associées à $\xi$ (voir \S \ref{parspin}). D'autre part, sur chaque composante connexe de $\RR_\tau\MP_{g,r}^{d,Spin}(X)$, toutes les structures $Spin$ réelles ont même première classe de Stiefel-Whitney. On définit donc le fibré $\dethunww$ sur $\RR_\tau\MP_{g,r}^{d,Spin}(X)$ de fibre $\dd\bigotimes_{w([\RR\surg]_i) = 1}H^1((\RR\surg)_i,\RR)$ au-dessus $(u,J_\sur,J,\underline{z},\xi)$, où $\xi$ est de première classe de Stiefel-Whitney $w$. On note encore $\dethunww$ le fibré induit sur $\RR_\tau\MM_{g,r}^{d,Spin}(X)$.

Le Corollaire suivant découle de la Proposition \ref{fredholm} et du Corollaire \ref{actionspin}.

\begin{Corollaire}\label{MODNSEP}
  Soit $d\in H_2(X,\ZZ)$ tel que $c_1(X)d$ est pair. Le fibré $\Pi^*\ddet(\pi)$ sur $\RR_\tau\MM_{g,r}^{d,Spin}(X)$ est canoniquement isomorphe au fibré
\[
\dethun^{\otimes n-1}\otimes \dethunww \otimes(\dethunw)^{\otimes 1-g}\otimes \PPP_X^\pm.
\]\qed
\end{Corollaire}

 Il existe une famille de variétés donnée dans le Lemme \ref{hyperlemme} pour laquelle on peut obtenir un résultat plus précis (voir Corollaire \ref{hypercorollaire}).

 \begin{Lemme}\label{hyperlemme}
   Soit $X_\delta \subset \cp[N]$ une hypersurface complexe lisse de degré $\delta$. Supposons que $N$ est supérieur ou égal à quatre et vérifie $N = 0$ ou $3 \mod 4$, que $\delta = N+1\mod 2$ et que $X_\delta$ est de partie réelle non vide. Alors $X_\delta$ admet une unique structure $Spin$ réelle $\xi_0$, et sa partie réelle est $Spin$. De plus, si $\delta = N+1\mod 4$, alors la première classe de Stiefel-Whitney de $\xi_0$ est nulle.
 \end{Lemme}

 \begin{proof}
   Notons $i$ l'inclusion de $X_\delta$ dans $\cp[N]$. D'après la formule d'adjonction, nous avons $c_1(X_\delta) +\delta i^*h = (N+1) i^*h$, où $h$ est le générateur de $H^2(\cp[N],\ZZ)$. Comme $\delta = N+1 \mod 2$, $c_1(X_\delta)$ est divisible par deux dans $H^2(X_\delta,\ZZ)$. En particulier $X_\delta$ admet une structure $Spin$. D'autre part, d'après le Théorème de l'hyperplan de Lefschetz, on a $H^1(X_\delta,\ZZ/2\ZZ) \cong H^1(\cp[N],\ZZ/2\ZZ) = \{0\}$. Ainsi $X_\delta$ admet une unique structure $Spin$, qui est donc réelle, et que l'on note $\xi_0$.

L'existence de $\xi_0$ implique que $\RR X_\delta$ est orientable. De nouveau grâce à la formule d'adjonction on a $w_2(\RR X_\delta) = \dd\frac{N(N+1)}{2} t^2$, où $t$ est le générateur de $H^1(\rp[N],\ZZ/2\ZZ)$. Ainsi, sous les hypothèses faites sur $N$, $w_2(\RR X_\delta) = 0$ et $\RR X_\delta$ est $Spin$.   
 \end{proof}

Notons que $\cp[N-1]$ vérifie les hypothèses du Lemme \ref{hyperlemme} lorsque $N$ est pair. En appliquant le Théorème \ref{MODNSEP} et la Proposition \ref{fredholm} aux cas mis en avant dans le Lemme \ref{hyperlemme}, on obtient le Corollaire \ref{hypercorollaire} suivant.

\begin{Corollaire}\label{hypercorollaire}
Soit $X_\delta$ une hypersurface lisse de degré $\delta$ de $\cp[N]$, $N\geq 4$, $N = 0$ ou $3 \mod 4$, de partie réelle non vide et vérifiant $\delta = N+1\mod 4$ et $\delta \leq N+1$. Soit $d\in H_2(X_\delta,\ZZ)$, $r\geq 1$ et $\tau$ une permutation de $\{1,\ldots,r\}$ d'ordre $2$ et ayant au moins un point fixe. Pour $J \in\RR J_{\omega_{Fub}}^{reg}(X_\delta)$ on a
 \[
 w_1(\RR_\tau\MM_{g,r}^d(X_\delta,J)) = w_1(L_r) + (\delta-1) w_1(\dethun).
 \]
\end{Corollaire}

\begin{proof}
  D'après le Lemme \ref{hyperlemme}, la partie réelle de $X_\delta$ est $Spin$, donc la première classe de Stiefel-Whitney de $(\dethunw)^{\otimes 1-g}\otimes \PPP^{\pm}_X$ est nulle. D'après ce même Lemme, $X_\delta$ admet une unique structure $Spin$ réelle qui est de plus de première classe de Stiefel-Whitney nulle. Cette structure fournit une section du revêtement $\RR_\tau\MM_{g,r}^{d,Spin}(X_\delta)\rightarrow\MM_{g,r}^{d}(X_\delta)$. On conclut alors grâce au Théorème \ref{MODNSEP} et à la Proposition \ref{fredholm}.
\end{proof}

Donnons l'analogue du Lemme \ref{tang} qui donne une condition pour calculer le terme $w_1(L_r)$ du Corollaire \ref{hypercorollaire}.

\begin{Lemme}\label{tangs}
   Le fibré $L_r$ est orientable sur les composantes connexes de $\RR_\tau\MM_{g,r}^{d,Spin}(X)$ qui regroupent les courbes dont toutes les composantes de la partie réelle contiennent trois points marqués.\qed
\end{Lemme}

\subsection{Polarisations}

Le but de ce paragraphe est de déduire du Théorème \ref{metath} des résultats sur la première classe de Stiefel-Whitney des espaces $\RR \MM_{g}^d(X,J)$. Nous y parvenons en prenant une polarisation sur $(X,\omega,c_X)$ (\S \ref{defpol}). En effet, celle-ci induit une polarisation sur certaines courbes (\S \ref{transpol}) ce qui nous donne le Théorème \ref{caspol}. Dans le \S \ref{existpol}, nous discutons l'existence de polarisations sur $(X,\omega,c_X)$.

\subsubsection{Diviseurs admissibles}\label{defpol}

Considérons $(X,\omega,c_X)$ une variété symplectique réelle de dimension $2n$, $n\geq 2$.

\begin{Definition}
  Un diviseur réel sur $(X,\omega,c_X)$ est une somme finie $D = \sum a_V V$, où $a_V\in\ZZ$ et $V\subset X$ est une sous-variété symplectique de dimension $2n-2$, telle que $(c_X)_* D = D$.
\end{Definition}

\begin{Lemme}
Soit $D$ un diviseur réel sur $(X,\omega,c_X)$. Supposons qu'il existe un élément $J_0\in \RR J_D(X) := \{J\in \RR J_{\omega}(X) \ |\ \forall V\in D,\ V\text{ est } J\text{- holomorphe}\}$ qui soit compatible avec $\omega$. Alors $\RR J_D(X)$ est une sous-variété de Banach contractile de $\RR J_{\omega}(X)$. Son espace tangent en un point $J\in\RR J_D(X)$ est isomorphe à $\{W\in \Gamma^l(X,\Lambda^{0,1}X\otimes TX)_{+1} |\ \forall V\in\ D,\ W(TV)\subset TV\}$.
\end{Lemme}

\begin{proof}
Notons $||.||_{J_0}$ la norme induite par la métrique $\omega(.,J_0.)$ sur $\Gamma^l(X,End_{\RR}(TX))$ et posons $\WW = \{W\in\Gamma^l(X,End_{\RR}(TX))|\ J_0W = -WJ_0\text{ et } ||W||_{J_0}<1\}$. Cet ensemble est muni d'une involution $(c_X)^*$. La transformation de Sévennec
\[
J\in J_{\omega}(X) \rightarrow (J+J_0)^{-1}\circ(J-J_0) \in \WW
\]
est un difféomorphisme $\ZZ/2\ZZ$-équivariant (voir la Proposition 1.1 de \cite{wel1}). L'image de $\RR J_D(X)$ sous cette application est égale à
\[
\RR \WW \cap \{ W\in \WW\ | \ \forall V\in D,\ \forall x\in V,\ W(T_x V) \subset T_x V\}.
\]
Ce dernier espace est une sous-variété de Banach contractile de $\WW$, ce qui prouve le Lemme.
\end{proof}

\begin{Definition}
On dira qu'un diviseur réel $D$ est admissible si
\begin{itemize}
\item toutes ses composantes s'intersectent transversalement,
\item toutes ses composantes apparaissent avec une multiplicité $\pm 1$,
\item et s'il existe un élément dans $\RR J_D(X)$ qui soit compatible avec $\omega$.
\end{itemize}
On notera $D_r$ la partie de $D$ formée des composantes stables par $c_X$.

 Un diviseur admissible $D$ tel que $D^{\pd} = c_1(X)\in H^2(X,\ZZ)$ et $(\RR D_r)^{\pd} = w_1(\RR X)\in H^1(\RR X,\ZZ/2\ZZ)$ est appelé une polarisation de $(X,\omega,c_X)$, qui sera alors dite polarisée.
\end{Definition}

\begin{Rem}
Pour un diviseur réel admissible quelconque, $\RR D$ n'est pas égal à $\RR D_r$. En effet, deux composantes complexes conjuguées de $D$ s'intersectent en une sous-variété symplectique de dimension $2n-4$ stable par $c_X$ qui peut avoir une partie réelle non vide de dimension $n-2$.
\end{Rem}

\subsubsection{Transversalité}\label{transpol}

Fixons un diviseur $D$ admissible sur $(X,\omega,c_X)$. Nous commençons par décrire dans le Lemme \ref{dl} le comportement local d'une courbe pseudo-holomorphe réelle $u : (\surg,c_\sur,J_\sur)\rightarrow (X,c_X,J)$ au voisinage d'un point d'intersection avec $D$. Nous prenons pour ceci des coordonnées locales de la forme suivante.

\begin{enumerate}
\item \label{casun} Si $V$ est une composante du support de $D_r$, pour chaque $p\in u^{-1}(\RR V)\cap \RR\surg$ (resp. pour chaque $q,c_\sur(q)\in u^{-1}(\RR V)\setminus \RR\surg$) on prend une carte locale $J_\sur$-holomorphe centrée en $p$ et $\ZZ/2\ZZ$-équivariante (resp. deux cartes locales $J_{\sur}$-holomorphes complexes conjuguées centrées en $q$ et $c_\sur(q)$), et des coordonnées locales $(v_1,\ldots,v_{2n-2},x,y)\in \RR^{2n}$ centrées en $u(p)$ (resp. en $u(q) = u(c_\sur(q))$), $\ZZ/2\ZZ$-équivariantes, c'est-à-dire telles que $(v_1,\ldots,v_{2n-2},x,y)\circ c_X = (v_1,\ldots,-v_{2n-2},x,-y)$, et telles que $(v_1,\ldots,v_{2n-2})$ soient des coordonnées locales pour $V$, $x = y = 0$ et $J\frac{\partial}{\partial x} = \frac{\partial}{\partial y}$ le long de $V$,
\item \label{casdeux}Si $V$ est une composante du support de $D-D_r$, pour chaque $p\in u^{-1}(V \cap \RR X)\cap \RR\surg$ (resp. pour chaque $q,c_\sur(q)\in u^{-1}(V \cap \RR X)\setminus \RR\surg$) on prend une carte locale $J_\sur$-holomorphe centrée en $p$ et $\ZZ/2\ZZ$-équivariante (resp. deux cartes locales $J_{\sur}$-holomorphes complexes conjuguées centrées en $q$ et $c_\sur(q)$), et des coordonnées locales $(v_1,\ldots,v_{2n-2},x,y)\in \RR^{2n}$ centrées en $u(p)$ (resp. en $u(q) = u(c_\sur(q))$), $\ZZ/2\ZZ$-équivariantes, c'est-à-dire telles que $(v_1,\ldots,v_{2n-2},x,y)\circ c_X = (v_1,\ldots,-v_{2n-4}, x, -y , v_{2n-3}, -v_{2n-2})$, et telles que $(v_1,\ldots,v_{2n-2})$ soient des coordonnées locales pour $V$, $x = y = 0$ et $J\frac{\partial}{\partial x} = \frac{\partial}{\partial y}$ le long de $V$,
\item \label{castrois} Si $V$ est une composante du support de $D$, pour chaque $q,c_\sur(q) \in u^{-1}(V\setminus (V\cap \RR X))\setminus \RR\surg$, on prend des cartes $J_\sur$-holomorphes complexes conjuguées centrées en $q$ et $c_\sur(q)$, et des coordonnées locales $(v_1,\ldots, v_{2n-2},x,y)$ centrées en $u(q)$ et $(v_1,\ldots,-v_{2n-2},x,-y)\circ c_X$ centrées en $u(\overline{q})$ telles que $(v_1,\ldots,v_{2n-2})$ soient des coordonnées locales de $V$ au voisinage de $u(q)$, $x=y=0$ et $J\frac{\partial}{\partial x} = \frac{\partial}{\partial y}$ le long de $V$,
\end{enumerate}

Nous pouvons maintenant énoncer le Lemme \ref{dl} qui est la version réelle du Lemme $3.4$ de \cite{IonelP}.

\begin{Lemme}\label{dl}
  Soit $J\in\RR J_D(X)$ et fixons une composante $V$ du support de $D$. Soit $u: (\surg,c_\sur,J_{\sur})\rightarrow (X,c_X,J)$ une courbe $J$-holomorphe réelle qui n'est pas contenue dans $V$. Alors, dans les coordonnées locales données ci-dessus,
\begin{itemize}
  \item[Cas \ref{casun} : ] il existe un entier $m\geq 1$ et un réel $a\in\RR^*$ (resp. un complexe $a\in\CC^*$) vérifiant
\[
\begin{array}{c}
u(z) = (O(|z|),az^m + O(|z^{m+1}|)),\text{ au voisinage de }p\\
\text{resp. } u(z) = (O(|z|),az^m + O(|z^{m+1}|)),\text{ au voisinage de }q\\
\text{et } u(z) = (O(|z|),\overline{a}z^m + O(|z^{m+1}|)),\text{ au voisinage de }c_\sur(q)\\
\end{array}
\]

\item[Cas \ref{casdeux} : ] il existe un entier $m\geq 1$ et un réel $a\in \RR^*$ (resp. un complexe $a\in\CC^*$) vérifiant
\[
\begin{array}{c}
u(z) = (O(|z|),az^m + O(|z^{m+1}|),az^m + O(|z^{m+1}|)),\text{ au voisinage de }p\\
\text{resp. }u(z) = (O(|z|), O(|z|),az^m + O(|z^{m+1}|))\text{ au voisinage de }q,\\
\text{et } u(z) = (O(|z|),\overline{a}z^m + O(|z^{m+1}|), O(|z|))\text{ au voisinage de }\overline{q}.
\end{array}
\]
\item[Cas \ref{castrois} : ]  il existe un entier $m\geq 1$ et un complexe $a\in\CC^*$ vérifiant
\[
\begin{array}{c}
u(z) = (O(|z|),az^m + O(|z^{m+1}|))\text{ au voisinage de }q,\\
  u(z) = (O(|z|),\overline{a}z^m + O(|z^{m+1}|))\text{ au voisinage de }\overline{q}.
\end{array}
\]
  \end{itemize} 
De plus, dans tous les cas, l'entier $m$ est défini indépendamment des choix de coordonnées locales. On le notera $mult_V(u,p)$ (resp. $mult_V(u,q)$).\qed
\end{Lemme}

En particulier, si $J\in\RR J_D(X)$, une courbe $J$-holomorphe qui n'est pas contenue dans $D$ intersecte le diviseur positivement en un nombre fini de points. De plus, il n'est pas difficile de voir qu'un point d'intersection de multiplicité $m$ avec une composante $V$ compte pour $m$ dans le calcul de l'intersection homologique $u(\surg)\bullet V$.

\begin{Proposition}\label{transv}
Soit $D$ un diviseur admissible sur $(X,\omega,c_X)$. La restriction de la projection $\pi : \RR\MM_g^d(X) \rightarrow \RR J_{\omega}(X)$ à l'ouvert $\{[u,J_\sur,J]\in\RR \MM_g^d(X)\ |\ u(\surg)\not\subset D\}$ est transverse à $\RR J_D(X)$.
\end{Proposition}

La démonstration suit celle de la Proposition $1.11$ de \cite{wel1}, nous passons donc assez rapidement.

\begin{proof}
 Soit $J\in\RR J_D(X)$ et $(u,J_\sur,J)\in\RR \MP_g^d(X)$ telle que $u(\surg)$ n'est pas contenue dans $D$. Notons $N_u$ le fibré normal à $u$ et $D^N : L^{k,p}(\surg,N_u)\rightarrow L^{k-1,p}(\surg,N_u)$ l'opérateur de Cauchy-Riemann généralisé induit par $D$ sur $N_u$ (voir \S 1.5 de \cite{shev} et \S 1.4 de \cite{wel1}). Le conoyau de $\d_{[u,J_\sur,J]} \pi$ est canoniquement isomorphe à $H^1_{D^N}(\surg,N_u)_{+1}$. Soit $\psi\in (H_{D^N}^1(\surg,N_u)_{+1})^* = H_{D^N}^0(\surg,K_{\sur}\otimes (N_u)^*)_{-1}$ non nul. Il nous faut trouver $\dot{J}\in T_J\RR J_D(X)$ tel que $<\psi,\dot{J}\circ\d u\circ J_{\sur}>\neq 0$.

D'après le Lemme \ref{dl}, il existe un ouvert $U\subset\surg$ non vide tel que $u(U)\cap D = \emptyset$. On peut de plus supposer que $c_{\sur}(U) = U$ et que la restriction de $u$ à $U$ est un plongement car $u$ est injective quelque part. Comme les zéros de $\psi$ sont isolés, on impose enfin que $\psi$ ne s'annule pas sur $U$. Il existe alors une section réelle $\alpha$ de $\Lambda^{0,1}\surg\otimes E_u$ à support dans $U$ et telle que $<\psi,\alpha>\neq 0$. Comme la restriction de $u$ à $U$ est un plongement, on peut prendre alors $\dot{J}\in T_J\RR J_{\omega}(X)$ à support dans un voisinage de $U$ et vérifiant $\dot{J}\circ\d u\circ J_{\sur} = \alpha$. Il est alors immédiat que $\dot{J}\in T_J\RR J_D(X)$ et vérifie $<\psi,\dot{J}\circ\d u\circ J_{\sur}>\neq 0$.
\end{proof}

D'après la Proposition \ref{transv}, l'intersection $\RR \MM_g^d(X,D) := \pi^{-1}(\RR J_D(X))\cap\{[u,J_\sur,J]\in\RR \MM_g^d(X)\ |\ u(\surg)\not\subset D\}$ est une sous-variété de Banach de $\RR \MM_g^d(X)$.

Supposons à partir de maintenant que $D$ est une polarisation de $(X,\omega,c_X)$. Décomposons $D$ en $D^+-D^-$, où $D^+$ et $D^-$ sont deux diviseurs effectifs. On dira qu'un quadruplet d'entiers naturels $\ud$ est adapté à $(D,d)$ si $d_1+2d_3 = d\bullet D^{+}$ et $d_{2} + 2d_{4} = d\bullet D^-$. En reprenant les mêmes notations qu'au \S \ref{espmod}, posons
\[
\begin{array}{c l}
\RR \MM_{g,pol}^d(X,D) = \{[u,J_\sur,J,\ux]\in\RR\MM_{g,pol}^d(X)\ | & [u,J_\sur,J]\in\RR\MM_g^d(X,D)\\
& \text{et }\ux\in\surg^{(\ud)}\text{ avec }\ud\text{ adapté à } (D,d)\}\\
& \subset \Pi^{-1}(\RR\MM_g^d(X,D)),  
\end{array}
\]
où $\Pi : \RR\MM_{g,pol}^d(X) \rightarrow \RR \MM_g^d(X)$ est l'application d'oubli de la polarisation.

Pour tout quadruplet d'entiers $\ud$, le produit symétrique $\surg^{(\ud)}$ admet un revêtement $\surg^{\ud}$ non ramifié de groupe de transformations $\SSS_{\ud}$ (voir \S \ref{diviseur}). La variété $\RR \MM_{g,pol}^d(X,D)$ hérite aussi d'un revêtement non ramifié $\eta :\RR \widetilde{\MM}_{g,pol}^d(X,D)\rightarrow\RR \MM_{g,pol}^d(X,D)$ donné par la numérotation des points marqués et dont le groupe est $\SSS_{\ud}$ au-dessus des composantes contenant les courbes dont les points marqués sont dans $\surg^{(\ud)}$.

Nous disposons alors d'une application d'évaluation
\[
ev_{pol} : \RR\widetilde{\MM}_{g,pol}^d(X,D) \rightarrow X^{pol}
\]
où $X^{pol} = \dd\bigsqcup_{
  \begin{subarray}{c}
r_+ +2s_+ = d\bullet D^+\\ r_- +2s_-=d\bullet D^-    
  \end{subarray}
} (\RR X)^{r_+ + r_-}\times (\RR X^2)^{s_++s_-}$. Notons enfin pour tout $\ud = (r_+,r_-,s_+,s_-)\in\NN^4$, $D^{\ud} =  (\RR D^+)^{r_+}\times(\RR D^-)^{r_-}\times (\RR (D^+)^2)^{s_+}\times(\RR(D^-)^2)^{s_-}$ et $D^{pol} = \dd\bigsqcup_{  \begin{subarray}{c}
\ud \text{ adapté à }(D,d)
  \end{subarray}
}D^{\ud}\subset X^{pol}$.

\begin{Definition}
Soit $D$ une polarisation de $(X,\omega,c_X)$. On pose  
\[
\begin{array}{c l}
\RR \MM_g^d(X,D)^{\pitchfork} = \{[u,J_{\sur},J]\in\RR\MM_{g}^d(X,D)| & \forall V\in D,\ \forall p\in u(\surg)\cap V,\ mult_V(u,p) = 1 \\
& \forall V'\neq V \in D, u(\surg)\cap (V\cap V') = \emptyset\}.  
\end{array}
\]
\end{Definition}

\begin{Lemme}\label{evtrans}
Soit $D$ une polarisation de $(X,\omega,c_X)$.
  \begin{itemize}
  \item
\[
\eta(ev_{pol}^{-1}(D^{pol})) \subset \Pi^{-1}\RR \MM_g^d(X,D)^{\pitchfork}.
\]
\item Pour tous $(r_+,r_-,s_+,s_-)\in \NN^4$ et $(r'_+,r'_-,s'_+,s'_-)\in \NN^4$ distincts tels que $r_++2s_+ = r'_++2s'_+ = d\bullet D^+$ et $r_-+2s_- = r'_-+2s'_- = d\bullet D^-$, $ev_{pol}^{-1}(D^{(r_+,r_-,s_+,s_-)})\cap ev_{pol}^{-1}(D^{(r'_+,r'_-,s'_+,s'_-)}) = \emptyset$.
\item L'application $ev_{pol}$ est transverse à $D^{pol}$.
  \end{itemize}
\end{Lemme}

\begin{proof}
  Soit $[u,J_\sur,J,\ux]\in ev_{pol}^{-1}(D^{pol})$. On a alors
\[
\card(\ux) = d\bullet (D^{+}+ D^-) = \dd\sum_{p\in u^{-1}(D)}\sum_{V\in D^{\pm}} mult_V(u,p) \geq \card(\ux).
\]
La première égalité provient de la définition de $\RR\MM_{g,pol}^d(X,D)$. La seconde est une conséquence du Lemme \ref{dl}. L'inégalité n'est une égalité que lorsque $u$ est de multiplicité $1$ en chaque point d'intersection avec $D$ et que ceux-ci ne se trouvent pas sur deux composantes de $D$ distinctes. Ceci montre le premier point.

Le deuxième point est immédiat. Le troisième découle du premier.
\end{proof}

Le Lemme \ref{evtrans} implique que $\eta(ev_{pol}^{-1}(D^{pol}))$ est une sous-variété de $\Pi^{-1}\RR\MM_g^d(X,D)$. De plus, la restiction de $\Pi$ à $\eta(ev_{pol}^{-1})$ est un difféomorphisme sur l'ouvert $\RR \MM_g^d(X,D)^{\pitchfork}$ de $\RR \MM_g^d(X,D)$. L'inverse de $\Pi$ est donnée par l'intersection des éléments de $\RR \MM_g^d(X,D)^{\pitchfork}$ avec $D$. 

\begin{Rem}
Les éléments de $\RR\MM_g^d(X,D)^{\pitchfork}$ sont les courbes dont tous les points d'intersection avec $D$ sont de multiplicité $1$ et se trouvent sur une seule composante de $D$. En particulier, d'après le Lemme \ref{dl}, le complémentaire $\RR\MM_g^d(X,D)\setminus \RR\MM_g^d(X,D)^{\pitchfork}$ est formé d'une strate de codimension $1$ qui contient les courbes qui ont exactement un point d'intersection réel avec $D$ de multiplicité $2$ ou bien un point d'intersection réel de multiplicité $1$ avec deux composantes distinctes de $D$, et de strates de codimensions supérieures (voir aussi \cite{Ionel}).
\end{Rem}

\begin{Def}
  Soit $(X,\omega,c_X)$ une variété symplectique réelle de dimension $2n$ et polarisée par le diviseur $D$. Une section polarisante associée à $D$ est une section $s$ réelle de $\det_J(T X) = \Lambda_J^n T X$ pour un certain $J\in \RR J_D(X)$ vérifiant
\begin{itemize}
\item $s^{-1}(\{0\}) = D$,
\item $s$ s'annule transversalement le long de chaque composante de $D$ sauf aux endroits où deux telles composantes s'intersectent,
\item le signe de chaque composante de $D$ est donné par la comparaison de l'orientation induite par $s$ avec celle venant de la forme symplectique.
\end{itemize}
\end{Def}

Nous discutons l'existence d'une telle section dans le \S \ref{existpol}.

Nous utilisons les fibrés définis par le Lemme \ref{fibr} dans le Théorème suivant.

\begin{Theoreme}\label{caspol}
Soit $(X,\omega,c_X)$ une variété symplectique réelle polarisée par le diviseur $D$. Le fibré en droites réelles $\ddet(\pi : \RR\MM_g^d(X,D)^{\pitchfork}\rightarrow \RR J_D(X))$ est canoniquement isomorphe à
\[
\PPP^+_X\otimes \Pi_*\DD\otimes \dethun^{\otimes n-1}\otimes \Pi_* T_{pol}
\]
au-dessus de $\RR\MM_g^d(X,D)^{\pitchfork}$.

De plus, si $(X,\omega,c_X)$ admet une section polarisante, alors le choix d'une telle section oriente le fibré $\DD$.
\end{Theoreme}

\begin{proof}
La première partie du résultat provient directement du Théorème \ref{metath}.

Supposons donc que $(X,\omega,c_X)$ admet une section polarisante $s_{J_0}$ de $\det_{J_0}(T X)$, pour un certain $J_0\in \RR J_D(X)$. Comme $\RR J_D(X)$ est contractile, il existe une famille continue d'isomorphismes $F_J : (\det_{J_0}(T X),\d c_X)\rightarrow (\det_{J}(T X),\d c_X)$ uniquement déterminée à homotopie près. On obtient ainsi une famille continue de section réelles polarisantes $(s_J = F_J(s_{J_0}))_{J\in\RR J_D(X)}$ de $\det_J(T X)$.

Prenons maintenant $[u,J_{\sur},J]\in\RR\MM_g^d(X,D)^{\pitchfork}$. Notons $\RR\Gamma_D(\det(E_u))$ l'espace des sections réelles $s$ de $\det(E_u)$ s'annulant seulement et transversalement aux points de $D_u =u^{-1}(u(\surg)\cap D)$ et dont l'indice d'annulation en un point est donné par le signe de ce point dans $D_u$. En particulier, la section $s_J$ induit un élément $s_u$ de $\RR\Gamma_D(\det(E_u))$. Choisissons une famille de coordonnées locales $J_{\sur}$-holomorphes $(U_x,\zeta_x)$, centrées au points $x\in D_u$ et $\ZZ/2\ZZ$-équivariantes, ainsi qu'une famille de trivialisations locales $\phi_x$ de $\det(E_u)$ sur les ouverts $U_x$.  On peut alors homotoper $s_u$ dans $\RR\Gamma_D(\det(E_u))$ en une section réelle $\tilde{s_u}$ de $\det(E_u)$ telle que dans la trivialisation $\phi_x$, $x\in D_u$, on ait $\tilde{s_u} = \pm\zeta_x$ si $mult_{D_u}(x) = +1$ et $\tilde{s_u} = \pm\overline{\zeta_x}$ si $mult_{D_u}(x) = -1$. En effet, soit $x\in D_u$ de multiplicité $+1$. Alors $\d_x s_u$ est de déterminant strictement positif. Si les deux valeurs propres de $\d_x s_u$ ne sont pas réelles positives, alors $(1-t)s_u + t\zeta_x$, $t\in [0,1]$, fournit une homotopie localement autour de $x$ que l'on prolonge sans problème. Si les deux valeurs propres de $\d_x s_u$ sont réelles positives, alors on homotope $s_u$ à $-\zeta_x$ sur $U_x$. Si $x$ est de multiplicité $-1$, on remplace localement $s_u$ par son conjugué et on applique le même procédé que ci-dessus.

On pose enfin $\tilde{\tilde{s_u}} = \tilde{s_u}$ sur $\surg\setminus \dd\bigcup_{\begin{subarray}{c} x\in D_u,\\ mult_{D_u}(x) = -1\end{subarray}}U_x$ et $\tilde{\tilde{s_u}} = \dd\frac{1}{\pm\zeta_x}$ sur $U_x$ pour $x\in D_u$ de multiplicité $-1$. On obtient alors un opérateur de Cauchy-Riemann polarisé sur $\det(E_u)$ de la façon suivante. Sur $\surg\setminus \dd\bigcup_{x\in D_u} U_x$, on a $\DB = \tilde{\tilde{s_u}}\DB_0 (\frac{1}{\tilde{\tilde{s_u}}})$ et sur $U_x$, $\DB = \frac{\tilde{\tilde{s_u}}}{\zeta_x^{mult_{D_u}(x)}}\DB_0 (\frac{\zeta_x^{mult_{D_u}(x)}}{\tilde{\tilde{s_u}}})$, où $\DB_0$ dénote l'opérateur standard sur les fonctions sur $\surg$. De plus, la section $\tilde{\tilde{s_u}}$ est méromorphe pour cet opérateur, de diviseur $D_u$. Notons de plus que l'opérateur obtenu est le seul ayant $\tilde{\tilde{s_u}}$ comme polarisation. On obtient donc une orientation de $\DD_{D_u}$. De plus, celle-ci ne dépend pas des choix de trivialisations locales faits.

Le fibré $\DD$ sur $\RR\MM_g^d(X,D)^{\pitchfork}$ est ainsi orienté.
\end{proof}

\begin{Rem}
 On montre que lorsque $(X,\omega,c_X)$ admet une section polarisante, $\DDD$ est orientable et orienté. La question naturelle qui se pose est à quelle condition deux sections polarisantes donnent la même orientation ? Par exemple, si $\RR X$ est orientable, si les deux sections induisent la même orientation sur $\RR X$ elles donnent la même orientation de $\DDD$ au-dessus des composantes de l'espace des modules qui contiennent les courbes séparantes.
\end{Rem}

\subsubsection{Exemples}\label{existpol}

\begin{Lemme}\label{sectpol}
  Soit $(X,\omega,c_X)$ une variété symplectique réelle, de partie réelle non vide, polarisée par $D$ et telle que $H^1(X,\ZZ) = 0$. Alors $(X,\omega,c_X)$ admet une section polarisante.
\end{Lemme}

\begin{proof}
  Notons $D = \sum_{i=1}^r a_i V_i + \sum_{j=1}^s a_j (W_j +c_X(W_j))$, où les $V_i$ sont symplectiques et fixes par $c_X$, et les $W_j$ sont symplectiques avec $W_j\pitchfork c_X(W_j)$. Fixons $J\in\RR J_D(X)$.

Prenons $1\leq i\leq r$ et fixons un voisinage tubulaire équivariant $U_{i}$ de $V_i$, c'est-à-dire que $c_X(U_{i}) = U_{i}$ et qu'il existe un difféomorphisme $\ZZ/2\ZZ$-équivariant et préservant les orientations $\varphi : U_{i}\rightarrow \nu(N_{V_i})$, où $\nu(N_{V_i})$ est un voisinage de la section nulle du fibré normal à $V_i$ dans $X$, envoyant $V_i$ sur la section nulle (voir \cite{bredon}). Le fibré $\pi : N_{V_i}\rightarrow V_i$ hérite d'une structure complexe et d'une structure réelle venant de celles de $TX$ et $TV_i$, ce qui en fait un fibré en droites complexes muni d'une structure réelle au-dessus de $V_i$. Choisissons une famille de trivialisations locales $\ZZ/2\ZZ$-équivariante $f^k_i : (N_{V_i})_{|U_{V_i}^k} \rightarrow U_{V_i}^k\times \CC$ au-dessus d'une famille d'ouverts $\ZZ/2\ZZ$-équivariante $U_{V_i}^k\subset V_i$, $k\in I$, et notons $g^{kl}_i : U_{V_i}^k\cap U_{V_i}^l \rightarrow \CC^*$, $k,l\in I$, les changements de trivialisations. Posons $U_i^k = \varphi^{-1}((N_{V_i})_{|U_{V_i}^k})$, $s^k_i = pr_2\circ f_i^k \circ \varphi : U_i^k \rightarrow \CC$, et $U_i^{\infty} = X\setminus V_i$ et $s^{\infty} = 1 : U_i^{\infty}\rightarrow \CC$. Alors les fonctions $G_i^{kl} = \frac{s_i^l}{s_i^k} : U_i^k\cap U_i^l \rightarrow \CC^*$, $k,l\in I\cup\{\infty\}$, sont bien définies, et même pour $k,l\in I$, $G_i^{kl} = g_i^{kl}\circ\pi$. De plus, elles définissent un fibré en droites complexes $L_{V_i}$ qui admet une structure réelle naturelle et une section $s_i = (s_i^k)_{k\in I\cup \{\infty\}}$ réelle qui s'annule transversalement et positivement exactement le long de $V_i$.

Prenons maintenant $1\leq j \leq s$. On construit de la même façon que précédemment un fibré en droites complexes $M_{W_j}$ muni d'une section $t_j$ s'annulant transversalement et positivement le long de $W_j$. Alors le fibré en droites complexes $L_{W_j} = M_{W_j} \otimes \overline{c_X^* M_{W_j}}$ est naturellement muni d'une structure réelle et d'une section réelle $s_{j+r} = t_j\otimes \overline{t_j\circ c_X}$ qui s'annule transversalement et positivement le long de $W_j$ et $c_X(W_j)$ sauf au niveau de leur intersection.

Considérons le fibré en droites complexes $L_D = \dd\bigotimes_{i=1}^r L_{V_i}^{a_i}\dd\bigotimes_{j=1}^s L_{W_j}^{a_j}$. Il est muni d'une structure réelle et d'une section $\sigma = s_1\otimes\ldots\otimes s_{r+s}$ polarisante. De plus, il est isomorphe à $\det_J(T X)$ en tant que fibré en droites complexes, et sa partie réelle a même première classe de Stiefel-Whitney que $\RR X$. Comme $H^1(X,\ZZ) = 0$, on sait que $L_D$ et $\det_J(T X)$ sont isomorphes par un isomorphisme équivariant (voir par exemple le Corollaire 2.7 de \cite{oktel}). On obtient ainsi la section voulue pour $\det_J(T X)$.
\end{proof}

\begin{Rem}
On peut relâcher la condition de simple connexité. Soulignons toutefois que ce problème d'existence d'une section polarisante associée à une polarisation donnée est relié à la classification des fibrés en droites complexes munis d'une structure réelle sur $(X,c_X)$, problème qui est non trivial en général (voir par exemple le \S 2 de \cite{oktel}).
\end{Rem}

\paragraph{Hypersurfaces projectives}\label{parhyp}

Reprenons les notations du \S \ref{hyperpar} et considérons $X_{\delta}\xhookrightarrow{i}\cp[N]$, $N\geq 3$, une hypersurface lisse réelle de partie réelle non vide et de degré $\delta$. Nous construisons maintenant des polarisations de $X_{\delta}$ pour appliquer le Théorème \ref{caspol}. D'après la formule d'adjonction, on a d'une part $c_1(X_{\delta}) = (N+1-\delta)i^*h$, où $h\in H^2(\cp[n],\ZZ)$ est le générateur hyperplan et $w_1(\RR X_{\delta}) = (N+1-\delta)i^*t$, où $t\in H^1(\rp[n],\ZZ/2\ZZ)$ est le générateur.

Notons $Q_N$ la variété des quadriques projectives de $\cp[N]$. D'après le Théorème de Bertini, le sous-ensemble $U\subset Q_N$ formé des quadriques lisses qui intersectent $X_{\delta}$ transversalement est un ouvert de Zariski non vide. D'autre part, le sous-ensemble $V\subset Q_N$ formé des quadriques sans point réel et qui intersectent transversalement leur conjugué est un ouvert non vide pour la topologie analytique. Ainsi il existe une quadrique lisse $Q\subset\cp[N]$ telle que $(Q + conj (Q))\cap X_{\delta}$ soit un diviseur. Si $\delta = N+1 \mod 4$ nous obtenons la polarisation voulue sous la forme d'une somme de quadriques complexes conjuguées sans point réel en répétant ce procédé. Si $\delta = N\mod 2$, on utilise un hyperplan réel avec le signe adéquat en plus de ces quadriques complexes conjuguées. Enfin, lorsque $\delta = N+3 \mod 4$, on ajoute une quadrique lisse réelle sans point réel aux quadriques complexes conjuguées. L'existence de ces polarisations est à nouveau garantie par le Théorème de Bertini.

Grâce à ces polarisations, nous retrouvons en particulier le Corollaire \ref{hypercorollaire}, et nous précisons les cas non traités précédemment.

\begin{Corollaire}\label{hyperbis}
Soit $N\geq 3$ et $X_\delta$ une hypersurface lisse réelle de $\cp[N]$ de partie réelle non vide et de degré $1\leq \delta\leq N+1$. Soit $D_{\delta}$ une polarisation de $X_{\delta}$ donnée comme ci-dessus, $J\in \RR J^{reg}_{D_{\delta}}(X_{\delta})$, $g,r\in\NN$, $\tau$ une permutation de $\{1,\ldots,r\}$ ayant au moins un point fixe et $d\in H_2(X_{\delta},\ZZ)$. Si $\delta = N+1\mod 4$,
\[
w_1(\RR_{\tau} \MM_{g,r}^d(X,J)) = w_1(\PPP^{\pm}_{X_\delta}) + w_1(L_r) + Nw_1(\dethun).
\]
Si $\delta = N+3 \mod 4$ 
\[
w_1(\RR_{\tau} \MM_{g,r}^d(X,J)) = w_1(\PPP^{\pm}_{X_\delta}) + w_1((\Pi_{pol})_* T_{pol}) + w_1(L_r) + Nw_1(\dethun).
\]
Si $\delta = N \mod 2$
\[
w_1(\RR_{\tau} \MM_{g,r}^d(X,D_{\delta},J)^{\pitchfork}) = w_1(\PPP^{+}_{X_\delta}) + w_1((\Pi_{pol})_* T_{pol}) + w_1(L_r) + Nw_1(\dethun).
\]
\qed
\end{Corollaire}

On peut interpréter la classe $w_1(\Pi_* T_{pol})$ dans le cas $\delta = N+3 \mod 4$ du Corollaire \ref{hyperbis} de la façon suivante. La polarisation $D_{\delta}$ contient une quadrique réelle sans point réel notée $Q$. L'intersection d'un élément générique de $\RR_{\tau} \MM_{g,r}^d(X,J)$ avec $Q$ donne $d$ paires de points complexes conjugués. En choisissant un point dans chaque paire et en identifiant deux tels choix s'ils diffèrent d'un nombre pair de points, on obtient un revêtement double de $\RR_{\tau} \MM_{g,r}^d(X,J)$ dont la classe est $w_1(\Pi_* T_{pol})$.

D'autre part, dans le cas des hypersurfaces réelles de $\cp[N]$ il est aisé de déterminer lesquelles ont une partie réelle $Pin^{\pm}_{X_\delta}$ ce qui permet de déterminer des cas où $w_1(\PPP^{\pm})$ est nulle.

\paragraph{Hypersurfaces de Donaldson}

Nous généralisons le résultat du paragraphe précédent à toutes le variétés symplectiques réelles. Rappelons pour ceci une version équivariante d'un Lemme dû à Gompf \cite{gompf}.

\begin{Lemme}\label{ortho}
Soit $(X,\omega,c_X)$ une variété symplectique réelle et soient $V$ et $W$ deux sous-variétés symplectiques réelles de codimension $2$ dans $X$. Supposons de plus que $V$ et $W$ s'intersectent transversalement le long d'une sous-variété symplectique de codimension $4$ dont l'orientation venant de $\omega$ et celle venant de l'intersection coïncident. Alors il existe une isotopie $(\psi_t)_{t\in [0,1]}$ à support dans un voisinage de $V\cap W$ telle que
\begin{enumerate}
\item $\psi_0 = \id_X$,
\item $\forall t\in [0,1]$, $\psi_t\in Symp(X,\omega)$ et $\psi_t\circ c_X = c_X \circ \psi_t$,
\item $\psi_1(W)$ et $V$ s'intersectent orthogonalement pour $\omega$ le long de $W\cap V$.
\end{enumerate}
  \qed
\end{Lemme}

La démonstration de Gompf s'adapte sans problème au cas équivariant en prenant garde à considérer un voisinage tubulaire réel de l'intersection (voir \cite{bredon}). Nous renvoyons donc au Lemme 2.3 de \cite{gompf}.

\begin{Proposition}
  Toute variété symplectique réelle admet une polarisation et une section polarisante associée.
\end{Proposition}

\begin{proof}
  Soit $(X,\omega,c_X)$ une variété symplectique réelle. Soit $\omega'$ une forme symplectique entière sur $X$ proche de $\omega$ et vérifiant $(c_X)^*\omega'=-\omega'$. Fixons sur $X$ une structure presque complexe réelle $J$ compatible avec $\omega'$. Soit $L$ un fibré en droites complexes sur $X$ de classe de Chern $[\omega']$ sur lequel on fixe une métrique hermitienne. Notons $M = \det_J(T X)$ et fixons une métrique hermitienne réelle sur $M$. D'après Gayet \cite{gayet} (voir aussi \cite{donald}), il existe un entier $N\geq 2$ tel que $L^{\otimes N}$ admette une structure réelle et tel qu'il existe une suite $(s_k)_{k\in \NN}$ de sections réelles de $L^{\otimes Nk}$ et une suite de sections réelles $(s'_k)_{k\in \NN}$ de $M\otimes L^{\otimes Nk}$ qui sont asymptotiquement $J$-holomorphes, c'est-à-dire que la partie $\CC$-antilinéaire de leurs dérivées covariantes sont bornées indépendamment de $k$, et transverses à $0$, c'est-à-dire que leur dérivée covariante est minorée par $\eta \sqrt{k}$ là où les sections s'annulent, où $\eta >0$ est indépendante de $k$. La suite $(s_k\oplus s'_k)_{k\in\NN}$ de sections réelles de $(M\oplus \TCC)\otimes L^{\otimes Nk}$ est alors asymptotiquement $J$-holomorphe. \`A nouveau d'après Gayet \cite{gayet}, Proposition 2 (voir aussi \cite{auroux}), on peut transversaliser cette dernière suite pour obtenir une suite $(\sigma_k)_{k\in\NN}$ de sections réelles de $(M\oplus \TCC)\otimes L^{\otimes Nk}$ qui est encore asymptotiquement $J$-holomorphe et transverse à $0$. De plus, on a des estimations sur $\sigma_k$ qui assurent que les projections $\sigma^M_k$ et $\sigma^{\TCC}_k$ de $\sigma_k$ sur les deux facteurs sont asymptotiquement $J$-holomorphes et transverses à $0$. Ainsi, pour $k$ assez grand, $V_k = (\sigma^M_k)^{-1}(\{0\})$ et $W_k = (\sigma^{\TCC}_k)^{-1}(\{0\})$ sont deux sous-variétés symplectiques réelles de codimension $2$ de $(X,\omega,c_X)$ qui s'intersectent transversalement le long d'une sous-variété symplectique réelle dont l'orientation induite par l'intersection est la même que celle induite par $\omega$. Prenons alors $(\psi^k_t)_{t\in [0,1]}$ comme dans le Lemme \ref{ortho}. En trivialisant la famille de fibrés en doites complexes munis de structures réelle $(\psi^k_t)^*(L^{\otimes Nk})$, $t\in [0,1]$, on obtient une section réelle $\tilde{\sigma}^{\TCC}_k$ de $(\psi^k_1)^*(L^{\otimes Nk})$ qui s'annule tranversalement et dont l'ensemble des zéros est $\psi_1(W_k)$. D'autre part, le diviseur $V_k - \psi_1(W_k)$ est une polarisation. En effet, comme $V_k$ et $\psi_1(W_k)$ s'intersectent orthogonalement pour $\omega$, il existe une structure presque complexe réelle compatible avec $\omega$ pour laquelle $V_k$ et $\psi_1(W_k)$ sont simultanément holomorphes. De plus, on a un isomorphisme $(M,c_M) = ((M,c_M)\otimes (L,c_L)^{\otimes Nk})\otimes \overline{(L,c_L)^{\otimes Nk}}$ qui induit une section polarisante $\sigma^M_k\otimes \overline{\tilde{\sigma}^{\TCC}_k}$ de $(M,c_M)$.
\end{proof}

\begin{Rem}
 Dans \cite{gayet}, Gayet montre que l'on peut choisir $W_k$ de partie réelle vide. Dans notre cas, il serait plus intéressant encore de pouvoir prendre $W_k$ \og imaginaire pure \fg, c'est-à-dire telle que $W_k$ et $c_X(W_k)$ s'intersectent orthogonalement pour $\omega$. Ceci permettrait d'annuler une partie de la contribution de $T_{pol}$ à l'orientabilité de l'espace de module (voir \S \ref{parhyp}). De même, on peut se demander si lorsque $\RR X$ est orientable on peut choisir $V_k$ de partie réelle vide, et si lorsque $(M,c_M)$ admet une racine carée réelle, on peut choisir $V_k$ imaginaire pure.
\end{Rem}

\bibliographystyle{plain-fr}
\bibliography{article}

\begin{thebibliography}{10}
\expandafter\ifx\csname fonteauteurs\endcsname\relax
\def\fonteauteurs{\scshape}\fi

\bibitem{arbarello}
E.~\bgroup\fonteauteurs\bgroup Arbarello\egroup\egroup{},
  M.~\bgroup\fonteauteurs\bgroup Cornalba\egroup\egroup{} et P.~A.
  \bgroup\fonteauteurs\bgroup Griffiths\egroup\egroup{} :
\newblock {\em Geometry of algebraic curves}, volume 268 de {\em Grundlehren
  der mathematischen Wissenschaften}.
\newblock Springer, 2011.

\bibitem{atiyah}
M.~\bgroup\fonteauteurs\bgroup Atiyah\egroup\egroup{} :
\newblock Riemann surfaces and ${S}pin$ structures.
\newblock {\em Annales scientifiques de l'{\'E}cole Normale Sup{\'e}rieure},
  22(2)\string:\penalty500\relax 47--62, 1989.

\bibitem{auroux}
D.~\bgroup\fonteauteurs\bgroup Auroux\egroup\egroup{} :
\newblock Asymptotically holomorphic families of symplectic submanifolds.
\newblock {\em Geom. Funct. Anal.}, 7(6)\string:\penalty500\relax 971--995,
  1997.

\bibitem{bredon}
G.~E. \bgroup\fonteauteurs\bgroup Bredon\egroup\egroup{} :
\newblock {\em Introduction to compact transformation groups}.
\newblock Academic Press, New York, 1972.
\newblock Pure and Applied Mathematics, Vol. 46.

\bibitem{cho}
C.-H. \bgroup\fonteauteurs\bgroup Cho\egroup\egroup{} :
\newblock Counting real {$J$}-holomorphic discs and spheres in dimension four
  and six.
\newblock {\em J. Korean Math. Soc.}, 45(5)\string:\penalty500\relax
  1427--1442, 2008.

\bibitem{article1}
R.~\bgroup\fonteauteurs\bgroup Cr\'etois\egroup\egroup{} :
\newblock Automorphismes r\'eels d'un fibr\'e et op\'erateurs de
  {C}auchy-{R}iemann.
\newblock Preprint arXiv:1110.6899v1.

\bibitem{summary}
R.~\bgroup\fonteauteurs\bgroup Cr\'etois\egroup\egroup{} :
\newblock Real bundle automorphisms, {C}auchy-{R}iemann operators and
  orientability of moduli spaces.
\newblock Preprint arXiv:1309.3110.

\bibitem{degitkha}
A.~\bgroup\fonteauteurs\bgroup Degtyarev\egroup\egroup{},
  I.~\bgroup\fonteauteurs\bgroup Itenberg\egroup\egroup{} et
  V.~\bgroup\fonteauteurs\bgroup Kharlamov\egroup\egroup{} :
\newblock On the number of components of a complete intersection of real
  quadrics.
\newblock Preprint math.AG/0806.4077v2, 2008.

\bibitem{donald}
S.~K. \bgroup\fonteauteurs\bgroup Donaldson\egroup\egroup{} :
\newblock Symplectic submanifolds and almost-complex geometry.
\newblock {\em J. Differential Geom.}, 44(4)\string:\penalty500\relax 666--705,
  1996.

\bibitem{earle}
C.~J. \bgroup\fonteauteurs\bgroup Earle\egroup\egroup{} :
\newblock {\em On moduli of closed {R}iemann surfaces with symmetries}, pages
  119--130.
\newblock Advances in the Theory of Riemann Surfaces, Annals of Mathematics
  Studies 66. Princeton University press and University of Tokyo press, 1971.

\bibitem{earleeels}
C.~J. \bgroup\fonteauteurs\bgroup Earle\egroup\egroup{} et
  J.~\bgroup\fonteauteurs\bgroup Eells\egroup\egroup{} :
\newblock The diffeomorphism group of a compact {R}iemann surface.
\newblock {\em Bulletin of the American Mathematical Society},
  73(4)\string:\penalty500\relax 557--559, 1967.

\bibitem{fooo}
K.~\bgroup\fonteauteurs\bgroup Fukaya\egroup\egroup{}, Y.-G.
  \bgroup\fonteauteurs\bgroup Oh\egroup\egroup{},
  H.~\bgroup\fonteauteurs\bgroup Ohta\egroup\egroup{} et
  K.~\bgroup\fonteauteurs\bgroup Ono\egroup\egroup{} :
\newblock {\em Lagrangian intersection {F}loer theory: anomaly and obstruction.
  {P}art {II}}, volume~46 de {\em AMS/IP Studies in Advanced Mathematics}.
\newblock American Mathematical Society, Providence, RI, 2009.

\bibitem{gayet}
D.~\bgroup\fonteauteurs\bgroup Gayet\egroup\egroup{} :
\newblock Hypersurfaces symplectiques r\'eelles et pinceaux de {L}efschetz
  r\'eels.
\newblock {\em J. Symplectic Geom.}, 6(3)\string:\penalty500\relax 247--266,
  2008.

\bibitem{georgieva}
P.~\bgroup\fonteauteurs\bgroup Georgieva\egroup\egroup{} et
  A.~\bgroup\fonteauteurs\bgroup Zinger\egroup\egroup{} :
\newblock Orientability in real {G}romov-{W}itten theory.
\newblock Preprint arXiv:1308.1347v1.

\bibitem{gompf}
R.~E. \bgroup\fonteauteurs\bgroup Gompf\egroup\egroup{} :
\newblock A new construction of symplectic manifolds.
\newblock {\em Ann. of Math. (2)}, 142(3)\string:\penalty500\relax 527--595,
  1995.

\bibitem{Ionel}
E.-N. \bgroup\fonteauteurs\bgroup Ionel\egroup\egroup{} :
\newblock {GW} invariants relative normal crossings divisors.
\newblock Preprint arXiv:1103.3977v2.

\bibitem{IonelP}
E.-N. \bgroup\fonteauteurs\bgroup Ionel\egroup\egroup{} et T.~H.
  \bgroup\fonteauteurs\bgroup Parker\egroup\egroup{} :
\newblock Relative {G}romov-{W}itten invariants.
\newblock {\em Ann. of Math. (2)}, 157(1)\string:\penalty500\relax 45--96,
  2003.

\bibitem{ivshev}
S.~\bgroup\fonteauteurs\bgroup Ivashkovich\egroup\egroup{} et
  V.~\bgroup\fonteauteurs\bgroup Shevchishin\egroup\egroup{} :
\newblock Structure of the moduli space in a neighborhood of a cusp-curve and
  meromorphic hulls.
\newblock {\em Inventiones mathematicae}, 136(3)\string:\penalty500\relax
  571--602, 1999.

\bibitem{kirby}
R.~\bgroup\fonteauteurs\bgroup Kirby\egroup\egroup{} et
  L.~\bgroup\fonteauteurs\bgroup Taylor\egroup\egroup{} :
\newblock ${P}in$ structures on low-dimensional manifolds.
\newblock \emph{In} {\em Geometry of low-dimensional manifolds (Durham, 1989)},
  num\'ero  151 de London Mathematical Society, Lecture note series, pages
  177--242. Cambridge University Press, 1990.

\bibitem{kravetz}
S.~\bgroup\fonteauteurs\bgroup Kravetz\egroup\egroup{} :
\newblock On the geometry of {T}eichm{\"u}ller spaces and the structure of
  their modular groups.
\newblock {\em Ann. Acad. Sci. Fenn. Ser. A I}, 278\string:\penalty500\relax
  1--35, 1959.

\bibitem{MDS}
D.~\bgroup\fonteauteurs\bgroup McDuff\egroup\egroup{} et
  D.~\bgroup\fonteauteurs\bgroup Salamon\egroup\egroup{} :
\newblock {\em J-holomorphic Curves and Symplectic Topology}, volume~52 de {\em
  Colloquium Publications}.
\newblock American Mathematical Society, 2004.

\bibitem{oktel}
C.~\bgroup\fonteauteurs\bgroup Okonek\egroup\egroup{} et
  A.~\bgroup\fonteauteurs\bgroup Teleman\egroup\egroup{} :
\newblock Abelian {Y}ang-{M}ills theory on real tori and theta divisors of
  {K}lein surfaces.
\newblock Preprint arXiv:1011.1240v2.

\bibitem{puignau}
N.~\bgroup\fonteauteurs\bgroup Puignau\egroup\egroup{} :
\newblock Sur la premi\`ere classe de {S}tiefel-{W}hitney de l'espace des
  applications stables r\'eelles vers l'espace projectif.
\newblock {\em Ann. Inst. Fourier (Grenoble)}, 60(1)\string:\penalty500\relax
  149--168, 2010.

\bibitem{sepp}
M.~\bgroup\fonteauteurs\bgroup Sepp{\"a}l{\"a}\egroup\egroup{} :
\newblock Teichm{\"u}ller spaces of {K}lein surfaces.
\newblock {\em Ann. Acad. Sci. Fenn. Ser. A I Mathematica Dissertationes},
  15\string:\penalty500\relax 1--37, 1978.

\bibitem{shev}
V.~\bgroup\fonteauteurs\bgroup Shevchishin\egroup\egroup{} :
\newblock Pseudoholomorphic curves and the symplectic isotopy problem.
\newblock Preprint math.SG/0010262, 2000.

\bibitem{Solomon}
J.~\bgroup\fonteauteurs\bgroup Solomon\egroup\egroup{} :
\newblock {\em Intersection theory on the moduli space of holomorphic curves
  with {L}agrangian boundary conditions}.
\newblock Th\`ese de doctorat, MIT, 2008.

\bibitem{wel1}
J.-Y. \bgroup\fonteauteurs\bgroup Welschinger\egroup\egroup{} :
\newblock Invariants of real symplectic $4$-manifolds and lower bounds in real
  enumerative geometry.
\newblock {\em Inventiones mathematicae}, 162(1)\string:\penalty500\relax
  195--234, 2005.

\bibitem{wel3}
J.-Y. \bgroup\fonteauteurs\bgroup Welschinger\egroup\egroup{} :
\newblock Spinor states of real rational curves in real algebraic convex
  3-manifolds and enumerative invariants.
\newblock {\em Duke Mathematical Journal}, 127\string:\penalty500\relax
  89--121, 2005.

\end{thebibliography}

\noindent
\textsc{Universit\'e de Gen\'eve \\
Section de Math\'ematiques} \\
2-4 rue du Lièvre, Case postale 64\\
1211 Genève 4\\
Suisse

\noindent
Remi.Cretois@unige.ch\\

\end{document}